\documentclass[a4paper,12pt]{book}
\usepackage{amsmath,amssymb,amsbsy,amsfonts,amsthm,latexsym,
           amsopn,amstext,amsxtra,euscript,amscd,amsthm,graphicx, color} 

\usepackage[pdf]{pstricks}
\usepackage{fancyhdr}
\usepackage{tikz,xtab}
\usepackage[T1]{fontenc} 

\newlength{\plarg}
\newlength{\glarg}
\usepackage{amsfonts}
\usepackage{latexsym}
\usepackage{mathrsfs}
\usepackage[margin=3cm]{geometry}
\usepackage{amsmath, amsthm, amssymb, amsfonts, latexsym, epsfig, palatino}
\usepackage{ragged2e} 
\usepackage{mathrsfs}
\usepackage{verbatim}
\newcommand{\R}{{\mathbb R}}
\newcommand{\C}{{\mathbb C}}
\newcommand{\N}{{\mathbb N}}
\newcommand{\Z}{{\mathbb Z}}

\begin{document}

\newtheorem{lem}{Lemma}
\newtheorem{lemma}[lem]{Lemma}

\newtheorem{prop}{Proposition}
\newtheorem{proposition}[prop]{Proposition}

\newtheorem{thm}{Theorem}
\newtheorem{theorem}[thm]{Theorem}
\newtheorem*{theorem*}{Theorem}
\newtheorem{exemple}[subsection]{Example}

\newtheorem{cor}{Corollary}
\newtheorem{corollary}[cor]{Corollary}

\newtheorem{conj}{Conjecture}
\newtheorem{conjecture}[conj]{Conjecture}

\newtheorem{defi}{Definition}
\newtheorem{definition}[defi]{Definition}

\newtheorem{exam}{Example}
\newtheorem{example}[exam]{Example}
\newtheorem{remark.}[subsection]{Remark}
\newtheorem{prob}{Problem}
\newtheorem{problem}[prob]{Problem}

\newtheorem{ques}{Question}
\newtheorem{question}[ques]{Question}


\def\mand{\qquad\mbox{and}\qquad}

\def\scr{\scriptstyle}
\def\\{\cr}
\def\({\left(}
\def\){\right)}
\def\[{\left[}
\def\]{\right]}
\def\<{\langle}
\def\>{\rangle}
\def\fl#1{\left\lfloor#1\right\rfloor}
\def\rf#1{\left\lceil#1\right\rceil}

\def\cA{{\mathcal A}}
\def\cB{{\mathcal B}}
\def\cC{{\mathcal C}}
\def\cE{{\mathcal E}}
\def\cF{{\mathcal F}}
\def\cI{{\mathcal I}}
\def\cL{{\mathcal L}}
\def\cM{{\mathcal M}}
\def\cN{{\mathcal N}}
\def\cR{{\mathcal R}}
\def\cS{{\mathcal S}}
\def\cP{{\mathcal P}}
\def\cQ{{\mathcal Q}}
\def\cT{{\mathcal T}}
\def\cX{{\mathcal X}}
\def\N{{\mathbb N}}
\def\Z{{\mathbb Z}}
\def\K{{\mathbb K}}
\def\Q{{\mathbb Q}}
\def\R{{\mathbb R}}
\def\C{{\mathbb C}}
\def\eps{\varepsilon}

\newcommand{\comm}[1]{\marginpar{\fbox{#1}}}
\def\xxx{\vskip5pt\hrule\vskip5pt}
\def\yyy{\vskip5pt\hrule\vskip2pt\hrule\vskip5pt}



\thispagestyle{empty}

\begin{titlepage}
\begin{minipage}{\plarg}

{\includegraphics[scale=0.13]{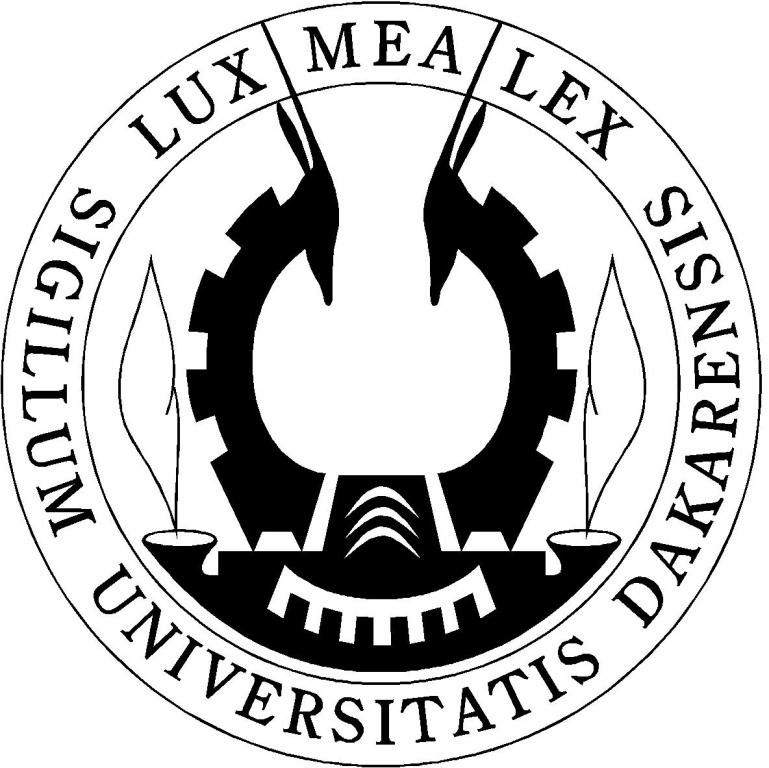} \space\space\space\space\space\space\space\space\space\space
\space\space\space\space\space\space\space\space\space\space\space\space\space\space \space\space\space\space\space\space\space\space\space\space\space\space\space\space\space\space
\space\space\space\space\space\space\space\space\space\space\space\space\space\space\space\space\space\space\space\space\space\space
\space\space\space\space\space\space\space\space\space\space\space\space\space\space\space\space\space\space\space\space
 \includegraphics[scale=0.22]{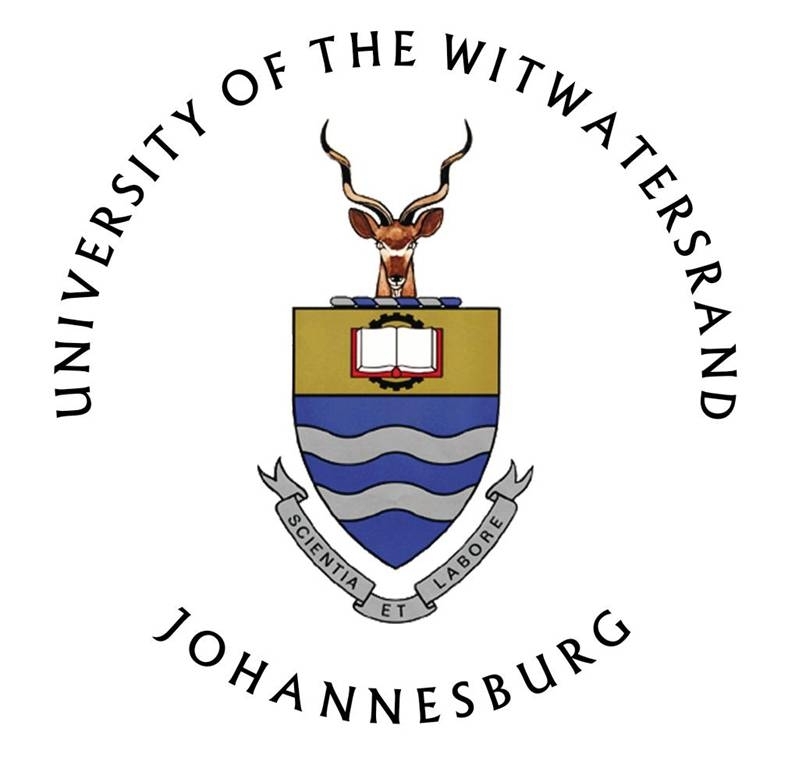}}
\vspace{2mm}
{\rule{\plarg}{1pt}} \vspace{3mm}
\end{minipage}

\begin{center}
\begin{minipage}{\plarg}
\centering
{\huge\bfseries Diophantine Equations with Arithmetic Functions and Binary Recurrences Sequences }\\ \vspace{1cm}
by \\ \vspace{2mm}
{\Large\bfseries Bernadette Faye}\\ \vspace{1cm}
A Thesis submitted to the Faculty of Science, University of the Witwatersrand and to the University Cheikh Anta Diop of Dakar(UCAD) in fulfillment of the requirements for a Dual-degree for Doctor in Philosophy in Mathematics\\ \vspace{4mm}

{\rule{\plarg}{1pt}} \vspace{1mm}

\vspace{1cm}

\centering 
Supervisors
\\ \vspace{2mm}
\begin{tabular}{p{5cm}p{5cm}p{4cm}}
\vfill Pr. Florian Luca,   & \vfill  Univ. of Witwatersrand, Johannesburg, South Africa.\\
\vfill Pr. Djiby  Sow,   & \vfill  Univ. Cheikh Anta Diop, Dakar, Senegal.\\
\end{tabular}
\end{minipage}
\end{center}

\vspace*{1.5cm}
\center{November 6th, 2017}
\end{titlepage}
\frontmatter

\justify
\pagenumbering{arabic}
%
\chapter*{Declaration}

By submitting this dissertation electronically, I declare that the entirety of the work contained therein is my own, original work, that I am the sole author thereof, that reproduction and publication thereof by Wits University or University Cheikh Anta Diop of Dakar will not infringe any third party rights and that I have not previously in its entirety or in part submitted it for obtaining any qualification.

\vspace*{1.5cm}

\begin{tabular}{lllll}
  Signature: & \includegraphics[scale=0.18]{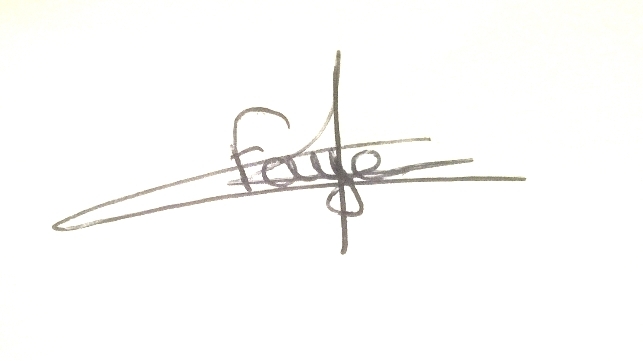}& & Date: & November 06th, 2017.
   
\end{tabular}
\medskip
\medskip

\vspace*{5.5cm}

{\flushleft{Copyright @ 2017 Wits University \\  All right reserved}}
\chapter*{Acknowledgments}


Cette Th\'ese de Doctorat, bien qu'\'etant un accomplissement acad\'emique, est aussi un chapitre important de ma vie. Durant ces quatres derni\'eres ann\'ees, j'ai exp\'eriment\'e les moments de ma vie les plus important avec des hauts et des bas. J'aurai jamais p\^u imaginer le nombre de personnes avec qui j'aurai chemin\'e durant ces moments; j'ai crois\'e le chemin de beaucoup de personnes que j'aimerai remercier par leurs noms.

Tout D'abord, je remercie chaleureusement  mes encadreurs de Th\'ese, Professeur Florian Luca et Professeur Djiby Sow, pour leur encouragement constant dans mes travaux de recherche et leur disponibilit\'e \`a m'apporter leur soutien dans ma vie professionelle et personnelle. Ils m'ont beaucoup soutenu et encourg\'e \`a faire des voyages de recherches, diff\'erents visites et conf\'erences pour pr\'esenter mes travaux, ce qui m'a permis de rencontrer un  nombre d'experts dans mon domaine et me faire un r\'eseau de collaborateurs. J'ai beaucoup appris de vous durant ces ann\'ees de th\'ese. Que Dieu vous b\'enisse!

Je suis \'egalement tr\'es reconnaissante au d\'epartement de math\'ematique et d'informatique de l'universit\'e Cheikh Anta Diop de Dakar esp\'ecialement \`a tous les membres du laboratoire d'Alg\'ebre, de Cryptologie, de G\'eometrie Alg\'ebrique et Applications(LACGAA). Un remerciement sp\'ecial aux Professeurs Mamadou Sanghar\'e, Hamidou Dathe, Omar Diankha, Thi\'ecoumna Gu\'eye, Dr. Barry,  Dr. Amadou Lamine Fall, Dr. Abdoul Aziz Ciss. Je remercie de mani\'ere special le Dr. Amadou Tall qui m'a mis en rapport avec le Pr. Luca depuis AIMS-Senegal pour l'encadrement de mon m\'emoire de Master, qui a aboutit \`a l'encadrement durant ma th\'ese de doctorat. A mes camarades de promotion en \'etude doctorale: Nafissatou Diarra, Mbouye Khady Diagne, Ignace Aristide Milend.

I am also extremely thankful for the support of the school of Mathematics at the University of Witwatersrand in Johannesburg. A special Thanks to Professor Bruce Watson who facilitated the dual degree agreement between the two universities for my doctoral studies at Wits, when he was the head of the school. A special thanks to the people from the school of maths, special thank to Pr. Augustine Munagui, Dr. Bertin Zinsou. Special thanks to Jyoti, Safia and Phyllis for their lovely and constant assistance. Thanks to my peers in the graduate program: Jean Juste, Francis, Marie Chantal.

A special thanks to the African Institute for Mathematical Sciences(AIMS), for their support through the AIMS-Alumni Small Research Grant(AASRG) whithin the framework of the AIMS Research for Africa Project. It allowed me to attend great conferences and work at the AIMS-Senegal Centre during the years of my Ph.D. A special thanks to Pr. Dr.  Mouhamed Moustapha Fall, the AIMS-Senegal Humboldt Chair Holder who always encourages me to cultivate my gifts and talents to achieve my goals and dreams. 

I could not have completed this degree without the support of the  Organization for Women in Science for Developing countries(OWSD) and Sida (Swedish International Development Cooperation Agency) for a scholarship during my Ph.D. studies at the university of Witwatersrand. I attended many conferences presenting my research founding and I met many great people by the help of this Award. This have contribute to my career development and has impact my professional and personal life in a great way. I have been able to built a network with great scientists working in my field of interest. Thanks to Tonya Blowers and Tanja Bole who always responded to my emails and requested in a prompt and professional way.

I am grateful to Pr. Pieter Moree at the Max Planck Institute for Mathematics for his unconditional support during my Ph.D. He invited me several times at the MPIM to conduct research with him and introduced me to the interesting subject of the {\it Discriminator}! Thanks Pr. Dr.  Peter Stollmann at the university of TU Chemnitz in Germany, for your invitation to research visits at your department, funded by means of the German Academic Exchange Service(DAAD) in collaboration with AIMS-Senegal.

For my closest friends, I am so grateful and very happy that they are ever-present
at such key moments in my life, this being no exception. To Lucienne Tine and Daba Diouf, my dears, sweet darling friends, thank you for reminding me "be who you are and be that perfectly well."

Because of my family, I am the women who stands here today. To my brothers Martin, Daniel and Thomas, to my sister Martine Monique, thank you for believing in me and for putting a smile on my face when I most need it. 

To my mother Th\'er\`ese and my Dad Francois, thank you for always being there, every step of the way! thank you for inspiring me to aim high and to also laugh at myself along the way. To all my uncles and aunts, my cousins...
\medskip

\begin{flushleft}
Mii Gueureume Mbine Ndou!\\
Dieureudieuf!
\end{flushleft}

\chapter*{Dedication}

\begin{center}
\textit{To my Family and friends.}
\end{center}

\chapter*{Abstract}

This thesis is about the study of Diophantine equations involving binary recurrent sequences with arithmetic functions. Various Diophantine problems are investigated and new results are found out of this study. Firstly, we study several questions concerning the intersection between two classes of non-degenerate binary recurrence sequences and provide, whenever possible, effective bounds on the largest member of this intersection. Our main study concerns Diophantine equations of the form  $\varphi(|au_n |)=|bv_m|,$ where $\varphi$ is the Euler totient function, $\{u_n\}_{n\geq 0}$ and $\{v_m\}_{m\geq 0}$ are two non-degenerate binary recurrence sequences and $a,b$ some positive integers. More precisely, we study problems involving members of the recurrent sequences being rep-digits, Lehmer numbers, whose Euler's function remain in the same sequence. We particularly study the case when $\{u_n\}_{n\geq 0}$ is the Fibonacci sequence $\{F_n\}_{n\geq 0}$, the Lucas sequences $\{L_n\}_{n\geq 0}$ or the Pell sequence $\{P_n\}_{n\geq 0}$ and its companion $\{Q_n\}_{n\geq 0}$. Secondly, we look of  Lehmer's conjecture on some recurrence sequences. Recall that a composite number $N$ is said to be Lehmer if $\varphi(N)\mid N-1$. We prove that there is no Lehmer number neither in the Lucas sequence $\{L_n\}_{n\geq 0}$ nor in the Pell sequence $\{P_n\}_{n\geq 0}$. The main tools used in this thesis are lower bounds for linear forms in logarithms of algebraic numbers, the so-called Baker-Davenport reduction method, continued fractions, elementary estimates from the theory of prime numbers and sieve methods.

\medskip

   \chapter*{Notation}

\begin{tabular}{lrl}
   $p,q,r$ & & \quad \hbox{prime numbers}  \\
   $\gcd(a,b)$ &  & \quad \hbox{greatest common divisor of $a$ and $b$}  \\
   $\log$ &  & \quad \hbox{natural logarithm}  \\
   $\nu_p(n)$ &  & \quad \hbox{exponent of $p$ in the factorization of $n$}  \\
   $P(\ell)$ &  & \quad \hbox{the largest prime factor of $\ell$ with the convention that $P(\pm 1)=1$}\\
   $\(\frac{a}{p}\)$ &  & \quad \hbox{Legendre symbol of $a$ modulo $p$}  \\
   $F_n$ &  & \quad \hbox{$n$-th Fibonacci number}  \\
   $L_n$ &  & \quad \hbox{$n$-th Lucas number}  \\
   $P_n$ &  & \quad \hbox{$n$-th Pell number}  \\
   $u_n(r,s)$ &  & \quad \hbox{fundamental Lucas sequence}  \\
   $v_n(r,s)$ &  & \quad \hbox{companion Lucas sequence}  \\
   $\sigma(n)$ &  & \quad \hbox{sum of divisors of $n$}  \\
   $\Omega(n)$ &  & \quad \hbox{number of prime power factors of $n$}  \\
   $\omega(n)$ &  & \quad \hbox{number of distinct prime factors of $n$}  \\
   $\tau(n)$  &  & \quad \hbox{number of divisors of $n$, including $1$ and $n$}  \\
   $\varphi(n)$ &  & \quad \hbox{Euler totient function of $n$}  \\
   $\phi$   &  & \quad \hbox{Golden Ratio}  \\
   $|A|$ &  & \quad cardinal of the set A  \\
   $\square$ & & \quad a square number \\
   $\eta$ &  & \quad \hbox{algebraic number}  \\
   $h(\eta)$ &  & \quad \hbox{logarithm height of $\eta$}  \\
   $\mathbb{Q}$ &  & \quad \hbox{field of rational numbers}  \\
   $\mathbb{K}$ &  & \quad \hbox{number field over $\mathbb{Q}$}  \\

\end{tabular}

\tableofcontents
\mainmatter


\chapter{Introduction}

Diophantine equations are one of the oldest subjects in number theory. They have been first studied by the \textit{Greek} mathematician \textit{Diophantus of Alexandria} during the third century. By definition, a Diophantine equation is a polynomial equation of the form 

\begin{equation}
\label{eqi:intro}
P(x_1 ,\ldots, x_n)=0.
\end{equation}
What is of interest is to find all its integer solutions, that is all the $n-$uplets  $(x_1 ,\ldots, x_n)$ in $\mathbb{Z}^n$ which satisfy equation \eqref{eqi:intro}.

Historically, one of the first Diophantine equation is the equation $x^2+y^2=z^2.$ This arises from the problem of finding all the rectangular triangles whose sides have integer lengths. Such triples $(x,y,z)$ are called Pythagorean triples. Some Pythagorean triples are $(3,4,5),(5,12,13),(8,15,17)$ but these are not all. All Pythagorean triples can be obtained as follows: if $(x,y,z)$ is a solution, then $(x/z,y/z)$ is a rational solution. We have then $(x/z)^2+(y/z)^2=1$, namely $(x/z, y/z)$ is the unit circle and has rational coordinates. Using the parametrization of the circle  $\cos\theta = \frac{1-t^2}{1+t^2}$ and $\sin\theta = \frac{2t}{1+t^2}$ where $t=\tan(\theta/2)$, the rational values of $t$ give all the solutions of the equation. Another example is the linear Diophantine equation $ax + by = c$ where $a,b,c$ are fixed integers and $x,y$ are integer unknowns.

Given a Diophantine equation, the fundamental problem is to study is whether solutions exist. If they exist one would like to know how many there are and how to find all of them. Certain Diophantine equations have no solutions in non zero integers like the \textit{Fermat} equation, $ x^n + y^n = z^n$ with $n\geq 3.$ 

The study of Diophantine equations helped to develop many tools in modern number theory. For example, for the proof of \textit{Fermat's Last Theorem}, many tools from algebraic geometry, elliptic curves, algebraic number theory, etc. were developed.

Among the $23$ problems posed by Hilbert in $1900$, the $10^{th}$ Problem concerned Diophantine equations. Hilbert asked if there is a universal method for solving all Diophantine equations. Here we reformulate it: 

\begin{center}
\textit{"Given a Diophantine equation with any number of unknown
quantities and with rational integral numerical coefficients: To
devise a process according to which it can be determined by a
finite number of operations whether the equation is solvable in
integers".}
\end{center}

In $1970$, Y. Matiyasevich gave a negative solution to Hilbert's $10$th Problem. His result is the following.

\begin{theorem}[Y. Matiyasevich] There is no algorithm which, for a given arbitrary Diophantine equation, would tell whether the equation has an integer solution or not.
\end{theorem}

\begin{remark.}
For rational solutions, the analog of Hilbert's 10th problem is not yet solved. That is, the question whether there exists an algorithm to decide if a Diophantine equation has a \textit{rational} solution or not is still open.
\end{remark.}

Since there is no general method to solve Diophantine equations, some techniques were found to solve particular families of Diophantine equations. Many great mathematicians like \textit{Pierre Fermat}, \textit{Leonhard Euler}, \textit{Joseph Louis Lagrange} and \textit{Henri Poincar\'e} have interesting work on the subject. Many tools have also been developed such as transcendental number theory and computational number theory.
\medskip
\medskip

In this thesis, we study certain Diophantine equations involving arithmetic functions and binary recurrence sequences.

\section{Motivation and Overview}
Many problems in number theory may be reduced to finding the intersection of two sequences of positive integers. The heuristic is that the finiteness of this intersection should depend on how quickly the two sequences grow. During his life time, P. Erd\H{o}s and his collaborators devoted a lot of work to the study of the intersection of two arithmetic functions. In that line of research, we might add the recurrence sequences. 

In the last years, many papers have been published concerning Diophantine equations of the form 

\begin{eqnarray}
u_n &=& v_m, \quad\hbox{or}\label{eqi:1}\\
& & \nonumber\\
\varphi(|au_n |) &= & |bv_m|, \label{eqi:2}
\end{eqnarray} 
where $\{u_n\}_{n\geq 0}$ and $\{v_m\}_{m\geq 0}$ are two non-degenerate binary recurrence sequences, $m\geq 0, n\geq 0$ and $a,b$ are fixed positive integers. We refer to the papers \cite{BD}, \cite{KP}, \cite{VM}, \cite{FLl1}.

Considering the Diophantine equation of the form \eqref{eqi:2}, one can see that on the one-hand the Euler function $\varphi$ is a multiplicative function so it behaves well with respect to the multiplicative properties of the integers while on the other-hand the recurrence sequence has some additive properties. So, the study of the intersection between the multiplicative and the additive structure makes such equations interesting. 

In $1978$, M. Mignotte (see \cite{MM1} and \cite{MM2}) proved that the equation (\ref{eqi:1}) has only finitely many solutions that are effectively computable under certain conditions. For example, if $\{u_n\}_{n\geq 0}$  and $\{v_m\}_{m\geq 0}$ are two non degenerate binary recurrence sequences whose characteristic equation has real roots, then it suffices that the logarithm of the absolute values of the largest roots to be linearly independent over $\mathbb{Q}.$
In $1981,$ M\'aty\'as \cite{FM} gave a criterion for determining whether two second order linear recurrence sequences have nonempty intersection.

Later in $2002$, F. Luca studied Diophantine equations of form \eqref{eqi:2} (see \cite{FL2}), involving the Euler totient function of binary recurrent sequences  and proved that if $\{u_n \}_{n\geq 0}$ and $\{v_m \}_{m\geq 0}$ are two non-degenerate binary recurrent sequences of integers such that $\{v_m\}_{m\geq 0}$ satisfies some technical assumptions, then the Diophantine equation \eqref{eqi:2} has only finitely many effectively computable positive integer solutions $(m, n).$ Though, for two given binary recurrent sequences, it is in general difficult to find all such solutions. Furthermore, since these results are ineffective,  the determination of all the solutions is a challenge.

Our goal in this thesis is to continue this line of research by solving effectively certain equations of the form \eqref{eqi:1} and (\ref{eqi:2}).
\medskip

The material presented in this thesis covers all the results from the following  journal papers:

\begin{enumerate}

\item[\cite{FLT}]  B. Faye, F. Luca, A. Tall {\it On the equation $\phi(5^m -1) = 5^n-1$} Korean Journal of Math. Soc. {\bf 52} (2015) No. 2, 513-514. 

\item[\cite{BL}]  B. Faye, F. Luca, {\it On the equation $\phi(X^m -1) = X^n-1$} International Journal of Number Theory, {\bf 11}, No. 5, (2015) 1691-1700. 

\item[\cite{JBLT}]  B. Faye, Jhon J. Bravo, F. Luca, A. Tall  {\it Repdigits as Euler functions of Lucas numbers} An. St. Math. Univ. Ovidius Constanta {\bf 24}(2) (2016) 105-126. 

\item[\cite{BL2}]  B. Faye, F. Luca, {\it Pell and Pell Lucas numbers with only one distinct digit} Ann. Math. Informaticae, {\bf 45} (2015) 55-60. 

\item[\cite{FL1}]  B. Faye, F. Luca, {\it Pell Numbers whose Euler Function is a Pell Number} Publications de l'Institut Math\'ematique nouvelle s\'erie (Beograd), {\bf 101} (2015) 231-245. 

\item[\cite{FL}]  B. Faye, F. Luca, {\it Lucas Numbers with Lehmer Property} {\it Mathematical Reports,\/} {\bf 19(69)}, 1(2017), 121-125. 

\item[\cite{FL02}]  B. Faye, F. Luca, {\it Pell Numbers with the Lehmer property}, Afrika Matematika, {\bf 28(1-2)} (2017), 291-294. 

\end{enumerate}

\section{Some background and Diophantine Problems}
In this section, we give an overview of the different Diophantine problems  which have been studied in this thesis. All these problems have been treated in papers which are either published or have been submitted for publication.

\subsection*{\small Pell numbers whose Euler function is a Pell number}

In \cite{G}, it is shown that $1,~2$, and $3$ are the only Fibonacci numbers whose Euler function is also a Fibonacci number, while in \cite{Lu2}, Luca found all the Fibonacci numbers whose Euler function is a power of $2$. In \cite{LuSt}, Luca and St\u anic\u a found all the Pell numbers whose Euler function is a power of $2$. In \cite{Lu1}, we proved a more general result which contains the results of \cite{Lu2} and \cite{LuSt} as particular cases. Namely, consider the Lucas sequence $\{u_n\}_{n\ge 0}$, with $u_0=0,~u_1=1$ and
$$
u_{n+2}=r u_{n+1}+s u_n\qquad {\text{\rm for~all}}\qquad n\ge 0,
$$
where $s\in \{\pm 1\}$ and $r\ne 0$ is an integer. We proved that there are finitely many terms of this sequence which their Euler's function are powers of $2$.
\medskip

In the same direction, we have investigated in this thesis, the solutions of the Diophantine equation $$\varphi(P_n)=P_m,$$ where $\{P_n\}_{n\geq 0}$ is the Pell sequence given by $P_0=0, P_1=1$ and $P_{n+2}=2P_{n+1}+P_n$ for all $n\geq0.$ In others words, we are interested in knowing which are the terms of the Pell sequence whose Euler's function are also the terms of the Pell sequence.

In Chapter \textcolor{red}{3}, we effectively solve the above equation. We prove the following

\begin{theorem*}[Chapter \textcolor{red}{3}, Theorem \textcolor{red}{16}]
The only solutions in positive integers $(n,m)$ of the equation
\begin{equation*}
\varphi(P_n)=P_m
\end{equation*}
are $(n,m)=(1,1),(2,1).$
\end{theorem*}

\section*{\small Repdigits and Lucas sequences}

Let $\{F_n\}_{n\ge 0}$ and $\{L_n\}_{n\ge 0}$ be the sequence of Fibonacci and Lucas numbers given by $F_0=0,~F_1=1$ and $L_0=2,~L_1=1$ and recurrences
$$
F_{n+2}=F_{n+1}+F_n\quad {\text{\rm and}}\quad L_{n+2}=L_{n+1}+L_n\quad {\text{\rm for~all}}\quad n\ge 0.
$$
In \cite{Lu4}, it was shown that the largest solution of the Diophantine equation 
\begin{equation}
\label{eqi:4}
\varphi(F_n)=d\left(\frac{10^m-1}{9}\right),\qquad d\in \{1,\ldots,9\}
\end{equation} is obtained when $n=11.$ Numbers of the form $d(10^m-1)/9$ are called {\it repdigits} in base $10$, since their base $10$ representation is the string ${\underbrace{dd\cdots d}_{m~{\text{\rm times}}}}$. Here, we look at Diophantine equation \eqref{eqi:4} with $F_n$ replaced by $L_n$:
\begin{equation}
\label{eqi:5}
\varphi(L_n)=d\left(\frac{10^m-1}{9}\right),\qquad d\in \{1,\ldots,9\}.
\end{equation}

Altrough we did not completely solve the problem, we obtained some interesting properties of the solutions of the above equation. These results are the main result in \cite{JBLT}, which is presented in Chapter \textcolor{red}{5}. We prove the following.
\begin{theorem*}[Chapter \textcolor{red}{5}, Theorem \textcolor{red}{19}]
Assume that $n>6$ is such that equation \eqref{eqi:5} holds with some $d$. Then
\begin{itemize}
\item $d=8$;
\item $m$ is even;
\item $n=p$ or $p^2$, where $p^3\mid 10^{p-1}-1$.
\item $10^9<p<10^{111}$.
\end{itemize}
\end{theorem*}
\medskip

Furthermore, we investigated the terms of the Pell sequence $\{P_n\}_{n\geq 0}$ and its companions $\{Q_n\}_{n\geq 0}$ given by $Q_0=2, Q_1=2$ and $Q_{n+2}=2Q_{n+1}+Q_n$ for all $n\geq0$, which are repdigits. This leads us to solve the equations

\begin{equation}
\label{eqi:P}
P_n=a\left(\frac{10^m-1}{9}\right) \quad {\text{\rm for some}}\quad a\in \{1,2,\ldots, 9\}
\end{equation}
and 

\begin{equation}
\label{eqi:Q}
Q_n=a\left(\frac{10^m-1}{9}\right) \quad {\text{\rm for some}}\quad a\in \{1,2,\ldots,9\}.
\end{equation}

One can see that this problem leads to a Diophantine equation of the form \eqref{eqi:1}. In fact, a straightforward use of the theory of linear form in logarithms gives some very large bounds on $\max\{m,n\}$, which then can be reduced either by using the LLL \cite{lll} algorithm or by using a procedure originally discovered by Baker and Davenport \cite{BD} and improved by Dujella and Peth\H o \cite{DP}.

In our case, we do not use linear forms in logarithms. We prove in an elementary way that the solutions of the equations \eqref{eqi:P} and \eqref{eqi:Q} are respectively $n=0,1,2,3$ and $n=0,1,2.$ Theses results are the main results of \cite{BL2} and are presented in Chapter \textcolor{red}{5}.

\subsection*{\small American Mathematical Monthly problem}

Problem $10626$ from the {\it American Mathematical Monthly} \cite{A} asks to find all positive integer solutions $(m,n)$ of the Diophantine equation
\begin{equation}
\label{eq:3i}
\varphi(5^m-1)=5^n-1.
\end{equation}
To our knowledge, no solution was ever received to this problem. In this thesis, we prove the following result.

\begin{theorem*}[Chapter \textcolor{red}{4}, Theorem \textcolor{red}{17}]
\label{thm:2}
Equation \eqref{eq:3i} has no positive integer solution $(m,n)$.
\end{theorem*}

In \cite{B}, it was shown that if $b\ge 2$ is a fixed integer, 
then the equation
\begin{equation}
\label{eq:b}
\varphi\left(x\frac{b^m-1}{b-1}\right)=y\frac{b^n-1}{b-1}\qquad x,y\in \{1,\ldots,b-1\}
\end{equation}
has only finitely many positive integer solutions $(x,y,m,n)$. That is, there are only finitely many repdigits in base $b$ whose Euler function is also a repdigit in base $b$. Taking $b=5$, with $x=y=4$, it follows that equation \eqref{eq:3i} has only finitely many positive integer solutions $(m,n)$. 

In \cite{BL}, our main result improve the result of \cite{B} for certain values of $x,y$. 

\begin{theorem*}[Chapter \textcolor{red}{4}, Theorem \textcolor{red}{18}]
Each one of the two equations 
\begin{equation}
\label{eq:b1}
\varphi(X^m-1)=X^n-1\qquad and\qquad \varphi\left(\frac{X^m-1}{X-1}\right)=\frac{X^n-1}{X-1}
\end{equation}
has only finitely many positive integer solutions $(X,m,n)$ with the exception $m=n=1$ case in which any positive integer $X$ leads to a solution of the second equation above. Aside from the above mentioned exceptions, all solutions have $X<e^{e^{8000}}$. 
\end{theorem*}


\subsubsection*{On Lehmer's Conjecture}

A composite positive integer $n$ is \textit{Lehmer}  if $\varphi(n)$ divides $n-1$. Lehmer \cite{Leh} conjectured that there is no such integer. To this day, the conjecture remains open. Counterexamples to Lehmer's conjecture have been dubbed {\it Lehmer numbers}.

Several people worked on getting larger and larger lower bounds on a potential Lehmer number. Lehmer himself proved that if $N$ is Lehmer, then $\omega(N)\geq 7$. This has been improved by Cohen and Hagis \cite{coh} to $\omega(N)\geq 14.$ The current record $\omega(N)\ge 15$ is due to Renze \cite{joh}. If in addition $3\mid N$, then $\omega(N)\geq 40\cdot 10^{6}$ and $N>10^{36\cdot 10^{7}}.$

Not succeeding in proving that there are no Lehmer numbers, some researchers have settled for the more modest goal of proving that there are no Lehmer numbers in certain interesting subsequences of positive integers. In $2007$, F. Luca \cite{L} proved that there is no Lehmer number in the Fibonacci sequence.  In \cite{GL}, it is shown that there is no Lehmer number in the sequence of Cullen numbers $\{C_n\}_{n\ge 1}$ of general term $C_n=n2^n+1$, while in \cite{Dajune} the same conclusion is shown to hold for generalized Cullen numbers. In \cite{F}, it is shown that there is no Lehmer number of the form $(g^n-1)/(g-1)$ for any $n\ge 1$ and integer $g \in [2,1000]$. In Chapter \textcolor{red}{6}, we adapt the method from \cite{L} to prove the following theorems:

\begin{theorem*}[Chapter \textcolor{red}{6}, Theorem \textcolor{red}{22}]
\label{thm:9}
There is no Lehmer number in the Lucas sequence $\{L_n\}_{n\geq 0}.$
\end{theorem*}


\begin{theorem*}[Chapter \textcolor{red}{6}, Theorem \textcolor{red}{23}]
\label{thm:9}
There is no Lehmer number in the Pell sequence $\{P_n\}_{n\geq 0}.$
\end{theorem*}

\section{Organization of the Thesis}

Our thesis consists of six chapters. This chapter, as the title suggests, gives a general introduction and the main motivation of this thesis, together with a description of the different Diophantine problems which have been studied. From Chapter \textcolor{red}{3} to Chapter \textcolor{red}{6} , each chapter contains new results concerning Diophantine equations with arithmetic function and some given binary recurrence sequences. 

In Chapter \textcolor{red}{2}, we  give the preliminary tools and main results that will be used in this work. We start by reminding some definitions and properties of binary recurrence sequences and arithmetic functions with the main emphasis of the Euler's function. In Section \textcolor{red}{\ref{sec4}}, we recall results on the Primitive Divisor Theorem of members of Lucas sequences. Chapter \textcolor{red}{2} concludes with a result due to Matveev \cite{matveev} which gives a general lower bound for linear forms in logarithms of algebraic numbers(Section \textcolor{red}{\ref{sec21}}) and results from Diophantine properties of continued fractions(Section \textcolor{red}{\ref{sec22}}).

In Chapter \textcolor{red}{3}, we investigate Pell numbers whose Euler function is also a Pell number. 

In Chapter \textcolor{red}{4}, we solve the Diophantine equations \eqref{eq:3i} and \eqref{eq:b1}. 

In Chapter \textcolor{red}{5}, we use some tools and results from Diophantine approximations and linear form in logarithms of algebraic numbers, continued fractions and sieve methods to solve Theorem \textcolor{red}{17}, Theorem \textcolor{red}{18} and Diophantine equations \eqref{eqi:P} and \eqref{eqi:Q}.

In Chapter \textcolor{red}{6}, we prove that the Lehmer Conjecture holds for the Lucas sequence $\{L_n\}_{n\ge 0}$ and the Pell sequence $\{P_n\}_{n\ge 0}$. Namely, none of these sequences contains a Lehmer number.



\chapter{Preliminary results}

In this chapter, we give some preliminaries that will be useful in this thesis. In Section \textcolor{red}{\ref{sec1}}, we discuss some basic definitions and properties of arithmetic functions, in Section \textcolor{red}{\ref{sec2}}, the Euler function and related results. In Section \textcolor{red}{\ref{sec3}}, we study the arithmetic of binary recurrence sequences.  In Section \textcolor{red}{\ref{sec4}}, we recall the Primitive Divisor Theorem for members of Lucas sequences. We conclude this chapter with results from linear forms in logarithms of algebraic numbers in Section \textcolor{red}{\ref{sec21}} and properties of continued fractions in Section \textcolor{red}{\ref{sec22}}. 

\section{Arithmetic Functions}
\label{sec1}
An arithmetic function is a function defined on the set of natural numbers $\mathbb{N}$  with real or complex values. It can be also defined as a sequence $\{a(n)\}_{n\geq 1}.$ An active part of number theory consists, in some way, of the study of these functions.
\medskip

%

The usual number theoretic examples of arithmetic functions are $\tau(n)$, $\omega(n)$ and $\Omega(n)$ for the number of divisors of $n$ including $1$, the number of distinct prime factors of $n$ and the number of prime power factors of $n$, respectively. These functions can be described in terms of the prime factorization of $n$ as in the following proposition:

\begin{proposition}
Suppose that $n>1$, has the prime factorization 

$$n=\prod_{j=1}^{m}p_j^{k_j}.$$

Then,

$$\tau(n)=\prod_{j=1}^{m}(k_j+1), \quad \omega(n)=m, \quad \Omega(n)=\sum_{j=1}^{m}k_j.$$
\end{proposition}

\begin{proof}
The expression of $\omega(n)$ and $\Omega(n)$ follow from their definition. Divisors of $n$ have the form $\prod_{j=1}^{m}p_j^{r_j}$ where for each $j$, the possible values of $r_j$ are $0,1,\ldots,k_j.$ This gives the expression of $\tau(n).$
\end{proof}
\medskip

\begin{definition}
An arithmetic function $a$, such that $a(1)=1$,  is said to be 
\begin{itemize}
\item completely multiplicative if $a(mn)=a(m)a(n) \quad \hbox{for all m, n;}$

\item multiplicative if $a(mn)=a(m)a(n)\quad \hbox{when $(m,n)=1.$}$
\end{itemize}
\end{definition}
 
\begin{example}
\begin{enumerate}
\item For any $s$, $a(n)=n^s$ is completely multiplicative.
\item $\tau$ is multiplicative.
\item Neither $\omega$ nor $\Omega$ is multiplicative.
\end{enumerate}
\end{example}

\section{The Euler Totient $\varphi$ }
\label{sec2}
\begin{definition}
For a positive integer $n$, the Euler totient function $\varphi(n)$ counts the number of positive integers $m\leq n$ which are coprime to $n$, that is:
$$
\varphi(n)=\Big|\{1\leq m\leq n : (m,n)=1\}\Big|.
$$
\end{definition}
Clearly, if $n$ is a prime number, then $\varphi(n)=n-1.$ Further, let $\mathbb{Z}/n\mathbb{Z}$ be the set of congruence classes $a\pmod{n}$. This set is also a ring and its invertible elements form a group whose cardinality is $\varphi(n)$. Lagrange's theorem from group theory tells us that the order of every element in a finite group is a divisor of the order of the group. In this particular case, this theorem implies that
$$
a^{\varphi(n)}\equiv 1 \pmod n
$$
holds for all integers $a$ coprime to $n$. The above relation is known as the Euler Theorem.

\begin{theorem}
The Euler totient function $\varphi(n)$ is multiplicative, i.e
$$
\varphi(mn)=\varphi(m)\varphi(n) \quad\hbox{for all $m\geq 1$, $n\geq 1$ and $(m,n)=1.$}
$$
\end{theorem}

\begin{proof} If $m=1$ and $n=1$, then the result holds. Suppose now that $m>1$ and $n>1.$ We denote by

$$U=\{u: 1\leq u\leq m, (u,m)=1 \},$$
$$V=\{v: 1\leq v\leq n, (v,n)=1 \},$$
$$W=\{w: 1\leq w\leq m\times n, (w,m\times n)=1 \}.$$

Then, we obtain that $|U|=\varphi(m),$ $|V|=\varphi(n)$  and  $|W|=\varphi(m\times n).$ We now show that the set $W$ has as many elements as the set $U\times V=\{(u,v): u\in U, v\in V\}$, then the theorem follows.
\medskip

%
\end{proof}

\begin{theorem}
If $ n=p_1^{\alpha_1}p_2^{\alpha_2}\ldots p_{k}^{\alpha_k}$ , where $p_1,\ldots,p_k$ are distinct primes and $\alpha_1,\ldots,\alpha_{k}$ are positive integers, then
$$
\varphi(n)=\prod p_i^{\alpha_i-1}(p_i-1)=p_1^{\alpha_1-1}(p_1-1)\ldots p_k^{\alpha_k-1}(p_k-1),
$$
and
$$
\sigma(n)= \left(\frac{p_1^{\alpha_1+1}-1}{p_1-1}\right)\ldots\left(\frac{p_k^{\alpha_k+1}-1}{p_k-1}\right),
$$
where $\sigma$ is the sum of divisors of $n$.
\end{theorem}

\medskip

\begin{exemple}.
\newline
\begin{itemize}
\item $\varphi(68) = \varphi(2^2.17) = (2^{2-1}(2-1))(17-1) = 2.16 = 32.$
\item $\sigma(68)=  2 + 17 = 126.$
\end{itemize}
\end{exemple}

The next lemma from \cite{Lu3} is useful in order to obtain an upper bound on the sum appearing in the right--hand side of \eqref{eq1:Sn}.

\begin{lemma}
\label{lem1:useful}
We have
$$
\sum_{d\mid n} \frac{\log d}{d}<\left(\sum_{p\mid n} \frac{\log p}{p-1}\right) \frac{n}{\varphi(n)}.
$$
\end{lemma}

\medskip

\section{Linearly Recurrence Sequences}
\label{sec3}
Linear recurrence sequences have interesting properties and played a central role in number theory. The arithmetic of these sequences have been studied by Fran\c{c}ois Edouard Anatole Lucas (1842-1891). We shall  reference some theorems from the multitude of results  that had been proved over recent years.
\medskip

\begin{definition}
In general, a linear recurrence sequence $\{u_n\}_{n\geq}$ of order $k$ is

$$u_{n+k} = c_{k-1}u_{n+k-1} +c_{k-2}u_{n+k-2} +\ldots  +c_0u_n,~~~~~~~~~~~~~~~~\text{$n\geq 0$}$$
with $c_0\neq 0$, for $n\geq k$, where $c_0,c_1,\dots,c_{k-1}$ are constants. The values $u_0, \ldots, u_{k-1}$ are not all zero.  
\end{definition}

Linear recurrence sequences of order  $2$ are called \textit{binary} and the ones of order $3$ \textit{ternary}. For a linear recurrence sequence of order $k$, the $k$ initial values determine all others elements of the sequence.
\medskip

\begin{exemple}
For  $k=3$, let $u_0=u_1= 1,$ $u_2 = 2$ and 

$$u_{n+3}=3u_{n+2}+ 2u_{n+1}+u_{n} ~~~~~~~~~~~~~~~~~~~~~~~~~~~\text{for $n\geq 0$}$$
is a linear recurrence sequence.
\end{exemple}

\medskip

\begin{definition}
The characteristic polynomial of the linear recurrence of recurrence
$$
u_{n+k} = c_{k-1}u_{n+k-1} +c_{k-2}u_{n+k-2} +\ldots  +c_0u_n,~~~~~~~~~~~~~~~~\text{$n\geq 0$}
$$
is the polynomial,
$$
f(x)= x^{k} - c_{k-1}x^{k-1} - \ldots - c_{1}x - c_0.
$$
\end{definition}

We assume that this polynomial has distinct roots $\alpha_1,\alpha_2,\ldots,\alpha_s$ i.e.,

$$f(X)=\prod_{i=1}^{s}\(X-\alpha_i\)^{\beta_i}$$ 
of multiplicities $\beta_1,\ldots, \beta_s$ respectively. 

\medskip

It is well known from the theory of linearly recurrence sequences that for all $n$, there exist uniquely determined polynomials $g_i \in \mathbb{Q}(u_0,\ldots, c_0,\ldots,c_k, \alpha_1,\ldots, \alpha_k)[x]$ of degree less than $\beta_i(i=1,\ldots,s)$ such that

\begin{equation}
\label{eq1:f(x)}
f_n(x)=\sum_{i=1}^sg_i(n)\alpha_i^n,\quad \hbox{for $n\geq 0$}
\end{equation}

In this thesis, we consider only integer recurrence sequences, namely recurrence sequences whose coefficients and initial values are integers i.e $s=k$(all the roots of $f(x)$ are distinct). Hence, $g_i(n)$ is an algebraic number for all $i = 1,\ldots, k$ and $n\in \mathbb{Z}.$ Thus, we have the following result.


\begin{theorem}
\label{th1:poly} 
Assume that $f(X)\in\mathbb{Z}[X]$ has distinct roots. Then there exist constants $\gamma_1,\ldots,\gamma_k \in \mathbb{K}=\mathbb{Q}(\alpha_1,\ldots,\alpha_k)$ such that the formula 
$$u_n=\sum_{i=1}^{k}\gamma_i\alpha_i^{n} \quad \hbox{holds for all $n\geq 0.$}$$
\end{theorem}
\begin{proof}
Conf \cite{dio}, page $8.$
\end{proof}

A geometric sequence is a simple example of linear recurrence sequences. It is defined by 
\begin{center}
$$
\left\{
\begin{array}{ll}
u_0 = a,\\
u_{n+1} = cu_n.
\end{array}
\right.
$$
\end{center}

The constant ratio $\(u_{n+1}/{u_n}\)=c$ is called the {\it common ratio} of the sequence. 

%

\subsection{Lucas sequence}

Let $(r,s)\in\mathbb{N}^2$ such that $(r,s)=1$ and $r^2+4s\neq 0.$ 

\begin{definition}
A Lucas sequence $u_n(r,s)$ is a binary recurrence sequence that satisfies the recurrence sequence
$u_0=0$, $u_1=1.$ Its characteristic polynomial is of the form
\begin{equation*}
x^2 -rx -s = 0.
\end{equation*}
\end{definition}
Clearly, $u_n$ is an integer for all $n\geq0$. Let $\alpha_1$ and  $\alpha_2 $ denote the two roots of its characteristic polynomial and its \textit{discriminant} $\Delta = r^2 + 4s \neq0$. Then  the roots are
$$
\alpha_1=\frac{r + \sqrt{\Delta}}{2} \quad {\rm and}\quad \alpha_2=\frac{r - \sqrt{\Delta}}{2}.
$$

\begin{remark.}
If $(r,s)= 1$, $u_0=0$, and $u_1=1$ then $\{u_n\}_{n\geq 0}$ is called a \textit{Lucas sequence} of the \textit{first kind}.
\end{remark.}

\subsubsection*{Binet's Formula}
Binet's formula is a particular case of Theorem \ref{th1:poly} . Its named after the mathematician \textit{Jacques Phillipe Marie Binet} who found first the similar formula for the \textit{Fibonacci sequence}. For a Lucas sequence $\{u_n\}_{n\geq 0}$, Theorem \ref{th1:poly} tells us that the corresponding Binet formula is

\begin{equation}
\label{eq1:binet}
u_{n}= \gamma_1\alpha_1^n + \gamma_2\alpha_2^n \quad ~~for~~ all~~ n\geq 0. 
\end{equation}

The values of  $\gamma_1$ and $\gamma_2$ are found by using the formula \eqref{eq1:binet}, when $n = 0, 1$. Hence, one obtains the system 

$$\gamma_1+\gamma_2=u_0=0\quad\quad \gamma_1\alpha_1+\gamma_2\alpha_2=u_1=1.$$

Solving it, we get that $\gamma_1=\sqrt{\Delta}$ and $\gamma_1=-1/\sqrt{\Delta}$. Since $\sqrt{\Delta}=(\alpha_1-\alpha_2)$, we can write 
 
    \begin{equation}
    \label{binet1}
     u_n=\frac{\alpha_1^n - \alpha_2^n}{\alpha_1 -\alpha_2} \quad~~for~~all ~~ n\geq0,
    \end{equation}
    
The companion Lucas-sequence $\{v_n\}_{n\geq 0}$ of $\{u_n\}_{n\geq 0}$ has the following Binet formula
    \begin{equation}
    \label{binet2}
     v_n= \alpha_1^n + \alpha_2^n \quad ~~for~~all ~~ n\geq0.
    \end{equation}

\begin{exemple}.
\medskip
\begin{itemize}
\item $u(1,-3):\{0, 1, 1, 4, 7, 19, 40, 97, 217, 508, 1159, 2683, 6160, 14209, \ldots,\}$
\item $v(1,-3):\{2, 1, 7, 10, 31, 61, 154, 337, 799, 1810, 4207, 9637, 22258,\ldots,\}$
\end{itemize}
\end{exemple}

\begin{definition}
A Lucas sequence, as defined by formula \eqref{eq1:binet}, is said to be nondegenerate if $\gamma_1\gamma_2\alpha_1\alpha_2\neq 0$ and $\alpha_1/\alpha_2$ is not a root of unity.
\end{definition}

\medskip
There are several relations among the Lucas sequence $\{u_n\}_{n\geq0}$ and its companion sequence $\{v_n\}_{n\geq0}$ which can be proved using their Binet formulas. Here, we recall some of those relations that will be useful throughout this thesis.
\medskip
 
\begin{theorem}
\label{eq1:binet1}
Let $\{u_n\}_{n\geq0}$ be a Lucas sequence. Then, the following holds:
\begin{itemize}
\item $ u_{2n}=u_nv_n$ where $v_n$ is its companion sequence.
\item $(u_m,u_n)=u_{(m,n)}$ for all  positive integers $m,n$. Consequently, the integers $u_n$ and $u_m$ are \textit{relatively prime} when $n$ and $m$ are \textit{relatively prime}.
\item If $n\mid m$, then $u_n \mid u_m$.
\end{itemize}
\end{theorem}

\medskip


\begin{prop}
Assume that $p\nmid s$ is odd and $e=(\frac{\Delta}{p}).$ The following holds:
\begin{itemize}
\item If $p\mid \Delta=r^2+4s$, then $p\mid u_p.$
\item If $p\nmid\Delta$ and $\alpha\in \mathbb{Q}$ then $p\mid u_{p-1}$.
\item If $p\nmid\Delta$ and $\alpha\notin \mathbb{Q}$ then $p\mid u_{p-e}$.
\end{itemize}
\end{prop}
\medskip

\begin{definition}
For a prime $p$, let $z(p)$ be the order of appearance of $p$ in $\{u_n\}_{n\ge 0}$; i.e., the minimal positive integer $k$ such that $p\mid u_k$.
\end{definition}
\medskip

\begin{prop}
\begin{itemize}
\item If $p\mid u_n$, then $z(p)\mid n$.
\item If $p\nmid \Delta$, then $p\equiv \pm 1\pmod{z(p)}. $
\item If $p\mid s$, then $z(p)=p$.
\end{itemize}
\end{prop}
\medskip

\subsection{The Fibonacci sequence}

\begin{definition}
The Fibonacci sequence is a Lucas sequence defined by
\begin{center}
$$
\left\{
\begin{array}{ll}
F_0= 0\\
F_1= 1 \\
F_{n+2} = F_{n+1} + F_n~~~~~ for ~~ n\geq 0.
\end{array}
\right.
$$
\end{center}
\end{definition}
\medskip

The polynomial $f(x) = x^2-x-1$ is the characteristic polynomial for $F_n$ whose roots are $\alpha=(1+\sqrt{5})/2$ and $\beta=(1-\sqrt{5})/2$. For the Fibonacci numbers, the quotient of successive term  is not constant as it is the case for geometric sequences. Indeed, Johannes Kepler found that the quotient converges and it corresponds to $\phi:=\alpha=(1+\sqrt{5})/2$, the \textit{Golden Ratio}. The following theorems, which correspond to \eqref{binet1} and \eqref{binet2} respectively, give the Binet formula of the Fibonacci sequence $\{F_n\}_{n\ge 0}$ and its companion $\{L_n\}_{n\ge 0}$ in terms of the Golden Ratio. 

\medskip

\begin{theorem}
\label{eq1:fibo}
If $F_n$ is the $n^{th}$ Fibonacci number,
$$F_n=\frac{\alpha^n - \beta^{n}}{\alpha -\beta}= \frac{1}{\sqrt{5}}(\alpha^n - \beta^{n})\quad\text{for $n\geq 0$}.$$
\end{theorem}

\begin{theorem}
\label{eq1:Lucas}
We have
$$ 
L_n=\alpha^n-\beta^{n}~~~~~~~~~~~~\text{for $n\geq 0$}.
$$
\end{theorem}

%


The sequence of Fibonacci numbers can also be extended to negative indices $n$ by rewriting the recurrence by 
$$
F_{n-2} = F_{n} - F_{n-1}.
$$
Using Theorems \ref{eq1:fibo} and \ref{eq1:Lucas}, we have the following Lemma.
\begin{lemma}
\label{eq1:fibluc}
\begin{enumerate}
\item[(i)] $F_{-n}=(-1)^{n+1}F_n$ and $L_{-n}=(-1)^nL_n;$
\item[(ii)] $2F_{m+n}=F_mL_n+L_nF_m$ and $2L_{m+n}=5F_mF_n+L_mL_n;$
\item [(iii)]$L_{2n}=L_n^2+2(-1)^{n+1};$
\item[(iv)] $L_n^2-5F_n^2=4(-1)^n;$
\item[(v)] Let $p>5$ be a prime number. If $(\frac{5}{p})=1$ then $p\mid F_{p-1}.$ Otherwise $p \mid F_{p+1}.$ 
\item[(vi)] If $m\mid n$ and $\frac{n}{m} $ is odd, then $L_m\mid L_n$
\item[(vii)] Let $p$ and $n$ be positive integers such that $p$ is odd prime. Then $(L_p,F_n)>2$ if and only if $p\mid n $ and $n/p$ is even.
\item[(viii)] 
\begin{equation}
\label{eq33:3}
L_n-1=\left\{\begin{matrix}
5F_{(n+1)/2} F_{(n-1)/2} & {\text{\rm if}} & n\equiv 1\pmod 4;\\
L_{(n+1)/2} L_{(n-1)/2} & {\text{\rm if}} & n\equiv 3\pmod 4.
\end{matrix}\right.
\end{equation}
\end{enumerate}
\end{lemma}

\begin{proof} : See. \cite{A}, \cite{B}, \cite{FL2}.
\end{proof}


Many Diophantine equations involving Fibonacci and Lucas numbers being squares, perfect powers of the larger exponents of some others integers were proved over the past years.  Here we give few of them. 

\begin{lemma}[Bugeaud, Luca, Mignotte and Siksek, \cite{BMS} and \cite{BLMS} ]
\label{lem1:BLMS}
The equation $L_n=y^k$ with some $k\ge 1$ implies that $n\in \{1,3\}$. Furthermore, the only solutions of the equation $L_n=q^a y^k$ for some prime $q<1087$ and integers $a> 0,~k\ge 2$ have
$n\in \{0,2,3,4,5,6,7,8,9,11,13,17\}$.
\end{lemma}

We also recall a result about square-classes of members of Lucas sequences due to McDaniel and Ribenboim (see \cite{RM}).

\begin{lemma}
\label{lem1:RM}
If $L_mL_n=\square$ with $n>m\ge 0$, then $(m,n)=(1,3),~(0,6)$ or $(m,3m)$ with $3\nmid m$ odd.
\end{lemma}
\medskip

Before we end this section, we recall well-known results on the sequences $\{P_n\}_{n\geq 0}$ and $\{Q_n\}_{n\geq 0}$ which can be easily proved using formulas \eqref{binet1} and \eqref{binet2}.
\begin{lemma}
\label{lem1:PQ}
The relation
 $$Q_n^2 - 8P_n^2=4(-1)^n$$
holds for all $n\ge 0$.
\end{lemma}
\begin{lemma}
\label{lem1:orderof2}
The relations
\begin{itemize}
\item[(i)] $\nu_2(Q_n)=1$,
\item[(ii)] $\nu_2(P_n)=\nu_2(n)$
\item[(iii)]
\begin{equation} 
\label{eq3:RelP}
P_n-1=\left\{ \begin{matrix}
P_{(n-1)/2}Q_{(n+1)/2} & {\text{if}} & n\equiv 1\pmod 4;\\
P_{(n+1)/2}Q_{(n-1)/2} & {\text{if}} & n\equiv 3\pmod 4.
\end{matrix}\right.
\end{equation}
\end{itemize}
hold for all positive integers $n$.
\end{lemma}

\section{The Primitive Divisor Theorem}
\label{sec4}
In this section, we recall the Primitive Divisor Theorem of members for Lucas sequences.
In fact, the Primitive Divisor Theorem is an extension of Zygmondy's theorem.

\begin{definition}
Let $u_n$ be a Lucas sequence. The integer $p$, with $p\nmid \Delta$ is said to be a primitive divisor for $u_n$ if $p$ divide $ u_n$ and $p$ not divide $u_m$ for all $1<m<n$. In other words, a prime factor $p$ of $u_n$ such that $z(p)=n$ is a primitive prime.
\end{definition}
\medskip


The existence of primitive divisors is an old problem which has been completely solved by Yuri Bilu, Guillaume Hanrot and Paul M Voutier \cite{bilu}. Here we state it.

\begin{theorem}
For $n>30$, the $n$th term $u_n$ of any Lucas sequence has a primitive divisor.
\end{theorem}
\medskip

Carmichael proved the theorem in the case when the roots of the characteristic polynomial of $u_n$ are real. In particular he proved the following results.
 
\begin{lemma}[Carmichael \cite{Car}] \label{lem2:Car}
$F_n$ has a primitive divisor for all $n\geq 12.$
\end{lemma}

\begin{lemma}[Carmichael \cite{Car}] \label{lem1:Car}
$L_n$ has a primitive divisor for all $n\ne 6$, while $L_6=2\times 3^2$, and $2\mid L_3,~3\mid L_2$.   
\end{lemma}
\medskip

%
\medskip

Moreover, any primitive prime $p$ such that $p\mid u_n$, satisfies the congruence of the following theorem.

\begin{theorem}
\label{eq11:primitive}
If  $p$ is a primitive divisor of a Lucas sequences $u_n$, 
then
$$p\equiv \pm 1 \pmod{n}.$$
In fact $\(\frac{\Delta}{p}\)=\pm 1$. Further, when $u_n=F_n$, then

\begin{center}
If $p\equiv 1 \pmod{5}$ then $p\equiv 1 \pmod{n}$ 
\medskip

If $p\not\equiv 1 \pmod{5}$ then $p\equiv -1 \pmod{n}$
\end{center}

\end{theorem}

We recall a result of McDaniel on the prime factors of the Pell sequence $\{P_n\}_{n\geq 0}.$
\begin{lemma}[McDaniel \cite{MD}]
\label{lem1:Mac}
Let $\{P_n\}_{n\geq 0}$ be the sequence of Pell numbers. Then $P_n$ has a prime factor $q\equiv 1\pmod 4$
for all $n>14$.
\end{lemma}

It is known that $P_n$ has a primitive divisor for all $n\ge 2$ (see \cite{Car} or \cite{bilu}). Write $P_{z(p)}=p^{e_p} m_p$, where $m_p$ is coprime to $p$. It is known that if $p^k\mid P_n$ for some $k>e_p$, then $pz(p)\mid n$. In particular,
\begin{equation}
\label{eq1:nupvaluationofPn}
\nu_p(P_n)\le e_p\quad {\text{\rm whenever}}\quad p\nmid n.
\end{equation}
We need a bound on $e_p$. We have the following result.

\begin{lemma}
\label{lem1:zofp}
The inequality
\begin{equation}
\label{eq1:eofp}
e_p\le \frac{(p+1)\log \alpha}{2\log p}.
\end{equation}
holds for all primes $p$.
\end{lemma}

\begin{proof}
Since $e_2=1$, the inequality holds for the prime $2$. Assume that $p$ is odd. Then $z(p)\mid p+\varepsilon$ for some $\varepsilon\in \{\pm 1\}$. Furthermore, by Theorem \ref{eq1:binet1} we have
$$
p^{e_p}\mid P_{z(p)}\mid P_{p+\varepsilon}=P_{(p+\varepsilon)/2}Q_{(p+\varepsilon)/2}.
$$
By Lemma \ref{lem1:PQ}, it follows easily that $p$ cannot divide both $P_n$ and $Q_n$ for $n=(p+\varepsilon)/2$ since otherwise $p$ will also divide
$$
Q_n^2-8P_n^2=\pm 4,
$$
which is a contradiction since $p$ is odd. Hence, $p^{e_p}$ divides one of $P_{(p+\varepsilon)/2}$ or $Q_{(p+\varepsilon)/2}$. If $p^{e_p}$ divides $P_{(p+\varepsilon)/2}$, we have, by \eqref{eq1:sizePn}, that
$$
p^{e_p}\le P_{(p+\varepsilon)/2}\le P_{(p+1)/2}<\alpha^{(p+1)/2},
$$
which leads to the desired inequality \eqref{eq1:eofp} upon taking logarithms of both sides. In case $p^{e_p}$ divides $Q_{(p+\varepsilon)/2}$, we use the fact that $Q_{(p+\varepsilon)/2}$ is even by Lemma \ref{lem1:orderof2} (i). Hence, $p^{e_p}$ divides $Q_{(p+\varepsilon)/2}/2$, therefore, by formula \eqref{binet2}, we have
$$
p^{e_p}\le \frac{Q_{(p+\varepsilon)/2}}{2}\le \frac{Q_{(p+1)/2}}{2}<\frac{\alpha^{(p+1)/2}+1}{2}<\alpha^{(p+1)/2},
$$
which leads again to the desired conclusion by taking logarithms of both sides.
\end{proof}


Before we end our discussion on preliminaries of this thesis, it will be helpful to recall some basic definitions and results from Diophantine approximations which will be very  useful for Chapter \textcolor{red}{5}.  We will also introduce others tools as continued fractions in Section \textcolor{red}{\ref{sec22}}. 

\section{Linear forms in logarithms}
\label{sec21}
\subsection{Algebraic Numbers}
\begin{definition}
A complex (or real) number $\eta$ is an algebraic number if it is the root of a polynomial 
$$f(x)=a_nx^n + \ldots + a_1x + a_0$$
with integers coefficients $a_n\neq 0 \quad \hbox{for all $n$}$.
\end{definition}

\medskip

\medskip

Let $\eta$ be an algebraic number of degree $d$ over $\mathbb{Q}$ with minimal primitive polynomial over the integers
$$
f(X) = a_0 \prod_{i=1}^{d}(X-\eta^{(i)}) \in \mathbb{Z}[X],
$$
where the leading coefficient $a_0$ is positive. The \textit{logarithmic height} of  $\eta$ is given by
$$
h(\eta) = \dfrac{1}{d}\left(\log a_0 + \sum_{i=1}^{d}\log\max\{|\eta^{(i)}|,1\}\right).
$$ 
\medskip

\subsection{Linear forms in logarithms of rational numbers}
We first define what is meant by \textit{linear forms in logarithms}. We refer to the description given in \cite{Yan}.

Let $n$ be an integer. For $i=1,\ldots,n$,  let $x_i/y_i$ be a non zero rational number, $b_i$ be a positive integer and set $A_i:=max\{|x_i|,|y_i|,3\}$ and $B:=max\{b_i,\ldots,b_n,3\}$. We consider the quantity:
$$ \Lambda:=\(\frac{x_1}{y_1}\)^{b_1}\cdots \(\frac{x_n}{y_n}\)^{b_n} - 1,$$
which occurs naturally in many Diophantine equations. It's often easy to prove that $\Lambda\neq 0$ and to find an upper bound for it. For applications to Diophantine problems, it is important that not only the
above linear form is nonzero, but also that we have a strong enough lower bound for the absolute value of this linear form. Since $|\Lambda|\leq \frac{1}{2}$, the reason why it is called \textit{linear form in logarithms} is that,
$$ |\Lambda|\geq \frac{|\log(1 + \Lambda )|}{2}=\frac{1}{2}\Big|b_1\log \frac{x_1}{y_1} + \ldots + b_n\log \frac{x_n}{y_n}\Big|.$$
Assuming that $\Lambda$ is nonzero, one can state a trivial lower bound for $|\Lambda|$. Then,

$$ |\Lambda|\geq - \sum_{i=1}^{n}b_i\log |y_i|\geq -B\sum_{i=1}^{n}\log A_i.$$

The dependence on the $A_i$'s is very satisfactory, unlike the dependence on $B$. However,
to solve many Diophantine questions, we need a better dependence on $B$, even if the
dependence on the $A_i$'s is not the best possible. More generally, one can state analogous lower bounds when the $x_i/y_i$ are replaced by nonzero algebraic numbers $\eta_i$ , the real numbers $A_i$ being then expressed in terms of the logarithmic height of $\eta_i$.

\subsection*{Lower bounds for linear forms in logarithms of algebraic numbers} 

A lower bound for a nonzero expression of the form

$$\eta_1^{b_1}\ldots\eta_n^{b^n}-1,$$
where $\eta_1 , \ldots ,\eta_n$ are algebraic numbers and $b_1 ,\ldots,b_n$ are integers, is the same as a lower bound for a nonzero number of the form 

\begin{equation}
\label{eq:LF}
\sum_{i=1}^{n}b_i \log \eta_i,
\end{equation}
since $e^z-1\thicksim z$ for $z\rightarrow 0.$ The first nontrivial lower bounds were obtained by A.O. Gel'fond. His estimates were effective only for $n=2.$ Later, A. Schinzel deduced explicit Diophantine results using the approach introduced by A.O. Gel'fond. In $1968$, A. Baker succeeded to extend to any $n\geq 2$ the transcendence method used by A.O. Gel’fond for $n=2.$
\medskip

\begin{theorem}[A. Baker, $1975$] Let $\eta_1,\ldots, \eta_n$ be algebraic numbers from $\mathbb{C}$ different from $0, 1.$ Further, let $b_1,\ldots,b_n$ be rational integers such that
$$b_1 \log\eta_1 +\cdots + b_n \log\eta_n \neq 0.$$
Then
$$|b_1 \log\eta_1 +\cdots + b_n\log\eta_n |\geq (eB)^{-C},$$
where $B := \max(|b_1 |,\ldots, |b_n |)$ and $C$ is an effectively computable constant
depending only on $n$ and on $\eta_1,\ldots, \eta_n.$ 
\end{theorem}

Baker's result marked a rise of the area of effective resolution of the Diophantine equations of certain types, more precisely those that can be reduced to exponential equations. Many important generalizations and improvements of Baker's result have been obtained. We refer to a paper of Baker (see \cite{baker}) for an interesting survey on these results.
\medskip

From this theory, we recall results that we shall use in this thesis. We present a Baker type inequality with explicit constants which is easy to apply. Here, we give the result of Matveev \cite{matveev}.

\subsection{Matveev's Theorem}
\begin{theorem}[Matveev \cite{matveev}]\label{thm3:Matveev}
Let $\K$ be a number field of degree $D$ over ${\mathbb Q}$  $\eta_1, \ldots, \eta_t$ be positive
real numbers of ${\mathbb K}$, and $b_1, \ldots,  b_t$ rational integers. Put
$$
\Lambda = \eta_1^{b_1} \cdots \eta_t^{b_t}-1
\qquad
\text{and}
\qquad
B \geq \max\{|b_1|, \ldots ,|b_t|\}.
$$
Let $A_i \geq \max\{Dh(\eta_i), |\log \eta_i|, 0.16\}$ be real numbers, for
$i = 1, \ldots, t.$
Then, assuming that $\Lambda \not = 0$, we have
$$
|\Lambda| > \exp(-1.4 \times 30^{t+3} \times t^{4.5} \times D^2(1 + \log D)(1 + \log B)A_1 \cdots A_t).
$$
\end{theorem}

%
%
%
%

\section{Continued Fractions}
\label{sec22}

We recall some definitions and properties of continued fractions. The material of this section is mainly from \cite{dio}. We also present the reduction method based on a lemma of Baker--Davenport \cite{BD}.

\begin{definition}
A finite continued fraction is an expression of the form

\begin{enumerate}

\item 
\begin{equation*}
b_0 + \cfrac{1}{b_1 + \cfrac{1}{b_2 + \cfrac{1}{\ddots+\cfrac{1}{b_n}}}}
\end{equation*}
where $b_0 \in \mathbb{R}$ and $b_i \in \mathbb{R}_{>0}$ for all $1 \leq i \leq n.$ We use the
notation $[b_0 , b_1 ,\ldots,b_n ]$ for the above expression.
\item The continued fraction $[b_0 , b_1 ,\ldots,b_n ]$ is called simple if~~  $b_0 , b_1 ,\ldots,b_n \in\mathbb{Z}.$
\item  The continued fraction $C_j= [b_0 , b_1 ,\ldots,b_j ]$ with $0 \leq j \leq n$ is called the
$j$-th convergent of $[b_0 , b_1 ,\ldots,b_n ].$
\end{enumerate}
\end{definition}

One can easily see that every simple continued fraction is a rational number. Conversely, using the Euclidean algorithm, every rational number can be represented as a simple continued fraction; however the expression is not unique. For example, the continued fraction of $\frac{1}{4}=[0,4]=[0,3,1].$ However, if $b_n>1$, then the representation of a rational number as a finite continued fraction is unique. Continued fractions are important in many branches of mathematics, and particularly in the theory of approximation to real numbers by rationals.

\begin{definition}
Let $(a_n )_{n\geq 0}$ be an infinite sequence of integers with $a_n>0$ for all $n \geq 1.$ The infinite continued fraction is defined as the limit of the finite continued fraction

$$[a_0 , a_1 , \ldots] := \lim_{n\rightarrow \infty} C_n.$$

Infinite continued fractions always represent irrational numbers. Conversely, every irrational number can be expanded in an infinite continued fraction.
\end{definition}
\medskip

\begin{example}
The most basic of all continued fractions is the one using all 1's:
 $$[1,1,1,1,....]= 1 +\cfrac{1}{1+\cfrac{1}{...}}.$$
If we put $x$ for the dots numbers then $x=[1,x]$. So $x = 1+\frac{1}{x}$ which is equivalent to $x^2 - x - 1= 0$. The quadratic formula gives us, $x=\frac{1+\sqrt{5}}{2} $, i.e., the \textit{golden ratio.}
\end{example}

\subsection{Some properties of continued fractions}

In this section, we use the following notations:

\medskip

\begin{eqnarray*}
\begin{aligned}[l]
p_0 &= a_0,\\
p_1 &= a_0a_1 + 1,\\
p_j &= a_jp_j + p_{j-1},
\end{aligned}
\qquad\qquad
\begin{aligned}[l]
q_0 &= 1,\\
q_1 &= a_1,\\
q_j &=a_jq_{j-1} + q_{j-2}.
\end{aligned}
\end{eqnarray*}
\medskip
The following theorem indicates that the convergents $C_j=p_j/q_j$ give the best approximations by rationals of the irrational number $\alpha$.
\medskip

\begin{theorem}(Convergents). 
\begin{enumerate}

\item Let $\alpha$ be an irrational number and let $C_j=p_j/q_j$ for $j\geq 0$ be the convergents of the continued fraction of $\alpha$. If $r,s \in \mathbb{Z}$ with $s>0$ and $k$ is a positive integer such that 

$$|s\alpha - r| < |q_j \alpha - p_j |,$$

then $s \geq q_j+1.$
\item If $\alpha$ is irrational and $r/s$ is a rational number with $s > 0$ such that
$$  \Big| \alpha - \frac{r}{s}\Big| <\frac{1}{2s^2},$$ 

then $r/s$ is a convergent of $\alpha.$

\end{enumerate}
\end{theorem}

\begin{proof}
See \cite{dio}.
\end{proof}

\begin{theorem}[Legendre's Theorem] If $\alpha$ is an irrational number and $p/q$ is a
rational number in lowest terms, $q > 0$, such that

$$  \Big| \alpha - \frac{p}{q}\Big| <\frac{1}{2q^2},$$ 
then $p/q$ is a convergent of the continued fraction of $\alpha$.
\end{theorem}
Legendre's Theorem is an important result in the study of  continued fractions, because it tells us that good approximations of irrational numbers by rational numbers are given by its convergents.

\subsection{The Baker--Davenport Lemma}

In $1998$, Dujella and Peth\H o in \cite[Lemma 5$(a)$]{DP} gave a version of the reduction method based on a lemma of Baker--Davenport lemma \cite{BD}.  The next lemma from \cite{BL1}, gave a variation of their  result. This will be our key tool used to reduce the upper bounds on the variable $n$ in Chapter \textcolor{red}{5}.

\begin{lemma} \label{reduce}
Let $M$ be a positive integer, let $p/q$ be a convergent of the continued fraction of the irrational $\gamma$ such that $q>6M$, and let $A,B,\mu$ be some real numbers with $A>0$ and $B>1$. Let $\epsilon:=||\mu q||-M||\gamma q||$, where $||\cdot||$ denotes the distance from the nearest integer. If $\epsilon >0$, then there is no solution to the inequality

\begin{equation}
\label{eq:DP}
0<u\gamma-v+\mu<AB^{-w},
\end{equation}
in positive integers $u,v$ and $w$ with
$$
u\leq M \quad\text{and}\quad w\geq \frac{\log(Aq/\epsilon)}{\log B}.
$$
\end{lemma}

\begin{proof}
We proceed as in the proof of Lemma 5 in \cite{DP}. In fact, we assume that $0\leq u\leq M$. Multiplying the relation \eqref{eq:DP} by keeping in mind that $||q\gamma||=|p-q\gamma|$ (because $p/q$ is a convergent of $\gamma$), we then have that

\begin{eqnarray*}
qAB^{-w} &>& q\mu - qv + qu\gamma = |q\mu - qv + qu\gamma | \\
&=& |q\mu -(qv - up) -u(p -q\gamma)|\\
&\geq & |q\mu -(qv -up)| - u|p - q\gamma |\\
&\geq & ||q\mu|| - u||q\gamma || \\
&\geq &||q\mu|| - M ||q\gamma || = \epsilon ,
\end{eqnarray*}
leading to $$w < \frac{\log(Aq/\epsilon)}{\log B}.$$
\end{proof}


\chapter{Pell numbers whose Euler function is a Pell number}

In this chapter, we find all the members of the Pell sequence whose Euler's function is also a member of the Pell sequence. We prove that the only solutions are $1$ and $2$. The material of this chapter is the main result in \cite{FL1}.

\section{Introduction}

%
In this chapter, we have the following result.
\begin{theorem}
\label{thm1:2}
The only solutions in positive integers $(n,m)$ of the equation
\begin{equation}
\label{eq1:pb}
\varphi(P_n)=P_m
\end{equation}
are $(n,m)=(1,1),(2,1).$
\end{theorem}
For the proof, we begin by following the method from \cite{G}, but we add to it some ingredients from \cite{L}.

\section{Preliminary results}

Let $(\alpha,\beta)=(1+{\sqrt{2}},1-{\sqrt{2}})$ be the roots of the characteristic equation $x^2-2x-1=0$ of the Pell sequence $\{P_n\}_{n\ge 0}$. Formula \eqref{eq1:binet} implies easily that the inequalities
\begin{equation}
\label{eq1:sizePn}
\alpha^{n-2}\le P_n\le  \alpha^{n-1}
\end{equation}
hold for all positive integers $n$.


\medskip

We need some inequalities from the prime number theory. The following inequalities (i), (ii) and (iii) are inequalities (3.13), (3.29) and (3.41) in \cite{RS}, while (iv) is Theorem 13 from \cite{MNR}.

\begin{lemma}
\label{lem1:RS}
Let $p_1<p_2<\cdots$ be the sequence of all prime numbers. We have:
\begin{itemize}
\item[(i)] The inequality
$$
p_n<n(\log n+\log\log n)
$$
holds for all $n\ge 6$.
\item[(ii)]
The inequality
$$
\prod_{p\le x} \left(1+\frac{1}{p-1}\right)<1.79\log x\left(1+\frac{1}{2(\log x)^2}\right)
$$
holds for all $x\ge 286$.
\item[(iii)]
The inequality
$$
\varphi(n)>\frac{n}{1.79 \log\log n+2.5/\log\log n}
$$
holds for all $n\ge 3$.
\item[(iv)] The inequality
$$
\omega(n)<\frac{\log n}{\log \log n-1.1714}
$$
holds for all $n\ge 26$.
\end{itemize}
\end{lemma}
\medskip

For a positive integer $n$, we put ${\mathcal P}_n=\{p: z(p)=n\}$.
We need the following result.
\begin{lemma}
\label{lem1:sumSn}
Put
$$
S_n:=\sum_{p\in {\mathcal P}_n} \frac{1}{p-1}.
$$
For $n>2$, we have
\begin{equation}
\label{eq1:Sn}
S_n<\min\left\{\frac{2\log n}{n},\frac{4+4\log\log n}{\varphi(n)}\right\}.
\end{equation}
\end{lemma}

\begin{proof}
Since $n>2$, it follows that every prime factor $p\in {\mathcal P}_n$ is odd and satisfies the congruence $p\equiv \pm 1\pmod n$. Further, putting $\ell_n:=\#{\mathcal P}_n$, we have
$$
(n-1)^{\ell_n}\le \prod_{p\in {\mathcal P}_n} p\le P_n<\alpha^{n-1}
$$
(by inequality \eqref{eq1:sizePn}), giving
\begin{equation}
\label{eq1:S}
\ell_n\le \frac{(n-1)\log \alpha}{\log(n-1)}.
\end{equation}
Thus, the inequality
\begin{equation}
\label{eq1:ellofn}
\ell_n<\frac{n\log \alpha}{\log n}
\end{equation}
holds for all $n\ge 3$, since it follows from \eqref{eq1:S} for $n\ge 4$ via the fact that the function $x\mapsto x/\log x$ is increasing for $x\ge 3$, while for $n=3$ it can be checked directly. To prove the first bound,
we use \eqref{eq1:ellofn} to deduce that
\begin{eqnarray}
\label{eq1:last}
S_n & \le & \sum_{1\le \ell\le \ell_n} \left(\frac{1}{n\ell-2} +\frac{1}{n\ell}\right)\nonumber\\
& \le &  \frac{2}{n}\sum_{1\le \ell \le \ell_n} \frac{1}{\ell}+\sum_{m\ge n} \left(\frac{1}{m-2}-\frac{1}{m}\right)\nonumber\\
& \le & \frac{2}{n} \left(\int_{1}^{\ell_n} \frac{dt}{t} +1\right)+\frac{1}{n-2}+\frac{1}{n-1}\nonumber\\
& \le & \frac{2}{n}\left(\log \ell_n+1+\frac{n}{n-2}\right)\nonumber\\
& \le & \frac{2}{n} \log\left( n\left(\frac{(\log \alpha) e^{2+2/(n-2)}}{\log n}\right)\right).
\end{eqnarray}
Since the inequality
$$
\log n>(\log \alpha) e^{2+2/(n-2)}
$$
holds for all $n\ge 800$, \eqref{eq1:last} implies that
$$
S_n<\frac{2\log n}{n}\quad {\text{\rm for}}\quad n\ge 800.
$$
The remaining range for $n$ can be checked on an individual basis.
For the second bound on $S_n$, we follow the argument from \cite{L} and split the primes in ${\mathcal P}_n$ in three groups:
\begin{itemize}
\item[(i)] $p<3n$;
\item[(ii)] $p\in (3n,n^2)$;
\item[(iii)] $p>n^2$;
\end{itemize}
We have
\begin{equation}
\label{eq1:T1}
T_1=\sum_{\substack{p\in {\mathcal P}_n\\ p<3n}} \frac{1}{p-1}
\le \left\{
\begin{matrix}
{\displaystyle{\frac{1}{n-2}+\frac{1}{n}+\frac{1}{2n-2}+\frac{1}{2n}+\frac{1}{3n-2} }} & < &
{\displaystyle{\frac{10.1}{3n}}},& n\equiv 0\pmod 2,\\
{\displaystyle{\frac{1}{2n-2}+\frac{1}{2n}}} & < & {\displaystyle{\frac{7.1}{3n}}}, & n\equiv 1\pmod 2,\\
\end{matrix}\right.
\end{equation}
where the last inequalities above hold for all $n\ge 84$. For the remaining primes in ${\mathcal P}_n$, we have
\begin{equation}
\label{eq1:T2T3}
\sum_{\substack{p\in {\mathcal P}_n\\ p>3n}} \frac{1}{p-1}<
\sum_{\substack{p\in {\mathcal P}_n\\ p>3n}} \frac{1}{p}+\sum_{m\ge 3n+1} \left(\frac{1}{m-1}-\frac{1}{m}\right)=T_2+T_3+\frac{1}{3n},
\end{equation}
where $T_2$ and $T_3$ denote the sums of the reciprocals of the primes in ${\mathcal P}_n$ satisfying (ii) and (iii), respectively. The sum $T_2$ was estimated in \cite{L} using the large sieve inequality of Montgomery and Vaughan \cite{MV} which asserts that
\begin{equation}
\label{eq11:11}
\pi(x;d,1)<\frac{2x}{\varphi(d)\log(x/d)}\quad {\text{\rm for~all}}\quad x>d\ge 2,
\end{equation}
where $\pi(x;d,1)$ stands for the number of primes $p\le x$ with $p\equiv 1\pmod d,$ and the bound on it is
\begin{equation}
\label{eq1:T2}
T_2=\sum_{3n<p<n^2} \frac{1}{p}<\frac{4}{\varphi(n)\log n}+\frac{4\log\log n}{\varphi(n)}<\frac{1}{\varphi(n)}+\frac{4\log\log n}{\varphi(n)},
\end{equation}
where the last inequality holds for $n\ge 55$. Finally, for $T_3$, we use the estimate \eqref{eq1:ellofn} on $\ell_n$ to deduce that
\begin{equation}
\label{eq1:T3}
T_3< \frac{\ell_n}{n^2}<\frac{\log \alpha}{n\log n}<\frac{0.9}{3n},
\end{equation}
where the last bound holds for all $n\ge 19$. To summarize, for $n\ge 84$, we have, by \eqref{eq1:T1}, \eqref{eq1:T2T3}, \eqref{eq1:T2} and \eqref{eq1:T3},
$$
S_n<\frac{10.1}{3n}+\frac{1}{3n}+\frac{0.9}{3n}+\frac{1}{\varphi(n)}+\frac{4\log\log n}{\varphi(n)}<\frac{4}{n}+\frac{1}{\varphi(n)}+\frac{4\log\log n}{\varphi(n)}\le \frac{3+4\log\log n}{\varphi(n)}
$$
for $n$ even, which is stronger $~~$ that the desired inequality. Here, we $~~$ used that $\varphi(n)\le n/2$ for even $n$. For odd $n$, we use the same argument except that the first fraction $10.1/(3n)$ on the right--hand side above gets replaced by $7.1/(3n)$ (by \eqref{eq1:T1}), and we only have $\varphi(n)\le n$ for odd $n$.
This was for $n\ge 84$. For $n\in [3,83]$, the desired inequality can be checked on an individual basis.
\end{proof}


\section{Proof of Theorem \ref{thm1:2}}
\subsection*{A bird's eye view of the proof}

In this section, we explain the plan of attack for the proof Theorem \ref{thm1:2}.  We assume $n>2$. We put $k$ for the number of distinct prime factors of $P_n$ and $\ell=n-m$. 
We first show that $2^k\mid m$ and that any possible solution must be large. This only uses the fact that $p-1\mid \phi(P_n)=P_m$ for all prime factors $p$ of $P_n$, and all such primes with at most one exception are odd. 
We show that $k\ge 416$ and $n>m\ge 2^{416}$. This is Lemma \ref{lem1:1}. We next  bound $\ell$ in terms of $n$ by showing that $\ell<\log\log\log n/\log \alpha+1.1$  (Lemma \ref{lem1:2}). Next we show that $k$ is large, by proving that 
$3^k>n/6$ (Lemma \ref{lem1:3}). When $n$ is odd, then $q\equiv 1\pmod{4}$ for all prime factor $q$ of $P_n$. This implies that  $4^k\mid m$. Thus, $3^k>n/6$ and $n>m\ge 4^k$, a contradiction in our range for $n$. This is done in Subsection \ref{Subsec:odd}. 
When $n$ is even, we write $n=2^s n_1$ with an odd integer $n_1$ and bound $s$ and the smallest prime factor $r_1$ of $n_1$. We first show that $s\le 3$, that if $n_1$ and $m$ have a common divisor larger than $1$, then  $r_1\in \{3,5,7\}$ (Lemma \ref{lem1:4}). A lot of effort is spent into finding a small bound on $r_1$. As we saw, $r_1\le 7$ if $n_1$ and $m$ are not coprime. When $n_1$ and $m$ are coprime, we succeed in proving that $r_1<10^6$. Putting $e_r$ for the exponent of $r$ in the factorization of $P_{z(r)}$, it turns out that our argument works well when $e_r=1$ and we get a contradiction, but when $e_r=2$, then we need some additional information about the prime factors of $Q_r$. It is always the case  that $e_r=1$ for all primes $r<10^6$, except for $r\in \{13,31\}$ for which $e_r=2$, but, lucky for us, both $Q_{13}$ and $Q_{31}$ have two suitable prime factors each which allows us to obtain a contradiction. Our efforts in obtaining $r_1<10^6$ involve quite a complicated 
argument  (roughly the entire argument after Lemma \ref{lem1:4} until the end), which we believe it is justified by the existence of the mighty
prime $r_1=1546463$, for which $e_{r_1}=2$. Should we have only obtained say $r_1<1.6\times 10^6$, we would have had to say something nontrivial about the prime factors of $Q_{15467463}$, a nuisance which we succeeded in avoiding simply by proving that $r_1$ cannot get that large!

\subsection*{Some lower bounds on $m$ and $\omega(P_n)$}

Firstly, by a computational search we get that when $n\le 100$, the only solutions are obtained at $n=1,~2$. So, from now on $n>100$. We write
\begin{equation}
\label{eq1:1}
P_n=q_1^{\alpha_1}\ldots q_k^{\alpha_k},
\end{equation}
where $q_1<\cdots <q_k$  are primes and $\alpha_1,\ldots,\alpha_k$ are positive integers. Clearly, $m<n$.

Lemma \ref{lem1:Mac} on page $17$ applies, therefore
$$
4\mid q-1\mid \varphi(P_n)\mid P_m,
$$
thus $4\mid m$ by Lemma \ref{lem1:orderof2}. Further,
it follows from \cite{FL}, that $\varphi(P_n)\geq P_{\varphi(n)}.$ Hence, $m\ge \varphi(n)$.
Thus,
\begin{equation}
\label{eq1:low}
m\geq \varphi(n)\geq \frac{n}{1.79\log\log n + 2.5/\log\log n},
\end{equation}
by Lemma \ref{lem1:RS} (iii). The function
$$
x\mapsto \frac{x}{1.79 \log\log x+2.5/\log\log x}
$$
is increasing for $x\ge 100$. Since $n\ge 100$, inequality \eqref{eq1:low} together with the fact that $4\mid m$, show that $m\ge 24$.

Let $\ell=n-m$. Since $m$ is even, then $\beta^m>0$, and
\begin{equation}
\label{eq1:10tominus40}
\frac{P_n}{P_m}=\frac{\alpha^n-\beta^n}{\alpha^m -\beta^m}> \frac{\alpha^n-\beta^n}{\alpha^m}\ge \alpha^\ell -\frac{1}{\alpha^{m+n}} > \alpha^\ell - 10^{-40},
\end{equation}
where we used the fact that
$$
\frac{1}{\alpha^{m+n}}\le \frac{1}{\alpha^{124}}<10^{-40}.
$$
We now are ready to provide a large lower bound on $n$. We distinguish the following cases.

\medskip

\textbf{\small{Case 1}}: {\it $n$ is odd}.

\medskip

Here, we have $\ell\geq 1$. So,
$$
\frac{P_n}{P_m}> \alpha - 10^{-40}>2.4142.
$$
Since $n$ is odd, then $P_n$ is divisible only by primes $p$ with $z(p)$ being odd. There are precisely $2907$ among the first $10000$ primes with this property. We set them as
$$
\mathcal{F}_1=\{5, 13, 29, 37,53, 61, 101, 109, \ldots,104597, 104677, 104693, 104701, 104717\}.
$$
Since
$$
\prod_{p\in {\mathcal F}_1}\left(1-\frac{1}{p}\right)^{-1}<1.963<2.4142<\frac{P_n}{P_m}=\prod_{i=1}^k \left(1-\frac{1}{q_i}\right)^{-1},
$$
we get that $k>2907$. Since $2^k\mid \varphi(P_n)\mid P_m$, we get,
by Lemma \ref{lem1:orderof2}, that
\begin{equation}
\label{eq1:Case1}
n>m>2^{2907}.
\end{equation}

\textbf{\small{Case 2}}: $n\equiv 2\pmod 4$.
\medskip

Since both $m$ and $n$ are even, we get $\ell\geq 2.$ Thus,
\begin{equation}
\label{eq1:lowBoundCase2}
\frac{P_n}{P_m}> \alpha^2 - 10^{-40} > 5.8284.
\end{equation}
If $q$ is a factor of $P_n$, as in Case 1, we have that $4\nmid z(p).$  There are precisely $5815$ among the first $10000$ primes with this property. We set them again as
$$
\mathcal{F}_2=\{2, 5, 7, 13,23, 29, 31, 37, 41,47, 53, 61, \ldots, 104693, 104701, 104711, 104717\}.
$$
Writing $p_i$ as the $i$th prime number in  $\mathcal{F}_2$, a computation with Mathematica shows that
\begin{eqnarray*}
\prod_{i=1}^{415}\left(1-\frac{1}{p_i}\right)^{-1} & = & 5.82753\ldots\\
\prod_{i=1}^{416} \left(1-\frac{1}{p_i}\right)^{-1} & = & 5.82861\ldots,
\end{eqnarray*}
which via inequality \eqref{eq1:lowBoundCase2} shows that $k\ge 416$. Of the $k$ prime factors of $P_n$, we have that only $k-1$ of them are odd ($q_1=2$ because $n$ is even), but one of those is congruent to $1$ modulo $4$ by McDaniel's result (Lemma 9, page 17). Hence, $2^k\mid \varphi(P_n)\mid P_m$, which shows, via Lemma \ref{lem1:orderof2}, that
\begin{equation}
\label{eq1:Case2}
n>m\ge 2^{416}.
\end{equation}

\textbf{\small{Case 3}}: $4 \mid n$.
\medskip

In this case, since both $m$ and $n$ are multiples of $4$, we get that $\ell\geq 4$. Therefore,
$$
\frac{P_n}{P_m}> \alpha^4 - 10^{-40} > 33.97.
$$
Letting $p_1<p_2<\cdots$ be the sequence of all primes, we have that
$$
\prod_{i=1}^{2000}\left(1-\frac{1}{p_i}\right)^{-1}<17.41\ldots<33.97<\frac{P_n}{P_m}=
\prod_{i=1}^k \left(1-\frac{1}{q_i}\right),
$$
showing that $k>2000$. Since $2^{k}\mid \varphi(P_n)=P_m$, we get
 \begin{equation}
\label{eq1:Case3}
n> m\geq 2^{2000}.
\end{equation}
To summarize, from \eqref{eq1:Case1}, \eqref{eq1:Case2} and \eqref{eq1:Case3}, we get the following results.

\begin{lemma}
\label{lem1:1}
If $n>2$, then
\begin{enumerate}
\item $2^k\mid m$;
\item $k\ge 416$;
\item $n>m\ge 2^{416}$.
\end{enumerate}
\end{lemma}

\subsection*{\small Computing a bound for $\ell$ in term of $n$}

From the previous section, we have seen that $k\geq 416$. Since $n>m\ge 2^k$, we have
\begin{equation}
\label{eq1:k}
k<k(n):=\frac{\log n}{\log 2}.
\end{equation}
Let $p_i$ be the $i$th prime number. Lemma \ref{lem1:RS} shows that
$$
p_k\le p_{\lfloor k(n)\rfloor}\le k(n)(\log k(n) +\log\log k(n)):=q(n).
$$
We then have, using Lemma \ref{lem1:RS} (ii), that
$$
\frac{P_m}{P_n}=\prod_{i=1}^{k} \left(1-\frac{1}{q_i}\right)\geq \prod_{2\leq p\leq q(n)}\left(1-\frac{1}{p}\right)>\frac{1}{1.79\log q(n)(1+1/(2(\log q(n))^2))}.
$$
Inequality (ii) of Lemma \ref{lem1:RS} requires that $x\ge 286$, which holds for us with $x=q(n)$ because $k(n)\ge 416$. Hence, we get
$$
1.79\log q(n)\left(1+  \frac{1}{(2(\log q(n))^2)}\right)>\frac{P_n}{P_m}>\alpha^{\ell}-10^{-40}>\alpha^{\ell}
\left(1-\frac{1}{10^{40}}\right).
$$
Since $k\ge 416$, we have $q(n)> 3256$. Hence, we get
$$\log q(n) \left(1.79\left(1-\frac{1}{10^{40}}\right)^{-1}
\left(1+\frac{1}{2(\log(3256))^2}\right)\right)>\alpha^\ell,
$$
which yields, after taking logarithms, to
\begin{equation}
\label{eq1:5}
\ell \le \frac{\log\log q(n)}{\log \alpha}+0.67.
\end{equation}
The inequality
\begin{equation}
\label{eq1:1.5}
q(n)<(\log n)^{1.45}
\end{equation}
holds in our range for $n$ (in fact, it holds for all $n>10^{83}$, which is our case since for us $n>2^{416}>10^{125}$). Inserting inequality \eqref{eq1:1.5} into \eqref{eq1:5}, we get
$$
\ell<\frac{\log\log (\log n)^{1.45}}{\log \alpha}+0.67<\frac{\log\log\log n}{\log \alpha}+1.1.
$$
Thus, we proved the following result.

\begin{lemma}
\label{lem1:2}
If $n>2$, then
\begin{equation}
\ell < \frac{\log\log\log n}{\log \alpha}+1.1.
\end{equation}
\end{lemma}

\subsubsection*{\small Bounding the primes $q_i$ for $i=1,\ldots,k$}

Write
\begin{equation}
\label{eq1:51}
P_n= q_1\cdots q_k B,\quad {\text{\rm where}}\quad B=q_1^{\alpha_1-1} \cdots q_k^{\alpha_k-1}.
\end{equation}
Clearly, $B\mid \varphi(P_n)$, therefore $B\mid P_m$. Since also $B\mid P_n$, we have, by Theorem \ref{eq1:binet1}, that $B\mid \gcd(P_n,P_m)=P_{\gcd(n,m)}\mid P_{\ell}$ where the last relation follows again  by Theorem \ref{eq1:binet1} because $\gcd(n,m)\mid \ell.$ Using the inequality \eqref{eq1:sizePn} and Lemma \ref{lem1:2}, we get
\begin{equation}
\label{eq1:6}
B\leq P_{n-m}\leq \alpha^{n-m-1}\leq \alpha^{0.1}\log\log n.
\end{equation}
We now use the inductive argument from Section 3.3 in \cite{G} in order to find a bound for the primes $q_i$ for all $i=1,\ldots,k$. We write
$$
\prod_{i=1}^{k} \left(1-\frac{1}{q_i}\right)=\frac{\varphi(P_n)}{P_n}=\frac{P_m}{P_n}.
$$
Therefore,
$$ 1-
\prod_{i=1}^{k} \left(1-\frac{1}{q_i}\right)=1-\frac{P_m}{P_n}=\frac{P_n-P_m}{P_n} \ge \frac{P_n-P_{n-1}}{P_n}>\frac{P_{n-1}}{P_n}.
$$
Using the inequality
\begin{equation}
\label{eq1:7}
 1 -(1-x_1)\cdots(1-x_s) \leq x_1 + \cdots + x_s\quad {\text{\rm valid for all}}\quad x_i \in [0,1]~ {\text{\rm for}}~ i= 1,\ldots,s,
\end{equation}
we get,
therefore,
\begin{equation}
\label{eq1:8}
q_1< k \left(\frac{P_{n}}{P_{n-1}}\right) < 3k.
\end{equation}
Using an inductive argument on the index $i$ for $i\in \{1,\ldots,k\}$ , we now show that if we put
$$
u_i:=\prod_{j=1}^{i} q_j,
$$
then
\begin{equation}
\label{eq1:9}
u_i< \(2\alpha^{2.1} k \log\log n\)^{(3^i -1)/2}.
\end{equation}
\medskip
For $i=1$, we get 

$$ q_1 < \(2\alpha^{2.117}(\log\log n)k\) $$
which is implies by the inequality (\ref{eq1:8}) and the fact that $n>3\cdot 10^{150}$, we have that $\(2\alpha^{2.117}(\log\log n)\)> 61 >6$. We assume now that for $i\in \{1,\ldots,k-1\}$ the inequality (\ref{eq1:9}) is satisfied and let us prove it for $k$ by replacing $i$ by $i+1$. We have,

$$\prod_{j=i+1}^{k} \(1-\frac{1}{q_i}\) = \frac{q_1\cdots q_i}{(q_1-1)\cdots (q_i-1)}\cdot\frac{q_m}{q_n}=\frac{q_1\cdots q_i}{(q_1-1)\cdots (q_i-1)}\cdot\frac{\alpha^m-\beta^m}{\alpha^n -\beta^n}, $$
which we write as 

\begin{eqnarray*}
 1-\prod_{j=i+1}^{k} \(1-\frac{1}{q_i}\) &=& 1- \frac{q_1\cdots q_i}{(q_1-1)\cdots (q_i-1)}\cdot\frac{\alpha^m-\beta^m}{\alpha^n -\beta^n} \\
&=& \frac{\alpha^m((q_1-1) \cdots(q_i-1)\alpha^{n-m}  - q_1\cdots q_i)}{(q_1-1)\cdots(q_i-1)(\alpha^n -\beta^n)}\\
&+& \frac{\beta^m (q_1\cdots q_i   -\beta^{n-m}(q_1-1) \cdots(q_i-1))}{(q_1-1)\cdots(q_i-1)(\alpha^n -\beta^n)}\\
&=:& V+W
\end{eqnarray*}
with 

$$V:=\frac{\alpha^m((q_1-1) \cdots(q_i-1)\alpha^{n-m}  - q_1\cdots q_i)}{(q_1-1)\cdots(q_i-1)(\alpha^n -\beta^n)} ,$$
and
$$ W:=\frac{\beta^m (q_1\cdots q_i   -\beta^{n-m}(q_1-1) \cdots(q_i-1))}{(q_1-1)\cdots(q_i-1)(\alpha^n -\beta^n)}.$$

Since $m$ is even, then $|\beta|<1$. Therefore $W\geq 0 $. Further, since $n-m=\ell>0$,  and $\beta= - \alpha^{-1},$ it follows that $VW\neq 0$. Suppose that $V<0$. Then, 

\begin{eqnarray*}
1- \prod_{j=i+1}^{k} \(1-\frac{1}{q_i}\)&<& W < \frac{2q_1\cdots q_i}{ \alpha^m (q_1-1)\cdots(q_i-1)(\alpha^n-\beta^n)} \\
&<& \frac{2P_n}{\phi(P_n)(\alpha^m-\beta^m)(\alpha^n-\beta^n)} = \frac{1}{4P_m^2}.
\end{eqnarray*}

Since the denominator of the positive rational integer on the left hand side of the above inequality divides $q_{i+1}\cdots q_k \mid P_n$, it follows that this number is at least as large as $1/P_n$. Hence, 

$$\frac{1}{P_n} < \frac{1}{4P_m^2}, \quad \quad \hbox{ which gives   $P_m^2 < \frac{1}{4}P_n.$}$$

Since the inequalities $\alpha^{s-2}\leq P_s \leq \alpha^{s-1}$ hold for all $s\geq 2$, we get that, 

 $$ \alpha^{2m-4}\leq P_m^2 < \frac{1}{4}P_n \leq \frac{1}{4}\alpha^{n-1} $$

therefore, 

$$2m< 3 + \frac{\log(1/4)}{\log \alpha} + n. $$

Using Lemma \ref{lem1:2}, we have that 

$$m > n- 1.117 - \frac{\log\log \log n}{\log \alpha}. $$
 
Combining these inequalities, we get 

$$n<5.234 + \frac{\log(1/4)}{\log \alpha} + \frac{2\log\log \log n}{\log \alpha} < 3.67 + \frac{2\log\log \log n}{\log \alpha} $$
 which is not possible in our range of $n$. Hence,  $V>0$. Since also $W>0$, we get that 
 
$$1- \prod_{j=i+1}^{k} \(1-\frac{1}{q_i}\) > V. $$

Now, note that

$$((q_1-1)\cdots (q_i-1)\alpha^{n-m} - q_1\cdots q_i )((q_1-1)\cdots (q_i-1)\beta^{n-m} - q_1\cdots q_i )$$
is non zero integer since $\beta$ and  $\alpha$ are conjugate, therefore, 
$$\Big|((q_1-1)\cdots (q_i-1)\alpha^{n-m} - q_1\cdots q_i )((q_1-1)\cdots (q_i-1)\beta^{n-m} - q_1\cdots q_i )\Big|\geq 1.$$
Since we certainly have 
$$\Big|((q_1-1)\cdots (q_i-1)\beta^{n-m} - q_1\cdots q_i )\Big|<2q_1\cdots q_i$$ 
and $$((q_1-1)\cdots (q_i-1)\alpha^{n-m} - q_1\cdots q_i )>0$$ because $V>0$, we get that

 $$ ((q_1-1)\cdots (q_i-1)\alpha^{n-m} - q_1\cdots q_i )>\frac{1}{2q_1\cdots q_i}.$$
 
 Hence,
 
 $$1- \prod_{j=i+1}^{k} \(1-\frac{1}{q_i}\)>X>\frac{\alpha^m}{2(q_1\cdots q_i)^2(\alpha^n-\beta^n)}>\frac{\alpha^m-\beta^m}{2u_i^2(\alpha^n-\beta^n)}=\frac{P_m}{2u_i^2P_n}, $$
which combined with (\ref{eq1:7}) lead to

 $$ \frac{P_m}{2u_i^2P_n} <1- \prod_{j=i+1}^{k} \(1-\frac{1}{q_i}\) \leq \sum_{i=1}^{k} \frac{1}{q_i} < \frac{k}{q_{i+1}}.$$

Thus, $$q_{i+1}<2u_i^2k\frac{P_n}{P_m} .$$

However, $$ \frac{P_n}{P_m}<\alpha^{n-m+1} < \alpha^{2.117} \log \log n.$$

By Lemma \ref{lem1:2}. Hence,

$$q_{i+1} < (2\alpha^{2.117}k\log\log n)u_i^2.$$

Multiplying both sides by  $u_i$, we get that, 

$$u_{i+1}<(2\alpha^{2.117}k\log\log n)u_i^3.$$

Using the assumption hypothesis, we get that,

$$u_{i+1}<(2\alpha^{2.117}k\log\log n)^{1+3(3^i-1)/2)}= (2\alpha^{2.117}k\log\log n)^{(3^{i+1}-1)/2},$$

This ends the proof by induction of the estimate (\ref{eq1:9}). 
\medskip

\medskip


In particular,
$$
q_1\cdots q_k = u_k <(2\alpha^{2.1}k\log\log n)^{(3^k - 1)/2},
$$
which together with formulae \eqref{eq1:5} and \eqref{eq1:6} gives
$$
P_n=q_1\cdots q_k B<(2\alpha^{2.1}k\log\log n)^{1+(3^k-1)/2}= (2\alpha^{2.1}k\log\log n)^{(3^k+1)/2}.
$$
Since $P_n > \alpha^{n-2}$ by inequality \eqref{eq1:sizePn}, we have that
$$
(n-2)\log\alpha < \frac{(3^k+1)}{2}\log(2\alpha^{2.1}k\log\log n).
$$
Since $k<\log n/\log 2$ (see \eqref{eq1:k}), we get
$$
3^k>(n-2)\left(\frac{2\log\alpha}{\log(2\alpha^{2.1}(\log n)(\log\log n) (\log 2)^{-1})}\right)-1>0.17(n-2)-1>\frac{n}{6},
$$
where the last two inequalities above hold because $n>2^{416}$.

So, we proved the following result.

\begin{lemma}
\label{lem1:3}
We have
$$
3^k>n/6.
$$
\end{lemma}

\subsection*{The case when $n$ is odd}
\label{Subsec:odd}

Let $q$ be any prime factor of $P_n$. Reducing relation
\begin{equation}
\label{eq1:nodd}
Q_n^2 - 8P_n^2 = 4 (-1)^n
\end{equation}
of Lemma \ref{lem1:PQ} modulo $q$, we get $Q_n^2\equiv -4\pmod q$. Since $q$ is odd, (because $n$ is odd), then  $q\equiv 1\pmod 4$. This is satisfied by all prime factors
$q$ of $P_n$. Hence,
$$
4^k\mid \prod_{i=1}^k (q_i-1)\mid \varphi(P_n)\mid P_m,
$$
which, by Lemma \ref{lem1:orderof2} (ii), gives $4^k\mid m$. Thus,
$$
n>m\geq 4^{k},
$$
inequality which together with Lemma \ref{lem1:3} gives
$$
n>\left(3^k\right)^{\log 4/\log 3}>\left(\frac{n}{6}\right)^{\log 4/\log 3},
$$
so
$$
n<6^{\log 4/\log(4/3)}<5621,
$$
in contradiction with Lemma \ref{lem1:1}.

\subsection*{\small Bounding $n$}

From now on, $n$ is even. We write it as
$$
n=2^s r_1^{\lambda_1}\cdots r_t^{\lambda_t}=:2^s n_1,
$$
where $s\ge 1$, $t\ge 0$ and $3\le r_1<\cdots<r_t$ are odd primes. Thus, by inequality \eqref{eq1:10tominus40}, we have
$$
\alpha^{\ell}\left(1-\frac{1}{10^{40}}\right)<\alpha^{\ell}-\frac{1}{10^{40}}<\frac{P_n}{\varphi(P_n)}=\prod_{p\mid P_n} \left(1+\frac{1}{p-1}\right)=
2 \prod_{\substack{d\ge 3\\ d\mid n}} \prod_{p\in {\mathcal P}_d} \left(1+\frac{1}{p-1}\right),
$$
and taking logarithms we get
\begin{eqnarray}
\label{eq1:veryuseful}
\ell\log \alpha-\frac{1}{10^{39}} & < & \log\left(\alpha^{\ell} \left(1-\frac{1}{10^{40}}\right)\right)\nonumber\\
& < & \log 2+\sum_{\substack{d\ge 3\\ d\mid n}} \sum_{p\in {\mathcal P}_d} \log\left(1+\frac{1}{p-1}\right)\nonumber\\
& < & \log 2+\sum_{\substack{d\ge 3\\ d\mid n}} S_d.
\end{eqnarray}
In the above, we used the inequality $\log(1-x)>-10 x$ which is valid for all $x\in (0,1/2)$ with $x=1/10^{40}$ and that $x\geq\log(1+x)$ for $x\in \mathbb{R}$ with $x=p$ for all $p\in {\mathcal P}_d$ and all divisors $d\mid n$ with $d\ge 3$.

Let us deduce that the case $t=0$ is impossible. Indeed, if this were so, then $n$ is a power of $2$ and so, by Lemma \ref{lem1:1}, both $m$ and $n$ are divisible by $2^{416}$. Thus, $\ell\ge 2^{416}$. Inserting this into \eqref{eq1:veryuseful}, and using Lemma \ref{lem1:sumSn}, we get
$$
2^{416}\log \alpha-\frac{1}{10^{39}}<\sum_{a\ge 1} \frac{2\log(2^a)}{2^a}=4\log 2,
$$
a contradiction. 

Thus, $t\ge 1$ so $n_1>1$. We now put
$$
{\mathcal I}:=\{i: r_i\mid m\}\quad {\text{\rm and}}\quad {\mathcal J}=\{1,\ldots,t\}\backslash {\mathcal I}.
$$
We put
$$
M=\prod_{i\in {\mathcal I}} r_i.
$$
We also let $j$ be minimal in ${\mathcal J}$. We split the sum appearing in \eqref{eq1:veryuseful} in two parts:
$$
\sum_{d\mid n} S_d=L_1+L_2,
$$
where
$$
L_1:=\sum_{\substack{d\mid n\\ r\mid d\Rightarrow r\mid 2M}} S_d\quad {\text{\rm and}}\quad L_2:=\sum_{\substack{d\mid n\\ r_u\mid d~{\text{\rm for~some}}~ u\in {\mathcal J}}} S_d.
$$
To bound $L_1$, we note that all divisors involved divide $n'$, where
$$
n'=2^s\prod_{i\in {\mathcal I}}^j r_i^{\lambda_i}.
$$
Using Lemmas \ref{lem1:useful} and \ref{lem1:sumSn}, we get
\begin{eqnarray}
\label{eq1:S1}
L_1 & \le & 2\sum_{d\mid n'} \frac{\log d}{d}\nonumber\\
& < & 2\left(\sum_{r\mid n'} \frac{\log r}{r-1}\right)\left(\frac{n'}{\varphi(n')}\right)\nonumber\\
& = & 2\left(\sum_{r\mid 2M} \frac{\log r}{r-1}\right) \left(\frac{2M}{\varphi(2M)}\right).
\end{eqnarray}
We now bound $L_2$. If ${\mathcal J}=\emptyset$, then $L_2=0$ and there is nothing to bound. So, assume that ${\mathcal J}\ne \emptyset$. We argue as follows. Note that since $s\ge 1$, by Theorem \ref{eq1:binet1}, we have
$$
P_n=P_{n_1} Q_{n_1} Q_{2n_1}\cdots Q_{2^{s-1} n_1}.
$$
Let $q$ be any odd prime factor of $Q_{n_1}$. By reducing the relation of Lemma \ref{lem1:PQ} modulo $q$ and using the fact that $n_1$ and $q$ are both odd, we get
$2P_{n_1}^2\equiv 1\pmod q$, therefore ${\displaystyle{\left(\frac{2}{q}\right)=1}}$. Hence, $z(q)\mid q-1$ for such primes $q$. Now let $d$ be any divisor of $n_1$ which is a multiple of $r_{j}$. The number of them is $\tau(n_1/r_{j})$. For each such $d$, there is a primitive prime factor $q_d$ of $Q_d\mid Q_{n_1}$. Thus, $r_{j}\mid d\mid q_d-1$. This shows that
\begin{equation}
\label{eq1:x}
\nu_{r_{j}}(\varphi(P_n))\ge \nu_{r_{j}} (\varphi(Q_{n_1}))\ge \tau(n_1/r_{j})\ge \tau(n_1)/2,
\end{equation}
where the last inequality follows from the fact that
$$
\frac{\tau(n_1/r_{j})}{\tau(n_1)}=\frac{\lambda_{j}}{\lambda_{j}+1}\ge \frac{1}{2}.
$$
Since $r_{j}$ does not divide $m$, it follows from \eqref{eq1:nupvaluationofPn} that
\begin{equation}
\label{eq1:xx}
\nu_{r_{j}} (P_m)\le e_{r_{j}}.
\end{equation}
Hence, \eqref{eq1:x}, \eqref{eq1:xx} and \eqref{eq1:pb} imply that
\begin{equation}
\label{eq1:taun1}
\tau(n_1)\le 2e_{r_j}.
\end{equation}
Invoking Lemma \ref{lem1:zofp}, we get
\begin{equation}
\label{eq1:tau}
\tau(n_1)\le \frac{(r_{j}+1)\log \alpha}{\log r_{j}}.
\end{equation}
Now every divisor $d$ participating in $L_2$ is of the form $d=2^a d_1$, where $0\le a\le s$ and $d_1$ is a divisor of $n_1$ divisible by $r_u$ for some $u\in {\mathcal J}$. Thus,
\begin{equation}
\label{eq1:tt}
L_2\le \tau(n_1) \min\left\{ \sum_{\substack{0\le a\le s\\  d_1\mid n_1\\ r_u\mid d_1~{\text{\rm for~some}}~ u\in {\mathcal J}}} S_{2^ ad_1}\right\}:=g(n_1,s,r_1).
\end{equation}
In particular, $d_1\ge 3$ and since the function $x\mapsto \log x/x$ is decreasing for $x\ge 3$, we have that
\begin{equation}
\label{eq1:S2}
g(n_1,s,r_1)\le 2\tau(n_1) \sum_{0\le a\le s} \frac{\log(2^a r_{j})}{2^a r_{j}}.
\end{equation}
Putting also $s_1:=\min\{s,416\}$, we get, by Lemma \ref{lem1:1}, that $2^{s_1}\mid \ell$. Thus, inserting this as well as \eqref{eq1:S1} and \eqref{eq1:S2} all into \eqref{eq1:veryuseful}, we get
\begin{equation}
\label{eq1:main2}
\ell \log \alpha-\frac{1}{10^{39}}<2\left(\sum_{r\mid 2M} \frac{\log r}{r-1}\right) \left(\frac{2M}{\varphi(2M)}\right)+g(n_1,s,r_1).
\end{equation}
Since
\begin{equation}
\label{eq1:3***}
\sum_{0\le a\le s}
\frac{\log(2^a r_{j})}{2^a r_{j}}<\frac{4\log 2+2\log r_{j}}{r_{j}},
\end{equation}
inequalities \eqref{eq1:3***}, \eqref{eq1:tau} and \eqref{eq1:S2} give us that
$$
g(n_1,s,r_1)\le 2\left(1+\frac{1}{r_j}\right)\left(2+\frac{4\log 2}{\log r_j}\right)\log \alpha:=g(r_j).
$$
The function $g(x)$ is decreasing for $x\ge 3$.  Thus, $g(r_j)\le g(3)< 10.64$.
For a positive integer $N$ put
\begin{equation}
\label{eq1:f(M)}
f(N):=N\log \alpha-\frac{1}{10^{39}}- 2\left(\sum_{r\mid N} \frac{\log r}{r-1}\right) \left(\frac{N}{\varphi(N)}\right).
\end{equation}
Then inequality \eqref{eq1:main2} implies that both inequalities
\begin{eqnarray}
\label{eq1:200}
f(\ell)<g(r_j),\nonumber\\
(\ell-M)\log \alpha+f(M)<g(r_j)
\end{eqnarray}
hold. Assuming that $\ell\ge 26$, we get, by Lemma \ref{lem1:RS}, that
\begin{eqnarray*}
\ell \log \alpha-\frac{1}{10^{39}}-2(\log 2) \frac{(1.79 \log\log \ell+2.5/\log\log \ell)\log \ell}{\log\log \ell-1.1714}\le 10.64.
\end{eqnarray*}
Mathematica confirmed that the above inequality implies $\ell\le 500$. Another calculation with Mathematica
showed that the inequality
\begin{equation}
\label{eq1:111}
f(\ell)<10.64
\end{equation}
for even values of $\ell\in [1,500]\cap {\mathbb Z}$ implies that $\ell\in [2,18]$. The minimum of the function
$f(2N)$ for $N\in [1,250]\cap {\mathbb Z}$ is at $N=3$ and $f(6)>-2.12$. For the remaining positive integers $N$, we have $f(2N)>0$. Hence, inequality \eqref{eq1:200} implies
$$
(2^{s_1}-2)\log \alpha<10.64\quad {\text{\rm and}}\quad (2^{s_1}-2) 3\log \alpha<10.64+2.12=12.76,
$$
according to whether $M\ne 3$ or $M=3$, and either one of the above inequalities implies that $s_1\le 3$. Thus, $s=s_1\in \{1,2,3\}$.
Since $2M\mid \ell$, $2M$ is square free and $\ell\le 18$, we have that $M\in \{1,3,5,7\}$. Assume $M>1$ and let $i$ be such that $M=r_i$. Let us show that $\lambda_i=1$. Indeed, if $\lambda_i\ge 2$, then $$
199\mid Q_9\mid P_n,\quad 29201\mid P_{25}\mid P_n,\quad 1471\mid Q_{49}\mid P_n,
$$
and $3^2\mid 199-1,~5^2\mid 29201-1,~ 7^2\mid 1471-1$. Thus, we get that
$3^2,~ 5^2,~7^2$ divide $\varphi(P_n)=P_m$, showing that $3^2,~5^2,~7^2$ divide $\ell$. Since $\ell\le 18$, only the case $\ell=18$ is possible.  In this case, $r_j\ge 5$, and inequality \eqref{eq1:200} gives
$$
8.4<f(18)\le g(5)<7.9,
$$
a contradiction. Let us record what we have deduced so far.

\begin{lemma}
\label{lem1:4}
If $n>2$ is even, then $s\in \{1,2,3\}$. Further, if ${\mathcal I}\ne\emptyset$, then ${\mathcal I}=\{i\}$, $r_i\in \{3,5,7\}$ and $\lambda_i=1$.
\end{lemma}

We now deal with ${\mathcal J}$. For this, we return to \eqref{eq1:veryuseful} and use the better inequality namely
$$
2^s M\log \alpha-\frac{1}{10^{39}}\le \ell \log \alpha-\frac{1}{10^{39}}\le
\log \left(\frac{P_n}{\varphi(P_n)}\right)\le \sum_{d\mid 2^s M} \sum_{p\in {\mathcal P}_d} \log\left(1+\frac{1}{p-1}\right)+L_2,
$$
so
\begin{equation}
\label{eq1:300}
L_2\ge 2^s M\log \alpha-\frac{1}{10^{39}}-\sum_{d\mid 2^s M} \sum_{p\in {\mathcal P}_d} \log\left(1+\frac{1}{p-1}\right).
\end{equation}
In the right--hand side above, $M\in \{1,3,5,7\}$ and $s\in \{1,2,3\}$. The values of the right--hand side above are in fact
$$
h(u):=u\log \alpha-\frac{1}{10^{39}} -\log(P_u/\varphi(P_u))
$$
for $u=2^sM\in \{2,4,6,8,10,12,14,20,24,28,40,56\}$. Computing we get:
$$
h(u)\ge H_{s,M} \left(\frac{M}{\varphi(M)}\right)\quad {\text{\rm for}}\quad M\in \{1,3,5,7\},\quad s\in \{1,2,3\}, 
$$
where
$$
H_{1,1}>1.069,\quad H_{1,M}>2.81\quad {\text{\rm for}}\quad M>1,\quad H_{2,M}>2.426,\quad H_{3,M}>5.8917.
$$
We now exploit the relation
\begin{equation}
\label{eq1:L2}
H_{s,M} \left(\frac{M}{\varphi(M)}\right)<L_2.
\end{equation}
Our goal is to prove that $r_1<10^6$. Assume this is not so. We use the bound
$$
L_2<\sum_{\substack{d\mid n\\ r_u\mid d~{\text{\rm for~sume}}~u\in {\mathcal J}}} \frac{4+4\log\log d}{\varphi(d)}
$$
of Lemma \ref{lem1:sumSn}. Each divisor $d$ participating in $L_2$ is of the form $2^a d_1$, where $a\in [0,s]\cap {\mathbb Z}$ and $d_1$ is a multiple of a prime at least as large as $r_j$. Thus,
$$
\frac{4+4\log\log d}{\varphi(d)}\le \frac{4+4\log\log 8d_1}{\varphi(2^a) \varphi(d_1)}\quad {\text{\rm for}}\quad a\in \{0,1,\ldots,s\},
$$
and
$$
\frac{d_1}{\varphi(d_1)}\le \frac{n_1}{\varphi(n_1)}\le \frac{M}{\varphi(M)}\left(1+\frac{1}{r_j-1}\right)^{\omega(n_1)}.
$$
Using \eqref{eq1:tau}, we get
$$
2^{\omega(n_1)}\le \tau(n_1)\le \frac{(r_j+1)\log \alpha}{\log r_j}<r_j,
$$
where the last inequality holds because $r_j$ is large. Thus,
\begin{equation}
\label{eq1:omega}
\omega(n_1)<\frac{\log r_j}{\log 2}<2\log r_j.
\end{equation}
Hence,
\begin{eqnarray}
\label{eq1:totient}
\frac{n_1}{\varphi(n_1)}&\le& \frac{M}{\varphi(M)}\left(1+\frac{1}{r_j-1}\right)^{\omega(n_1)}\nonumber\\
&<& \frac{M}{\varphi(M)}\left(1+\frac{1}{r_j-1}\right)^{2\log r_j}\nonumber\\
&<&\frac{M}{\varphi(M)}\exp\left(\frac{2\log r_j}{r_j-1}\right)<\frac{M}{\varphi(M)}\left(1+\frac{4\log r_j}{r_j-1}\right),
\end{eqnarray}
where we used the inequalities  $1+x<e^x$, valid for all real numbers $x$, as well as $e^x<1+2x$ which is valid for $x\in (0,1/2)$ with $x=2\log r_j/(r_j-1)$ which belongs to $(0,1/2)$ because $r_j$ is large.  Thus, the inequality
$$
\frac{4+4\log\log d}{\varphi(d)}\le \left(\frac{4+4\log\log 8d_1}{d_1}\right)\left(1+
\frac{4\log r_j}{r_j-1}\right) \left(\frac{1}{\varphi(2^a)}\right) \frac{M}{\varphi(M)}
$$
holds for $d=2^a d_1$ participating in $L_2$. The function $x\mapsto (4+4\log\log(8x))/x$ is decreasing for $x\ge 3$. Hence,
\begin{equation}
\label{eq1:imp11}
L_2\le \left(\frac{4+4\log\log(8 r_j)}{r_j}\right) \tau(n_1) \left(1+\frac{4\log r_j}{r_j-1}\right)\left(\sum_{0\le a\le s} \frac{1}{\varphi(2^a)}\right) \left(\frac{M}{\varphi(M)}\right).
\end{equation}
Inserting inequality \eqref{eq1:tau} into \eqref{eq1:imp11} and using \eqref{eq1:L2}, we get
\begin{equation}
\label{eq1:r1}
\log r_j<4\left(1+\frac{1}{r_j}\right)\left(1+\frac{4\log r_j}{r_j-1}\right)(1+\log\log (8 r_j))(\log\alpha)\left( \frac{G_s}{H_{s,M}}\right),
\end{equation}
where
$$
G_s=\sum_{0\le a\le s} \frac{1}{\varphi(2^a)}.
$$
For $s=2,~3$, inequality \eqref{eq1:r1} implies $r_j<900,000$ and $r_j<300$, respectively. For $s=1$ and $M>1$, inequality \eqref{eq1:r1}  implies $r_j<5000$. When $M=1$ and $s=1$, we get $n=2n_1$. Here, inequality \eqref{eq1:r1} implies that $r_1<8\times 10^{12}$. This is too big,  so we use the bound
$$
S_d<\frac{2\log d}{d}
$$
of Lemma \ref{lem1:sumSn} instead for the divisors $d$ of participating in $L_2$, which in this case are all the divisors of $n$ larger than $2$. We deduce that
$$
1.06<L_2<2\sum_{\substack{d\mid 2n_1\\ d>2}} \frac{\log d}{d}<4\sum_{d_1\mid n_1} \frac{\log d_1}{d_1}.
$$
Since all the divisors $d>2$ of $n$ are either of the form $d_1$ or $2d_1$ for some divisor $d_1\ge 3$ of $n_1$, and the function $x\mapsto x/\log x$  is increasing for $x\ge 3$, hence the last inequality above follows immediately.
Using Lemma \ref{lem1:useful} and inequalities \eqref{eq1:omega} and \eqref{eq1:totient}, we get
\begin{eqnarray*}
1.06 & < & 4\left(\sum_{r\mid n_1} \frac{\log r}{r-1}\right) \left(\frac{n_1}{\varphi(n_1)}\right)<\left(\frac{4\log r_1}{r_1-1}\right) \omega(n_1) \left(1+\frac{4\log r_1}{r_1-1}\right)\\
& < & \left(\frac{4\log r_1}{r_1-1}\right)\left(2\log r_1\right)\left(1+\frac{4\log r_1}{r_1-1}\right),
\end{eqnarray*}
which gives $r_1<159$. So, in all cases, $r_j<10^6$. Here, we checked that $e_{r}=1$ for all such $r$ except $r\in \{13,31\}$ for which $e_{r}=2$. If $e_{r_j}=1$, we then get $\tau(n_1/r_j)\le 1$, so $n_1=r_j$. Thus, $n\le 8\cdot 10^6$, in contradiction with Lemma \ref{lem1:1}. Assume now that $r_j\in \{13,31\}$. Say $r_j=13$. In this case, $79$ and $599$ divide $Q_{13}$ which divides $P_n$, therefore $13^2\mid (79-1)(599-1)\mid \varphi(P_n)=P_m$. Thus, if there is some other prime factor $r'$ of $n_1/13$, then $13r'\mid n_1$, and there is a primitive prime $q$ of $Q_{13r'}$ such that $q\equiv 1\pmod {13 r'}$. In particular, $13\mid q-1$. Thus, $\nu_{13}(\varphi(P_n))\ge 3$, showing that $13^3\mid P_m$. Hence, $13\mid m$, therefore $13\mid M$, a contradiction. A similar contradiction is obtained if $r_j=31$ since $Q_{31}$ has two primitive prime factors namely $424577$ and $865087$ so $31\mid M$. This finishes the proof.


\chapter{On The Equation $\varphi(X^m-1)=X^n-1$}
In this chapter, we study all positive integer solutions $(m,n)$ of the Diophantine equation of the form  $\varphi(X^m-1)=X^n-1$. In Section \textcolor{red}{\ref{sec13}}, we first present the proof of the case when $X=5$. In Section \textcolor{red}{\ref{sec:1}}, we give the complete proof of the equation from the title, which is the main result in \cite{BL}.

\section{Introduction}
\label{sec111}
As we mentioned in the Introduction, Problem $10626$ from the {\it American Mathematical Monthly} \cite{A} asks to find all positive integer solutions $(m,n)$ of the Diophantine equation
\begin{equation}
\label{eq11:1}
\varphi(5^m-1)=5^n-1.
\end{equation}

In \cite{Lu4}, it was shown that if $b\ge 2$ is a fixed integer, 
then the equation
\begin{equation}
\label{eq11:b}
\varphi\left(x\frac{b^m-1}{b-1}\right)=y\frac{b^n-1}{b-1}\qquad x,y\in \{1,\ldots,b-1\}
\end{equation}
has only finitely many positive integer solutions $(x,y,m,n)$. That is, there are only finitely many repdigits in base $b$ whose Euler function is also a repdigit in base $b$. Taking $b=5$, it follows that equation \eqref{eq11:1} should have only finitely many positive integer solutions $(m,n)$. 

The main objective is to bound the value of $k=m-n$. Firstly, we make explicit the arguments from \cite{Lu4} together with some specific features which we deduce from the factorizations of $X^k-1$ for small values of $k$. 
\medskip

In the process of bounding $k$, we use two analytic inequalities i.e., inequality \eqref{eq11:11} and the approximation $(iii)$ of Lemma \ref{lem1:RS}.




\section{On the Equation $\varphi(5^m-1)=5^n-1$}
\label{sec13}

In this section we show that the equation of the form  $\varphi(5^m-1)=5^n-1$ has no positive integer solutions $(m,n)$, where $\varphi$ is the Euler function.  Here we follow \cite{FLT}.
\medskip

\begin{theorem}
\label{thm:1} 
The Diophantine equation 
\begin{equation}
\label{eq:1}
\phi(5^m-1)=5^n-1.
\end{equation}
has no positive integer solution $(m,n)$.
\end{theorem}

\section*{The proof of Theorem \ref{thm:1}}

For the proof, we make explicit the arguments from \cite{Lu4} together with some specific features which we deduce from the factorizations of $5^k-1$ for small values of $k$. Write
\begin{equation}
\label{eq:2}
5^m-1=2^{\alpha} p_1^{\alpha_1}\cdots p_r^{\alpha_r}.
\end{equation}
Thus,
\begin{equation}
\label{eq:3}
\phi(5^m-1)=2^{\alpha-1}p_1^{\alpha_1-1}(p_1-1)\cdots p_r^{\alpha_r-1}(p_r-1).
\end{equation}

We achieve the proof of Theorem \ref{thm:1} as a sequence of lemmas. The first one is known but we give a proof of it for the convenience of the reader.

\begin{lemma}
\label{lem:1}
In equation \eqref{eq:1}, $m$ and $n$ are not coprime.
\end{lemma}

\begin{proof}
Suppose that $\gcd(m,n)=1$. Assume first that $n$ is odd. Then $\nu_2(5^n-1)=2$. Applying the $\nu_2$ function in both sides of \eqref{eq:3} and comparing it with \eqref{eq:1}, we get
$$
(\alpha-1)+r\le (\alpha-1)+\sum_{i=1}^r \nu_2(p_i-1)=2.
$$
If $m$ is even, then $\alpha\ge 3$, and the above inequality shows that $\alpha=3,~r=0$, so $5^m-1=8$, false. Thus, $m$ is odd, so $\alpha=2$ and $r=1$. If $\alpha_1\ge 2$, then 
$$
p_1^{\alpha_1-1}\mid \gcd(\phi(5^m-1), 5^m-1)=\gcd(5^m-1,5^n-1)=5^{\gcd(m,n)}-1=4,
$$
is a contradiction. So, $\alpha_1=1,~5^m-1=4p_1$, and
$$
5^n-1=2(p_1-1)=\frac{5^m-1}{2}-2=\frac{5^m-5}{2},
$$
which is impossible. Thus, $n$ is even and since $\gcd(m,n)=1$, it follows that $m$ is odd so $\alpha=2$. Furthermore, a previous argument shows that in \eqref{eq:2} we have $\alpha_1=\cdots=\alpha_r=1$. 
Since $m$ is odd, we have that $5\cdot (5^{(m-1)/2})^2\equiv 1\pmod {p_i}$, therefore ${\displaystyle{\left(\frac{5}{p_i}\right)=1}}$ for $i=1,\ldots,r$. Hence, $p_i\equiv 1,4\pmod 5$. If $p_i\equiv 1\pmod 5$, it follows that $5\mid \phi(5^m-1)=5^n-1$, a contradiction. Hence, $p_i\equiv 4\pmod 5$ for $i=1,\ldots,r$. Reducing now relation \eqref{eq:2} modulo $5$, we get
$$
4\equiv 4^{1+r}\pmod 5,\quad {\text{\rm therefore}}\quad r\equiv 0\pmod 2.
$$
Reducing now equation
$$
2(p_1-1)\cdots (p_r-1)=\phi(5^m-1)=5^n-1
$$
modulo $5$, we get
$$
2\cdot (3^{r/2})^2\equiv 4\pmod 5, \quad {\text{\rm therefore}}\quad \left(\frac{2}{5}\right)=1,
$$
a contradiction.
\end{proof}

\begin{lemma}
\label{lem:2}
If $(m,n)$ satisfies equation \eqref{eq:1}, then $m$ is not a multiple of any number $d$ such that $p\mid 5^d-1$ for some prime $p\equiv 1\pmod 5$.
\end{lemma}

\begin{proof}
This is clear, because if $m$ is a multiple of a number $d$ such that $p\mid 5^d-1$ for some prime $p\equiv 1\pmod 5$, then $5\mid (p-1)\mid \phi(5^d-1)\mid \phi(5^m-1)=5^n-1$, which is false.
\end{proof}

Since $29423041$ is a prime dividing $5^{32}-1$, it follows that $\nu_2(m)\le 4$. From the Cunningham project tables \cite{Cun}, we deduced that if $q\le 512$ is an odd prime, then $5^q-1$ has a prime factor $p\equiv 1\pmod 5$
except for \newline $q\in  \{17,41,71,103,223,257\}$. So, if $q\mid m$ is odd, then 
\begin{equation}
\label{eq:Q}
q\in {\mathcal Q}:=\{17,41,71,103,223,257\}\cup \{q>512\}.
\end{equation}

\begin{lemma}
\label{lem:3}
The following inequality holds:
\begin{equation}
\label{eq:imp}
\log\left(\log\left(\frac{5^k e}{3.6}\right)\right)<20\sum_{q\mid m} \frac{\log\log q}{q}.
\end{equation}
\end{lemma}

\begin{proof}
Write
$$
m=2^{\alpha_0}\prod_{i=1}^s q_i^{\alpha_i}\qquad q_i~{\text{\rm odd~prime}}~i=1,\ldots,s.
$$
Recall that $\alpha_0\le 4$. None of the values $m=1,2,4,8,16$ satisfies equation \eqref{eq:1} for some $n$, so $s\ge 1$. Put $k=m-n$. Note that $k\ge 2$ because $m$ and $n$ are not coprime by Lemma \ref{lem:1}.  
Then
\begin{equation}
\label{eq:4}
5^{k}<\frac{5^m-1}{5^n-1}=\frac{5^m-1}{\phi(5^m-1)}=\prod_{p\mid 5^m-1} \left(1+\frac{1}{p-1}\right).
\end{equation}
For each prime number $p\ne 5$, we write $z(p)$ for the order of appearance of $p$ in the Lucas sequence of general term $5^n-1$. That is, $z(p)$ is the order of $5$ modulo $p$. Clearly, if $p\mid 5^m-1$, then $z(p)=d$ for some divisor $d$ of $m$. Thus, we can rewrite inequality \eqref{eq:4} as
\begin{equation}
\label{eq:5}
5^k<\prod_{d\mid m} \prod_{z(p)=d} \left(1+\frac{1}{p-1}\right).
\end{equation}
If $p\mid m$ and $z(p)$ is a power of $2$, then $z(p)\mid 16$, therefore $p\mid 5^{16}-1$. Hence, 
$$
p\in {\mathcal P}=\{2,3,13,17,313, 11489\}.
$$
Thus,
\begin{equation}
\label{eq:6}
\prod_{z(p)\mid 16} \left(1+\frac{1}{p-1}\right)\le \prod_{p\in P} \left(1+\frac{1}{p-1}\right)<3.5.
\end{equation}
Inserting \eqref{eq:6} into \eqref{eq:5}, we get
\begin{equation}
\label{eq:7}
\frac{5^k}{3.5}<\prod_{\substack{d\mid m\\ P(d)>2}} \prod_{z(p)=d} \left(1+\frac{1}{p-1}\right).
\end{equation}
We take logarithms in inequality \eqref{eq:7} above and use the inequality $\log(1+x)<x$ valid for all real numbers $x$ to get
$$
\log\left(\frac{5^k}{3.5}\right)<\sum_{\substack{d\mid m\\ P(d)>2}} \sum_{z(p)=d} \frac{1}{p-1}.
$$
If $z(p)=d$, then $p\equiv 1\pmod d$. If $P(d)>2$, then since $d\mid m$, we get that every odd prime factor of $d$ is in ${\mathcal Q}$. In particular, it is at least $17$. Thus, 
$p>34$. Hence,
$$
\log\left(\frac{5^k}{3.5}\right)<\sum_{\substack{d\mid m\\ P(d)>2}} \sum_{z(p)=d} \frac{1}{p}+\sum_{p\ge 37} \frac{1}{p(p-1)}<\sum_{\substack{d\mid m\\ P(d)>2}} \sum_{z(p)=d} \frac{1}{p}+0.007.
$$
We thus get that
\begin{equation}
\label{eq:8}
\log\left(\frac{5^k}{3.6}\right)<\sum_{\substack{d\mid m\\ P(d)>2}} S_d,
\end{equation}
where
\begin{equation}
\label{eq:9}
S_d:=\sum_{z(p)=d} \frac{1}{p}.
\end{equation}
We need to bound $S_d$. For this, we first take
$$
{\mathcal P}_d=\{p: z(p)=d\}.
$$
Put $\omega_d:=\#{\mathcal P}_d$. Since $p\equiv 1\pmod d$ for all $p\in {\mathcal P}_d$, we  have that 
\begin{equation}
\label{eq:10}
(d+1)^{\omega_d}\le \prod_{p\in {\mathcal P}_d} p<5^d-1<5^d,\quad {\text{\rm therefore}}\quad \omega_d<\frac{d\log 5}{\log(d+1)}.
\end{equation}
We now use inequality \eqref{eq11:11}. Put ${\mathcal Q}_d:=\{p<4d: p\equiv 1\pmod d\}$. Clearly, ${\mathcal Q}_d\subset \{d+1,2d+1,3d+1\}$ and since $d\mid m$ and $3\not\in {\mathcal Q}$, it follows that $d$ is not a multiple of $3$. In particular, one of $d+1$ and $2d+1$ is a multiple of $3$, so that at most one of these two numbers can be a prime. 
We now split $S_d$ as follows:
\begin{equation}
\label{eq:12}
S_d \le  \sum_{\substack{p<4d\\ p\equiv 1\pmod d}} \frac{1}{p}+\sum_{\substack{4d\le p\le d^2\\ p\equiv 1\pmod d}} \frac{1}{p}+\sum_{\substack{p>d^2\\ z(p)=d}} \frac{1}{p}:=T_1+T_2+T_3.
\end{equation}
Clearly,
\begin{equation}
\label{eq:S1}
T_1=\sum_{p\in {\mathcal Q}_d}\frac{1}{p}.
\end{equation}
For $S_2$, we use estimate \eqref{eq11:11} and Abel's summation formula to get
\begin{eqnarray*}
T_2 & \le & \frac{\pi(x;d,1)}{x}\Big|_{x=4d}^{d^2}+\int_{4d}^{d^2} \frac{\pi(t;d,1)}{t^2} dt\\
& \le & \frac{2d^2}{d^2\phi(d)\log d}+\frac{2}{\phi(d)} \int_{4d}^{d^2} \frac{dt}{t\log(t/d)}\\
& \le & \frac{2}{\phi(d)\log d}+\frac{2}{\phi(d)} \log\log(t/d)\Big|_{t=4d}^{d^2}\\
& = & \frac{2\log\log d}{\phi(d)}+\frac{2}{\phi(d)} \left(\frac{1}{\log d}-\log\log 4\right).
\end{eqnarray*}
The expression $1/\log d-\log\log 4$ is negative for $d\ge 34$, so 
\begin{equation}
\label{eq:S2}
T_2<\frac{2\log\log d}{\phi(d)}\quad {\text{\rm for~all}}\quad d\ge 34.
\end{equation}
Inequality \eqref{eq:S2} holds for $d=17$ as well, since there
$$
T_2<S_{17}=\frac{1}{409}+\frac{1}{466344409}<0.003<0.13<\frac{2\log\log 17}{\phi(17)}.
$$
Hence, inequality \eqref{eq:S2} holds for all divisors $d$ of $m$ with $P(d)>2$. 

As for $T_3$, we have by \eqref{eq:10}, 
\begin{equation}
\label{eq:S3}
T_3<\frac{\omega_d}{d^2}<\frac{\log 5}{d\log(d+1)}.
\end{equation}
Hence, collecting \eqref{eq:S1}, \eqref{eq:S2} and \eqref{eq:S3}, we obtain
\begin{equation}
\label{eq:13}
S_d<\sum_{p\in {\mathcal Q}_d} \frac{1}{p}+\frac{2\log\log d}{\phi(d)}+\frac{\log 5}{d\log(d+1)}.
\end{equation}
We now show that
\begin{equation}
\label{eq:14}
S_d<\frac{3\log\log d}{\phi(d)}.
\end{equation}
Since $\phi(d)<d$ and at most one of $d+1$ and $2d+1$ is prime, we get, via \eqref{eq:13}, that
\begin{eqnarray*}
S_d & < & \frac{1}{d+1}+\frac{1}{3d+1}+\frac{2\log\log d}{\phi(d)}+\frac{\log 5}{d\log(d+1)}\\
& < & \frac{1}{\phi(d)}\left(\frac{4}{3}+2\log\log d+\frac{\log 5}{\log(d+1)}\right).
\end{eqnarray*}
So, in order to prove \eqref{eq:14}, it suffices that
$$
\frac{4}{3}+\frac{\log 5}{\log(d+1)}<2\log\log d,\quad {\text{\rm which~holds~for~all}}\quad d>200.
$$
The only possible divisors $d$ of $m$ with $P(d)>2$ (so, whose odd prime factors are in ${\mathcal Q}$), and with $d\le 200$ are 
\begin{equation}
\label{eq:15}
R:=\{17,34,41,68,71,82,103,136,142,164\}.
\end{equation}
We checked individually that for each of the values of $d$ in $R$ given by \eqref{eq:15}, inequality \eqref{eq:14} holds. 

Now we write $d=2^{\alpha_d} d_1$, where $\alpha_d\in \{0,1,2,3,4\}$ and $d_1$ is odd. Since $d_1\ge 17>2^{\alpha_d}$, we have that $d<d_1^2$. Hence, keeping $d_1$ fixed
and summing over $\alpha_d$, we have that
\begin{equation}
\label{eq:16}
\sum_{\alpha_d=0}^4 S_{2^{\alpha_d} d_1}<3\sum_{\alpha_d=0}^{4} \frac{\log(2\log d_1)}{\phi(d_1)}\left(1+1+\frac{1}{2}+\frac{1}{4}+\frac{1}{8}\right)<\frac{8.7\log(2\log d_1)}{\phi(d_1)}.
\end{equation}
Inserting inequalities \eqref{eq:14} and \eqref{eq:16} into \eqref{eq:8}, we get that
\begin{equation}
\label{eq:17}
\log\left(\frac{5^k}{3.6}\right)<\sum_{\substack{d_1\mid m\\ d_1>1\\ d_1~{\text{\rm odd}}}} \frac{8.7\log(2\log d_1)}{\phi(d_1)}.
\end{equation}
The function
$$
a\mapsto 8.7\log(2\log a)
$$
is sub--multiplicative when restricted to the set ${\mathcal A}=\{a\ge 17\}$. That is, the inequality 
$$
8.7\log(2\log(ab))\le 8.7\log(2\log a) \cdot 8.7\log(2\log b)\quad {\text{\rm holds if}}\quad \min\{a,b\}\ge 17.
$$
Indeed, to see why this is true, assume say that $a\le b$. Then $\log ab\le 2\log b$, so it is enough to show that
$$
8.7\log 2+8.7\log(2\log b)\le 8-7\log(2\log a)\cdot8.7\log(2\log b)
$$
which is equivalent to 
$$
8.7\log(2\log b)\left(8.7\log(2\log a)-1\right)>8.7\log 2,
$$
which is clear for $\min\{a,b\}\ge 17$.  It thus follows that
$$
\sum_{\substack{d_1\mid m\\ d_1>1\\ d_1~{\text{\rm odd}}}} \frac{8.7\log(2\log d_1)}{\phi(d_1)}<\prod_{q\mid m}\left(1+\sum_{i\ge 1} \frac{8.7\log(2\log q^i)}{\phi(q^i)}\right)-1.
$$
Inserting the above inequality into \eqref{eq:17}, taking logarithms and using the fact that $\log(1+x)<x$ for all real numbers $x$, we get
\begin{equation}
\label{eq:18}
\log\left(\log\left(\frac{5^k e}{3.6}\right)\right)<\sum_{q\mid m} \sum_{i\ge 1} \frac{8.7\log(2\log q^i)}{\phi(q^i)}.
\end{equation}
Next we show that 
\begin{equation}
\label{eq:19}
\sum_{i\ge 1} \frac{8.7\log(2\log(q^i))}{\phi(q^i)}<\frac{20\log\log q}{q}\quad {\text{\rm for}}\quad q\in {\mathcal Q}.
\end{equation}
We check that it holds for $q=17$. So, from now on, $q\ge 41$. Since
$$
\log(2\log q^i)=\log(2i)+\log\log q<(1+\log i)+\log\log q\le i+\log\log q,
$$
we have that
\begin{eqnarray*}
\sum_{i=1}^{\infty} \frac{\log(2\log(q^i))}{\phi(q^i)} & < & \sum_{i\ge 1} \frac{i}{q^{i-1}(q-1)}+\log\log q \sum_{i\ge 1} \frac{1}{q^{i-1}(q-1)}\\
& = & \frac{q^2}{(q-1)^3}+(\log\log q) \left(\frac{q}{(q-1)^2}\right)\\
& < & (\log\log q)\left(\frac{q^2}{(q-1)^3}+\frac{q}{(q-1)^2}\right)\\
& = & (\log\log q) \left(\frac{2q^2-q}{(q-1)^3}\right)
\end{eqnarray*}
because $\log\log q>1$. Thus, it suffices that 
$$
8.7 \left(\frac{2q^2-q}{(q-1)^3}\right)<\frac{20}{q},\quad {\text{\rm which~holds~for}}\quad q\ge 41.
$$
Hence, \eqref{eq:19} holds, therefore \eqref{eq:18} implies
\begin{equation}
\label{eq:20}
\log\left(\log\left(\frac{5^k e}{3.6}\right)\right)<20\sum_{q\mid m} \frac{\log\log q}{q},
\end{equation}
which is exactly \eqref{eq:imp}. This finishes the proof of the lemma.
\end{proof}

\begin{lemma}
\label{lem:4}
If $q<10^4$ and $q\mid m$, then $q\mid n$. 
\end{lemma}

\begin{proof} This is clear for $q=2$, since then $24\mid 5^2-1\mid 5^m-1$, therefore $8=\phi(24)\mid \phi(5^m-1)=5^n-1$, so $n$ is even. Let now $q$ be odd. Consider the number
\begin{equation}
\label{eq:21}
\frac{5^q-1}{4}=q_1^{\beta_1}\cdots q_l^{\beta_l}.
\end{equation}
Assume that $l\ge 2$. Since $q_i\equiv 1\pmod q$ for $i=1,\ldots,l$, we have that $q^2\mid (q_1-1)\cdots (q_l-1)\mid \phi(5^m-1)=5^n-1$. Since $q\| 5^{q-1}-1$ 
for all odd $q<10^4$, we get that, $q\mid n$, as desired. So, it remains to show that $l\ge 2$ in \eqref{eq:21}. We do this by contradiction. Suppose that $l=1$. Since $q_1\equiv 4\pmod 5$, reducing equation \eqref{eq:21} modulo $5$ we get that
$$
1\equiv 4^{\beta_1}\pmod 5,
$$
so $\beta_1$ is even. Hence,
$$
\frac{5^n-1}{5-1}=\square.
$$
However, the equation
$$
\frac{x^n-1}{x-1}=\square
$$
for integers $x>1$ and $n>2$ has been solved by Ljunggren \cite{Lj} who showed that the only possibilities are $(x,n)=(3,5),~(7,4)$. This contradiction shows that $l\ge 2$ and finishes the proof of this lemma.
\end{proof}

\begin{remark.}
Apart from Ljunggren's result, the above proof was based on the computational fact that if $p<10^4$ is an odd prime, then $p\| 5^{p-1}-1$. In fact, the first prime failing this test is $q=20771$.
\end{remark.}

\begin{lemma}
\label{lem:5}
We have $k=2$.
\end{lemma}

\begin{proof} We split the odd prime factors $p$ of $m$ in two subsets
$$
U=\{q\mid n\}\quad {\text{\rm and}}\quad V=\{q\nmid n\}.
$$
By Lemma \ref{lem:3}, we have 
\begin{equation}
\label{eq:22}
\log\left(\log\left(\frac{5^k e}{3.6}\right)\right)\le 20\left(\sum_{q\in U} \frac{\log\log q}{q}+\sum_{p\in V} \frac{\log\log q}{q}\right):=20(T_1+T_2).
\end{equation}
We first bound $T_2$. By Lemma \ref{lem:4}, if $q\in V$, then $q>10^4$. In particular, $q>512$. Let $t\ge 9$, and put $I_t=[2^t,2^{t+1})\cap V$. Suppose that $r_1,\ldots,r_u$ are all the members of $I_t$. By the Primitive Divisor Theorem,  $5^{dr_u}-1$ has a primitive prime factor for all divisors $d$ of $r_1\cdots r_{u-1}$, and this prime is congruent to $1$ modulo 
$r_u$. Since the number $r_1\cdots r_{u-1}$ has $2^{u-1}$ divisors, we get that 
$$
2^{u-1}\le \nu_{r_u}(\phi(5^m-1))=\nu_{r_u}(5^n-1).
$$
Since $r_u\nmid n$, we get that 
$$
2^{u-1}\le \nu_{r_u}(5^{q_u-1}-1)<\frac{\log 5^{r_u}}{\log r_u}=\frac{r_u \log 5}{\log r_u}<\frac{2^{t+1}\log 5}{(t+1)\log 2}.
$$
The above inequality implies that $u\le t-1$, otherwise for $u\ge t$, we would get that
$$
2^{t-1}\le \frac{2^{t+1}\log 5}{(t+1)\log 2},\qquad {\text{\rm or}}\qquad 4\log 5\ge (t+1)\log 2\ge 10\log 2,
$$
a contradiction. This shows that $\#I_t\le t-1$ for all $t\ge 9$. Hence,
$$
20T_2\le \sum_{t\ge 9} \frac{20(t-1)\log\log 2^t}{2^t}<1.4.
$$
Hence, we get that
\begin{equation}
\label{eq:23}
\log\left(\log\left(\frac{5^k e}{3.6}\right)\right)<20\sum_{\substack{q\mid \gcd(m,k)\\ q>2}} \frac{\log\log q}{q}+1.4.
\end{equation}
We use \eqref{eq:23} to bound $k$ by better and better bounds. We start with 
$$
\log\left(\log\left(\frac{5^ke}{3.6}\right)\right)<20(\log\log k)\left(\sum_{\substack{p\mid k\\ q>2}} \frac{1}{q}\right)+1.4,
$$
which is implied by \eqref{eq:23}. Assume $k\ge 3$. We have
$$
\sum_{p\mid k}\frac{1}{p}<\sum_{d\mid k}\frac{1}{d}=\frac{\sigma(k)}{k}<\frac{k}{\phi(k)}<1.79\log\log k+\frac{2.5}{\log\log k}
$$
where the  last inequality above holds for all $k\ge 3$, except for $k=223092870$ by Lemma \ref{lem1:RS} $(iii)$. We thus get that
$$
\log k<\log(k\log 5+1-\log(3.6))<20(\log\log k)^2+51.4,
$$
which gives $\log k<1008$. Since 
$$
\sum_{17\le q\le  1051} \log q>1008>\log k,
$$
it follows that 
$$
T_1=\sum_{\substack{q\mid \gcd(m,k)\\ p>2}}\frac{\log \log q}{q}<\sum_{17\le q\le  1051} \frac{\log\log q}{q}<0.9.
$$
Hence, 
$$
\log\left(\log\left(\frac{5^k e}{3.6}\right)\right)<20\times 0.9+1.4;\quad {\text{\rm hence}}\quad k<2\times 10^8.
$$
By \eqref{eq:Q}, the first few possible odd prime factors of $n$ are $17,~41,~71,~103$ and $223$. Since 
$$
17\times 41\times 71\times 103\times 223>10^9>k,
$$
it follows that
$$
T_2\le \frac{\log\log 17}{17}+\frac{\log\log 41}{41}+\frac{\log\log 71}{71}+\frac{\log\log 103}{103}<0.13.
$$
Hence,
$$
\log\left(\log\left(\frac{5^k e}{3.6}\right)\right)<20\times 0.13+1.4=4;\quad {\text{\rm hence}}\quad k\le 34.
$$
If follows that $k$ can have at most one odd prime, so  
$$
T_2\le \frac{\log\log 17}{17}<0.07,
$$ 
therefore
$$
\log\left(\log\left(\frac{5^k e}{3.6}\right)\right)<20\times 0.07+1.4=2.8;\quad {\text{\rm hence}}\quad k\le 11.
$$
Thus, in fact $k$ has no odd prime factor, giving that $T_2=0$, so 
$$
\log\left(\log\left(\frac{5^k e}{3.6}\right)\right)<1.4,\quad {\text{\rm therefore}}\quad k\le 2.
$$
Since by Lemma \ref{lem:1}, $m$ and $n$ are not coprime, it follows that in fact $k\ge 2$, so $k=2$.
\end{proof}

\begin{lemma}
\label{lem:6}
We have $k>2$. 
\end{lemma}

\begin{proof}
Let $q_1$ be the smallest prime factor of $m$ which exists for if not $n\mid 16$, which is not possible. Let $q_1,\ldots,q_s$ be all the prime factors of $m$. 
For each divisor $d$ of $q_2\cdots q_{s-1}$, the number $5^{dq_1}-1$ has a primitive divisor which is congruent to $1$ modulo $q_1$. Since there are $2^{s-1}$ divisors 
of $q_2\cdots q_s$, we get that
$$
2^{s-1}\le \nu_{q_1}(\phi(5^m-1))=\nu_{q_1}(5^n-1)
$$
Since $q_1$ does not divide $n$ (otherwise it would divide $k=2$), we get that 
$$
2^{s-1}\le \nu_{q_1}\left(5^{q_1-1}-1\right)<\frac{\log 5^{q_1}}{\log q_1}=\frac{q_1\log 5}{\log q_1}<q_1.
$$
Hence,
$$
s<1+\frac{\log q_1}{q_1}.
$$
Lemmas \ref{lem:3} and \ref{lem:6} now show that
\begin{eqnarray*}
\log\left(\log\left(\frac{5^2 e}{3.6}\right)\right) & < & 20\sum_{\substack{q\mid m\\ q>2}} \frac{\log\log q}{q}< \frac{20 s \log\log q_1}{q_1}\\
& < & 20\left(1+\frac{\log q_1}{\log 2}\right) \frac{\log\log q_1}{q_1}.
\end{eqnarray*}
This gives $q_1<300$, so by Lemma \ref{lem:4}, we have $q_1\mid k$, which finishes the proof of this lemma.
\end{proof}

Obviously, Lemma \ref{lem:5} and \ref{lem:6} contradict each other, which completes the proof of the theorem.


\section{On the Equation $\varphi(X^m-1)=X^n-1$}
\label{sec:1}
In this section, we prove that the  equation of the form $\varphi(X^m-1)=X^n-1$ has only finitely many integers solutions $(m,n)$. Here we follow \cite{BL}.

\begin{theorem}
\label{thm2:1}
Each one of the two equations 
$$
\varphi(X^m-1)=X^n-1\qquad and\qquad \varphi\left(\frac{X^m-1}{X-1}\right)=\frac{X^n-1}{X-1}
$$
has only finitely many positive integer solutions $(X,m,n)$ with the exception $m=n=1$ case in which any positive integer $X$ leads to a solution of the second equation above. Aside from the above mentioned exceptions, all solutions have $X<e^{e^{8000}}$.
\end{theorem}

\subsection*{Proof of Theorem \ref{thm2:1}}

Here we follow the same approach as for the proof of Theorem \ref{thm:1}. Since most of the details are similar, we only sketch the argument.
\medskip

Since $\varphi(N)\le N$ with equality only when $N=1$, it follows easily that $m\ge n$ and equality occurs only when $m=n=1$, case in which only $X=2$  leads to a solution of the first equation while any positive integer $X$ leads to a solution of the second equation. From now on, we assume that $m>n\ge 1$. 
The next lemma gives an upper bound on $\frac{k}{2}\log X$.

\begin{lemma}
\label{lem2:3}
Assume that $X>e^{1000}$. Then the following inequality holds:
\begin{equation}
\label{eq11:imp}
\frac{k}{2}\log X<8\sum_{\substack{d\mid m\\ d\ge 5}} \frac{\log\log d}{\varphi(d)}.
\end{equation}
\end{lemma}

\begin{proof} 
Observe that either one of the two equations leads to
\begin{equation}
\label{eq11:4}
X^{k}<\frac{X^m-1}{X^n-1}\le \frac{X^m-1}{\varphi(X^m-1)}=\prod_{p\mid X^m-1} \left(1+\frac{1}{p-1}\right).
\end{equation}
Following the argument from Lemma \ref{lem:3}, we get that
\begin{equation}
\label{eq11:5}
X^k<\prod_{d\mid m} \prod_{z(p)=d} \left(1+\frac{1}{p-1}\right).
\end{equation}
Note also that $X$ is odd, since if $X$ is even, then both numbers $X^n-1$ and $(X^n-1)/(X-1)=X^{n-1}+\cdots+X+1$ are odd, and the only positive integers $N$ such that $\varphi(N)$ is odd are $N=1,2$. However, none of the equations $X^m-1=1,2$ or $(X^m-1)/(X-1)=1,2$ has any positive integer solutions $X$ and $m>1$. 
In the right--hand side of \eqref{eq11:5}, we separate the cases $d\in \{1,2,3,4\}$. Note that since $X$ is odd, it follows that
\begin{equation}
\label{eq11:d1to4}
\prod_{\substack{z(p)\le 4}} p\le  \frac{1}{8} (X^4-1)(X^2+X+1)<X^6.
\end{equation}
For $x\ge 3$, put
\begin{equation}
\label{eq11:L}
L(x):=1.79 \log\log x+\frac{2.5}{\log\log x}.
\end{equation}
Lemma \ref{lem1:RS}$(iii)$ together with inequality  \eqref{eq11:d1to4} show that 
\begin{equation}
\label{eq11:2}
\prod_{z(p)\le 4} \left(1+\frac{1}{p-1}\right)\le L(X^6).
\end{equation}
We take logarithms in \eqref{eq11:5} use \eqref{eq11:2} as well as the inequality 
$\log(1+x)<x$ valid for all real numbers $x$ to get
$$
k\log X<\log L(X^6)+\sum_{\substack{d\mid m\\ d\ge 5}} \sum_{z(p)=d} \frac{1}{p-1}.
$$
If $z(p)=d$, then $p\equiv 1\pmod d$. Thus, $p\ge 7$ for $d\ge 5$. Hence,
\begin{eqnarray}
\label{eq11:important}
k\log X & < & \log L(X^6)+\sum_{\substack{d\mid m\\ d\ge 5}} \sum_{z(p)=d} \frac{1}{p}+\sum_{p\ge 7} \frac{1}{p(p-1)}\nonumber\\
& < & \log L(X^6)+0.06+\sum_{\substack{d\mid m\\ d\ge 5}} \sum_{z(p)=d} \frac{1}{p}\nonumber\\
& = & \log L(X^6)+0.06+\sum_{\substack{d\mid m\\ d\ge 5}} S_d.
\end{eqnarray}
We now proceed to bound $S_d$. Letting for a divisor $d$ of $m$ the notation ${\mathcal P}_d$ stand for the set of primitive prime factors of $X^m-1$, the argument from the proof of Lemma \ref{lem:3} gives

\begin{equation}
\label{eq11:12}
S_d \le  \sum_{p\in {\mathcal Q}_d} \frac{1}{p}+\sum_{\substack{4d\le p\le d^2\log X\\ p\equiv 1\pmod d\\ z(p)=d}} \frac{1}{p}+\sum_{\substack{p>d^2\log X\\ z(p)=d}} \frac{1}{p}:=T_1+T_2+T_3.
\end{equation}

For $T_2$, we use estimate \eqref{eq11:11} and Abel's summation formula as in the upper bound of inequality \eqref{eq:S2} to get
\begin{eqnarray*}
T_2 & \le & \frac{\pi(x;d,1)}{x}\Big|_{x=4d}^{d^2\log X}+\int_{4d}^{d^2\log X} \frac{\pi(t;d,1)}{t^2} dt\\
& \le & \frac{2d^2}{d^2\varphi(d)\log (d\log X)}+\frac{2}{\varphi(d)} \int_{4d}^{d^2\log X} \frac{dt}{t\log(t/d)}\\
& \le & \frac{2}{\varphi(d)\log (d\log X)}+\frac{2}{\varphi(d)} \log\log(t/d)\Big|_{t=4d}^{d^2\log X}\\
& = & \frac{2\log\log (d\log X)}{\varphi(d)}+\frac{2}{\varphi(d)} \left(\frac{1}{\log (d\log X)}-\log\log 4\right).
\end{eqnarray*}
The expression $1/\log (d\log X)-\log\log 4$ is negative for $d\ge 5$ and $X> e^{1000}$, so 
\begin{equation}
\label{eq11:S2}
T_2<\frac{2\log\log (d\log X)}{\varphi(d)}\quad {\text{\rm for~all}}\quad d\ge 5.
\end{equation}
For $T_3$, we have that
\begin{equation}
\label{eq11:S3}
T_3<\frac{\omega_d}{d^2\log X}<\frac{1}{d\log(d+1)}.
\end{equation}
Hence, collecting \eqref{eq:S1}, \eqref{eq11:S2} and \eqref{eq11:S3}, we get that
\begin{equation}
\label{eq11:13}
S_d<\sum_{p\in {\mathcal Q}_d} \frac{1}{p}+\frac{1}{d\log(d+1)}+\frac{2\log\log(d\log X)}{\varphi(d)}.
\end{equation}
We show that
\begin{equation}
\label{eq11:14}
S_d<\left\{\begin{matrix} \frac{4\log\log d}{\varphi(d)} & {\text{\rm if}} & \log X\le d.\\ 
\frac{4\log\log d}{\varphi(d)}+\frac{2\log\log(\log X)^2}{\varphi(d)} & {\text{\rm if}} & 5\le d<\log X.\\ \end{matrix}\right.
\end{equation}
We first deal with the range when $d\ge \log X$. In this case, by \eqref{eq11:13}, using the fact that $\log X\le d$, so 
$$
\log\log(d\log X)\le \log\log d^2=\log 2+\log\log d,
$$ 
as well as the fact that
\begin{equation}
\label{eq11:100}
\sum_{p\in {\mathcal Q}_d} \frac{1}{p}\le \frac{1}{d+1}+\frac{1}{2d+1}+\frac{1}{3d+1}< \frac{11}{6d},
\end{equation}
we have
$$
S_d  < \frac{\log\log d}{\varphi(d)} \left(\frac{11\varphi(d)}{6d\log\log d}+\frac{\varphi(d)}{d\log(d+1) \log\log d}+
2+\frac{2\log 2}{\log\log d}\right),
$$
and the factor in parenthesis in the right--hand side above is smaller than $4$ since $d\ge \log X>1000$ and $\varphi(d)/d<1$. Assume now that 
$5\le d<\log X$. Then using \eqref{eq11:13}, we get that
$$
S_d\le  \sum_{p\in {\mathcal Q}_d} \frac{1}{p}+\frac{1}{d\log(d+1)}+\frac{2\log\log (\log X)^2}{\varphi(d)}.
$$
It remains to check that the sum of the first two terms in the right above is at most $4(\log\log d)/\varphi(d)$. Using the fact that the first term is at most $11/(6d)$ (see \eqref{eq11:100}), we get that it suffices that
$$
\frac{11}{6d}+\frac{1}{d\log(d+1)}\le \frac{4\log\log d}{\varphi(d)},
$$
which is equivalent to
$$
\frac{11\varphi(d)}{6d}+\frac{\varphi(d)}{d\log(d+1)}\le 4\log\log d.
$$
Since $\varphi(d)/d<1$, for $d\ge 7$ the left-hand side above is smaller than the number $11/6+1/\log 8<2.4$, while the right--hand side is larger than $4\log\log 7>2.6$, so the above inequality holds for $d\ge 7$. Once checks that it also holds for $d=6$ (because $\varphi(6)=2$), while for $d=5$, the only prime in ${\mathcal Q}_5$ is $11$ and 
$$
\frac{1}{11}+\frac{1}{5\log 6}<\frac{4\log\log 5}{\varphi(5)},
$$
so the desired inequality holds for $d=5$ as well. This proves \eqref{eq11:14}.

Inserting \eqref{eq11:14} into \eqref{eq11:important}, we get 
$$
k\log X<\log L(X^6)+0.06+2\log\log (\log X)^2 \sum_{5\le d<\log X} \frac{1}{\varphi(d)}+\sum_{\substack{d\mid m\\ d\ge 5}} \frac{4\log\log d}{\varphi(d)}.
$$
For the first sum above, we use Lemma \ref{lem1:RS}$(iii)$  to conclude that
$$
\frac{d}{\varphi(d)}\le L(d)\le L(\log X),
$$
therefore
\begin{eqnarray*}
\sum_{5\le d<\log X} \frac{1}{\varphi(d)} & < & L(\log X) \sum_{5\le d<\log X} \frac{1}{d}< L(\log X)\int_4^{\log X} \frac{dt}{t}\\
& = & L(\log X)(\log\log X-\log\log 4).
\end{eqnarray*}
Thus, putting
$$
M(x):=\log L(x^6)+0.06+2\log\log (\log x)^2 L(\log x)(\log\log x-\log\log 4),
$$
we get that
\begin{equation}
\label{eq11:20}
k\log X<M(X)+\sum_{\substack{d\mid m\\ d\ge 5}} \frac{4\log\log d}{\varphi(d)}.
\end{equation}
One checks that if $\log X>150$ (which is our case), then 
$$
M(X)<0.5\log X.
$$
Since $k\ge 1$, we have $k-0.5\ge k/2$, therefore inequality \eqref{eq11:20} implies that
$$
\frac{k}{2}\log X<8\sum_{\substack{d\mid m\\ d\ge 5}} \frac{\log\log d}{\varphi(d)},
$$
which is the required result.
\end{proof}

\begin{lemma}
\label{lem2:4}
Assume $X>e^{1000}$. We have 
$$
\log\(k\log X\) \leq 60\sum_{\substack{q\mid m \\ q\geq 5}}\frac{\log\log q}{q} + 25.
$$
\end{lemma}

\begin{proof} 
Lemma \ref{lem2:3} implies that 
\begin{equation}
\label{eq11:17}
k\log X<\sum_{\substack{d\mid m\\ d\ge 5}} \frac{10\log(2\log d)}{\varphi(d)}.
\end{equation}
Following the argument in the proof of Lemma \ref{lem:3}, one get that
$$
\sum_{\substack{d\mid m\\ d\ge 5}} \frac{10\log(2\log d)}{\varphi(d)}<\prod_{q\mid m}\left(1+\sum_{i\ge 1} \frac{10\log(2\log q^i)}{\varphi(q^i)}\right)-1.
$$
Inserting the above inequality into \eqref{eq11:17}, taking logarithms and using the fact that $\log(1+x)<x$ for all real numbers $x$, we get
\begin{equation}
\label{eq11:18}
\log\(k\log X\)<\sum_{q\mid m } \sum_{i\ge 1} \frac{10\log(2\log q^i)}{\varphi(q^i)}.
\end{equation}
Next, 
\begin{equation}
\label{yyy}
\sum_{\substack{i\geq 1\\ q\in \{2,3\}}} \frac{10\log(2\log (q^{i}))}{\varphi(q^{i})}<25.
\end{equation}
Hence, the inequality \eqref{eq:19} becomes
\begin{equation}
\label{eq11:19}
\sum_{i\ge 1} \frac{10 \log(2\log(q^i))}{\varphi(q^i)}<\frac{60\log\log q}{q}\quad {\text{\rm for}}
\quad q\ge 5.
\end{equation}
The desired inequality follows now from \eqref{yyy} and \eqref{eq11:19}.
\end{proof}

\begin{lemma}
\label{lem2:5}
We have $X<e^{e^{8000}}$.
\end{lemma}

\begin{proof}
By Lemma \ref{lem2:4}, we have 
\begin{eqnarray}
\label{eq11:22}
\log\(k\log X\) & < & 25 +60\left(\sum_{q\in U} \frac{\log\log q}{q}+\sum_{p\in V} \frac{\log\log q}{q}\right)\nonumber\\
& := & 25 + 60(T_1+T_2).
\end{eqnarray}
Clearly, if $q$ participates in $T_1$, then $q\ge 5$ divides both $m$ and $n$, so it is an odd prime factor of $k$. 

We next bound $T_2$ and we shall return to $T_1$ later. Following the argument in the proof of Lemma \ref{lem:4}, we get that
$$
60T_2\le 60\sum_{5\le q<1024} \frac{\log\log q}{q}+ 60\sum_{t\ge 10} \frac{(t+(\log\log X)/\log 2)\log\log 2^t}{2^t}.
$$
Since 
\begin{eqnarray*}
&& 60\sum_{5\le q<1024} \frac{\log\log q}{q}<96.74,\\
&& 60\sum_{t\ge 10} \frac{t\log\log 2^t}{2^t}<2.63,\\
&& 60\sum_{t\ge 10} \frac{\log\log 2^t}{2^t \log 2}<0.35,
\end{eqnarray*}
we get that
$$
60T_2< 100 + 0.35\log\log X.
$$
Hence, we get from \eqref{eq11:22}, that
\begin{equation}
\label{eq11:23}
\log\(k\log X\) 
\leq 60\(\sum_{\substack{q\mid \gcd(m,k)\\ q\geq 5}}\frac{\log\log q}{q}\)+ 0.35\log\log X + 
125,
\end{equation}
therefore
\begin{equation}
\label{eq11:24}
\log\log X<\frac{1}{0.65}\left(60T_1  -\log k + 125\right)<100 \log\log k-1.5\log k+200.
\end{equation}
If $k\in \{1,2\}$, then $T_1=0$, so $\log\log X<200$, or $X<e^{e^{200}}$.

Suppose now that $k\ge 3$. Then 
$$
T_1\le \log\log k\sum_{q\mid k} \frac{1}{q},
$$
which together with \eqref{eq11:24} gives
$$
\log\log X<100 \log\log k \sum_{q\mid k} \frac{1}{q}-1.5\log k + 200.
$$
However,
$$
\sum_{q\mid k}\frac{1}{q}<\sum_{d\mid k}\frac{1}{d}=\frac{\sigma(k)}{k}<\frac{k}{\varphi(k)}<L(k),
$$
by Lemma \ref{lem1:RS}$(iii)$. We thus get that
$$
\log\log X< 179(\log\log k)^2-1.5 \log k+450.
$$
Note that
$$
179 (\log\log k)^2 -1.5\log k+450=N(\log\log k),
$$
where $N(x):=179x^2 -1.5e^x+450$. The function $N(x)$ has a maximum at $x_0=7.48843\ldots$, and
 $N(x_0)<8000$, therefore 
$$
\log\log X<8000,
$$
giving $X<e^{e^{8000}}$, as desired.
\end{proof}

\begin{remark.}

A close analysis of our argument shows that $X<e^{e^{8000}}$ is in fact an upper bound for all positive integers $X$ arising from equations of the form
$$
\varphi\left(x\frac{X^m-1}{X-1}\right)=y\frac{X^n-1}{X-1}\qquad x,y\in \{1,2,\ldots,X\},
$$
assuming that $m\ge n+2$. Since $m\ge n$ must hold in the above equation, in order to infer whether equation \eqref{eq11:b} has only finitely or infinitely many solutions when $b\ge 2$ is a variable as well, it remains to only treat the cases $m=n$ and $m=n+1$. We leave the analysis of these cases as a future project.
\end{remark.}

\chapter{Repdigits and Lucas sequences}


\section{Introduction}
In this chapter, we are interested in finding all \textit{Repdigits} among members of some Lucas sequences. In Section \textcolor{red}{\ref{sec23}}, we study members of the Lucas sequence $\{L_n\}_{n\geq 0}$ whose Euler function is a repdigit. In Section \textcolor{red}{\ref{sec24}}, we find all the repdigits which are members of Pell or Pell-Lucas sequences. The main tools used to prove the results in this chapter are linear forms in logarithm \`a la Baker (Matveev Theorem), the Baker-Davenport reduction algorithm and the Primitive Divisor Theorem for members of Lucas sequences.

\section{Repdigits as Euler functions of Lucas numbers}
\label{sec23}
We prove in this section some results about the structure of all Lucas numbers whose Euler function is a repdigit in base $10$. For example, we show that if $L_n$ is such a Lucas number, then 
$n<10^{111}$ is of the form $p$ or $p^2$, where $p^3\mid 10^{p-1}-1$. We follow the material from \cite{JBLT}.

\medskip
As mentioned in the Introduction, here we look at the Diophantine equation 
\begin{equation}
\label{eq2:2}
\varphi(L_n)=d\left(\frac{10^m-1}{9}\right),\qquad d\in \{1,\ldots,9\}.
\end{equation}
We have the following result.
\begin{theorem}
\label{thm3:1}
Assume that $n>6$ is such that equation \eqref{eq2:2} holds with some $d$. Then:
\begin{itemize}
\item $d=8$;
\item $m$ is even;
\item $n=p$ or $p^2$, where $p^3\mid 10^{p-1}-1$.
\item $10^9<p<10^{111}$.
\end{itemize}
\end{theorem}


\section{The proof of Theorem \ref{thm3:1}}

\subsection{Method of the proof}
The main method for solving the Diophantine equation \eqref{eq2:2} consists
essentially of three parts: a transformation step, an application of the theory of
linear forms in logarithms of algebraic numbers, and a procedure for reducing upper
bounds. 
\medskip

Firstly, we transform  the equation \eqref{eq2:2} into a purely exponential equation
or inequality, i.e., a Diophantine equation or inequality where the unknowns are
in the exponents. 

Secondly, a straightforward use of the theory of linear form in logarithms gives a very large bound on $n$, which has been explicitly computed by using Matveev's theorem \cite{matveev}.

Thirdly, we use a Diophantine approximation algorithm, so-called the Baker-Davenport reduction method to reduce the bounds.

\subsection*{The exponent of $2$ on both sides of \eqref{eq2:2}}

Write
\begin{equation}
\label{eq2:3}
L_n=2^{\delta} p_1^{\alpha_1}\cdots p_r^{\alpha_r},
\end{equation}
where $\delta\ge 0,~r\ge 0$, $p_1<\ldots<p_r$ odd primes and $\alpha_i\geq 0$ for all $i=1,\ldots,r.$
Then
\begin{equation}
\label{eq2:Euler}
\varphi(L_n)=2^{\max\{0,\delta-1\}} p_1^{\alpha_1-1}(p_1-1)p_2^{\alpha_2-1}(p_2-1)\cdots p_r^{\alpha_r-1}(p_r-1).
\end{equation}
Applying the $\nu_2$ function on both sides of \eqref{eq2:2} and using \eqref{eq2:Euler}, we get
\begin{align} 
\label{eq2:4}
\max\{0,\delta-1\}+\sum_{i=1}^r \nu_2(p_i-1)=\nu_2(\varphi(L_n)) = \nu_2\left(d\left(\frac{10^m-1}{9}\right)\right) =  \nu_2(d).
\end{align}
Note that $\nu_2(d)\in \{0,1,2,3\}$. Note also that $r\le 3$ and since $L_n$ is never a multiple of $5$, we have that
\begin{equation}
\label{eq2:ineq}
\frac{\varphi(L_n)}{L_n}\ge \left(1-\frac{1}{2}\right)\left(1-\frac{1}{3}\right)\left(1-\frac{1}{7}\right)>\frac{1}{4},
\end{equation}
so $\varphi(L_n)>L_n/4$. This shows that if $n\ge 8$ satisfies equation \eqref{eq2:2}, then $\varphi(L_n)>L_8/4>10$, so $m\ge 2$. 

We will also use in the later stages of this section the Binet formula \eqref{binet2} with $(\alpha_1,\alpha_2):=(\alpha,\beta)=((1+{\sqrt{5}})/2,(1-{\sqrt{5}})/2)$. In particular,
\begin{equation}
\label{eq2:lll}
L_n-1=\alpha^n-(1-\beta^n)\le \alpha^n\qquad {\text{\rm for~all}}\qquad n\ge 0.
\end{equation}
Furthermore,
\begin{equation}
\label{eq2:llll}
\alpha^{n-1}\le L_n<\alpha^{n+1}\qquad {\text{\rm for~all}}\qquad n\ge 1.
\end{equation}

\subsection{The case of the digit $d\not\in \{4,8\}$}

 If $\nu_2(d)=0$, we get that $d\in \{1,3,5,7,9\}$,  $\varphi(L_n)$ is odd, so $L_n\in \{1,2\}$, therefore $n=0,~1$. If $\nu_2(d)=1$, we get  $d\in \{2,6\}$, and from \eqref{eq2:4} either $\delta=2$ and $r=0$, so $L_n=4$, therefore 
$n=3$, or $\delta\in \{0,1\}$, $r=1$ and $p_1\equiv 3\pmod 4$. Thus, $L_n=p_1^{\alpha_1}$ or $L_n=2p_1^{\alpha_1}$. Lemma \ref{lem1:BLMS} shows that $\alpha_1=1$ except for the case when $n=6$ when 
$L_6=2\times 3^2$. So, for $n\ne 6$, we get that $L_n=p_1$ or $2p_1$. Let us see that the second case is not possible. Assuming it is, we get $6\mid n$. Write $n=2^t \times 3 \times m$, where $t\ge 1$ and $m$ is odd. Clearly, $n\ne 6$. 

If $m>1$, then $L_{2^t 3 m}$ has a primitive divisor which does not divide the number $L_{2^t 3}$. Hence, $L_n=2p_1$ is not possible in this case. However, if $m=1$ then $t>1$, and both $L_{2^t}$ and $L_{2^t 3}$ have primitive divisors, so the equation $L_n=2p_1$ is not possible in this case either. So, the only possible case is $L_n=p_1$. Thus, we get
$$
\varphi(L_n)=L_n-1=d\left(\frac{10^m-1}{9}\right)\quad {\text{\rm and}}\quad d\in \{2,6\},
$$
so 
$$
L_n=d\left(\frac{10^m-1}{9}\right)+1\quad {\text{\rm and}}\quad d\in \{2,6\}.
$$
 When $d=2$, we get $L_n\equiv 3\pmod 5$. For the Lucas sequence $\{L_n\}_{n\ge 0}$, the value of the period modulo $5$ is $4$ . Furthermore, from $L_n\equiv 3\pmod 5$, we get that $n\equiv 2\pmod 4$. Thus, $n=2(2k+1)$ for some $k\ge 0$. 
 However, this is not possible for $k\ge 1$, since for $k=1$, we get that $n=6$ and $L_6=2\times 3^2$, while for $k>1$, we have that $L_n$ is divisible by both the primes $3$ and at least another prime, namely a primitive prime factor of $L_n$, so $L_n=p_1$ is not possible. Thus, $k=0$, so $n=2$.
 
 When $d=6$, we get that $L_n\equiv 2\pmod 5$. This shows that $4\mid n$. Write $n=2^t (2k+1)$ for some $t\ge 2$ and $k\ge 0$. 
 As before, if $k\ge 1$, then $L_n$ cannot be a prime since either $k=1$, so $3\mid n$, and then $L_n>2$ is even, or $k\ge 2$, and then 
  $L_n$ is divisible by at least two primes, namely the primitive prime factors of $L_{2^t}$ and of $L_n$. Thus, $n=2^t$. Assuming $m\ge 2$, and reducing both sides of the above formula
$$
L_{2^{t-1}}^2-2=L_{2^t}=6\left(\frac{10^m-1}{9}\right)+1
$$
modulo $8$, we get $7\equiv -5\pmod{8}$, which is not possible. This shows that $m=1$, so $t=2$, therefore $n=4$. 

To summarize, we have proved the following result.

\begin{lemma}
The equation 
\begin{equation*}
\varphi(L_n)=d\left(\frac{10^m-1}{9}\right),\qquad d\in \{1,\ldots,9\},
\end{equation*} 
has no solutions with $n>6$ if $d\not\in \{4,8\}$.
\end{lemma}

\subsection{The case of $L_n$ even}

Next we treat the case $\delta>0$. It is well-known and easy to see by looking at the period of $\{L_n\}_{n\ge 0}$ modulo $8$ that $8\nmid L_n$ for any $n$. Hence, we only need to deal with the cases $\delta=1$ or $2$.

If $\delta=2$, then $3\mid n$ and $n$ is odd. Furthermore, relation \eqref{eq2:4} shows that $r\le 2$. Assume first that $n=3^t$. We check that $t=2,3$ are not convenient. For $t\ge 4$, we have  $L_9,~L_{27}$ and $L_{81}$ are divisors of $L_n$ and all have odd primitive divisors which are prime factors of $L_n$, contradicting the fact that $r\le 2$. Assume now that $n$ is a multiple of some prime $p\ge 5$. Then $L_p$ and $L_{3p}$ already have primitive prime factors, so $n=3p$, for if not, then $n>3p$, and $L_n$ would have (at least) one additional prime factor, namely a primitive prime factor of $L_n$. Thus, $n=3p$. Write
$$
L_n=L_{3p}=L_p(L_p^2+3).
$$
The two factors above are coprime, so, up to relabeling the prime factors of $L_n$, we may assume that $L_p=p_1^{\alpha_1}$ and $L_p^2+3=4p_2^{\alpha_2}$. Lemma \ref{lem1:BLMS} shows that $\alpha_1=1$. Further, since $p$ is odd,
we get that $L_p\equiv 1,4\pmod 5$, therefore the second relation above implies that $p_2^{\alpha_2}\equiv 1\pmod 5$. If $\alpha_2$ is odd, we then get that $p_2\equiv 1\pmod 5$. This leads to 
$5\mid (p_2-1)\mid \varphi(L_n)=d(10^m-1)/9$ with $d\in \{4,8\}$, which is a contradiction. Thus, $\alpha_2$ is even, showing that 
$$
L_p^2+3=\square,
$$
which is impossible. 

If $\delta=1$, then $6\mid n$. Assume first that $p\mid n$ for some prime $p>3$. Write $n=2^t \times 3\times m$. If $t\ge 2$, then $r\ge 4$, since $L_n$ is then a multiple of a primitive prime factor of $L_{2^t}$, a primitive prime factor of $L_{2^t 3}$, a primitive prime factor of $L_{2^t p}$ and a primitive prime factor of $L_{2^t 3p}$. So, $t=1$. Then $L_n$ is a multiple of $3$ and of the primitive prime factors of $L_{2p}$ and $L_{6p}$, showing that $n=6p$, for if not, then $n>6p$ and $L_n$ would have 
(at least) an additional prime factor, namely a primitive prime factor of $L_n$. Thus, with $n=6p$, we may write
$$
L_n=L_{6p}=L_{2p}(L_{2p}^2-3).
$$
Further, it is easy to see that up to relabeling the prime factors of $L_n$, we may assume that $p_1=3,~\alpha_1=2$, $L_{2p}=3p_2^{\alpha_2}$ and $L_{2p}^2-3=6p_3^{\alpha_3}$. Furthermore, since $r=3$, relation \eqref{eq2:4} tells us that 
$p_i\equiv 3\pmod 4$ for $i=2,3$. Reducing equation
$$
L_p^2+2=L_{2p}=3p_2^{\alpha_2}
$$
modulo $4$ we get $3\equiv 3 ^{\alpha_2+1}\pmod 4$, so $\alpha_2$ is even. We thus get $L_{2p}=3\square$, an equation which has no solutions by Lemma \ref{lem1:BLMS}. 

So, it remains to assume that $n=2^t\times 3^s$. 

Assume $s\ge 2$. If also $t\ge 2$, then $L_n$ is divisible by the primitive prime factors of $L_{2^t}$, $L_{2^t 3}$ and $L_{2^t 9}$. This shows that $n=2^t \times 9$ and we have
$$
L_n=L_{2^t 9}=L_{2^t} (L_{2^t}^2-3)(L_{2^t 3}^2-3).
$$
Up to relabeling the prime factors of $L_n$, we get $L_{2^t}=p_1^{\alpha_1}$, $L_{2^t}^2-3=2p_2^{\alpha_2}$, $L_{2^t 3}^2-3=p_3^{\alpha_3}$ and $p_i\equiv 3\pmod 4$ for $i=1,2,3$. Reducing the last relation modulo $4$, we get
$1\equiv 3^{\alpha_3}\pmod 4$, so $\alpha_3$ is even. We thus get $L_{2^t 3}^2-3=\square$, and this is false. Thus, $t=1$. By the existence of primitive divisors Lemma \ref{lem1:Car}, $s\in \{2,3\}$, so $n\in \{18,54\}$ and none leads to a solution. 

Assume next that $s=1$. Then $n=2^t\times 3$ and $t\ge 2$.
We write
$$
L_n=L_{2^t 3}=L_{2^t} (L_{2^t}^2-3).
$$
Assume first that there exist $i$ such that $p_i\equiv 1\pmod 4$. Then $r\le 2$ by \eqref{eq2:4}. It then follows that in fact $r=2$ and up to relabeling the primes we have $L_{2^t}=p_1^{\alpha_1}$ and $L_{2^t}^2-3=2p_2^{\alpha_2}$. 
Since $L_{2^t}=L_{2^{t-1}}^2-2$, we get that $L_{2^{t-1}}^2-2=p_1^{\alpha_1}$, which reduced modulo $4$ gives $3\equiv p_1^{\alpha_1}\pmod 4$, therefore $p_1\equiv 3\pmod 4$. As for the second relation, we get $(L_{2^t}^2-3)/2=p_2^{\alpha_2}$,
which reduced modulo $4$ also gives $3\equiv p_2^{\alpha_2}\pmod 4$, so also $3\equiv p_2\pmod 4$. This is impossible since $p_i\equiv 1\pmod 4$ for some $i\in \{1,\ldots,r\}$. Thus, $p_i\equiv 3\pmod 4$ for all $i\in \{1,\ldots,r\}$. Reducing relation
$$
L_{2^t 3}^2-5F_{2^t 3}^2=4
$$ modulo $p_i$, we get that ${\displaystyle{\left(\frac{-5}{p_i}\right)=-1}}$, and since $p_i\equiv 3\pmod 4$, we get that ${\displaystyle{\left(\frac{5}{p_i}\right)=-1}}$ for $i\in \{1,\ldots,r\}$. Since $p_i$ are also primitive prime factors for 
$L_{2^t}$ and/or $L_{2^t 3}$, respectively, we get that $p_i\equiv -1\pmod {2^t}$.

Suppose next that $r=2$. We then get $d=4$, 
$$
L_{2^{t-1}}^2-2=L_{2^t}=p_1^{\alpha_1}\qquad {\text{\rm and}}\qquad L_{2^t}^2-3=2p_2^{\alpha_2}.
$$
Reducing the above relations modulo $8$, we get that $\alpha_1,\alpha_2$ are odd. Thus, 
\begin{eqnarray*}
&& 4\left(\frac{10^m-1}{9}\right)=\varphi(L_n)=p_1^{\alpha_1-1}(p_1-1)p_2^{\alpha_2-1}(p_2-1)\\
& \equiv & (-1)^{\alpha_1-1}(-2) (-1)^{\alpha_2-1}(-2)\pmod {2^t}\equiv 4\pmod {2^t},
\end{eqnarray*}
giving 
$$
\frac{10^m-1}{9}\equiv 1\pmod {2^{t-2}}\qquad {\text{\rm therefore}}\qquad 10^m\equiv 10\pmod {2^{t-2}},
$$
so $t\le 3$ for $m\ge 2$. Thus, $n\in \{12, 24\}$, in these cases equation \eqref{eq2:2} has no solutions.

Assume next that $r=3$. We then get that $d=8$ and either
$$
L_{2^{t-1}}^2-2=L_{2^t}=p_1^{\alpha_1} p_2^{\alpha_2}\qquad {\text{\rm and}}\qquad L_{2^t}^2-3=2p_3^{\alpha_3},
$$
or
$$
L_{2^{t-1}}^2-2=L_{2^t}=p_1^{\alpha_1}\qquad {\text{\rm and}}\qquad L_{2^t}^2-3=2p_2^{\alpha_2}p_3^{\alpha_3}.
$$
Reducing the above relations modulo $8$ as we did before, we get that exactly one of $\alpha_1,\alpha_2,\alpha_3$ is even and  the other two are odd. Then
\begin{eqnarray*}
&& 8\left(\frac{10^m-1}{9}\right)=\varphi(L_n)=p_1^{\alpha_1-1}(p_1-1)p_2^{\alpha_2-1}(p_2-1) p_3^{\alpha_3-1}(p_3-1)\\
&\equiv& (-1)^{\alpha_1+\alpha_2+\alpha_3-3} (-2)^3\pmod {2^t}\equiv 8\pmod {2^t}
\end{eqnarray*}
giving 
$$
\frac{10^m-1}{9}\equiv 1\pmod {2^{\max\{0,t-3\}}}\quad {\text{\rm therefore}}\quad 10^m\equiv 10\pmod {2^{\max\{0,t-3\}}},
$$
which implies that $t\le 4$ for $m\ge 2$. The only new possibility is $n=48$, which does not fulfill \eqref{eq2:2}. 

So, we proved the following result.

\begin{lemma}
There is no $n>6$ with $L_n$ even such that the relation
\begin{equation*}
\varphi(L_n)=d\left(\frac{10^m-1}{9}\right),\qquad d\in \{1,\ldots,9\},
\end{equation*}
holds.
\end{lemma}

\subsection{The case of $n$ even}

Next we look at the solutions of \eqref{eq2:2} with $n$ even. Write $n=2^t m$, where $t\ge 1$, $m$ is odd and coprime to $3$.

Assume first that there exists an $i$ such that $p_i\equiv 1\pmod 4$. Without loss of generality we assume that $p_1\equiv 1\pmod 4$.  It then follows from \eqref{eq2:4} that $r\le 2$, and that $r=1$ if $d=4$. So, if $d=4$, then $r=1$, $L_n=p_1^{\alpha_1}$, and by Lemma  \ref{lem1:BLMS}, we get that $\alpha_1=1$. In this case, by the existence of primitive divisors Lemma \ref{lem1:Car}, we get that $m=1$, otherwise $L_n$ would be divisible both by a primitive prime factor of $L_{2^t}$ as well as by a primitive prime factor 
of $L_{n}$. Hence, $L_{2^t}=p_1$, so
$$
L_{2^t}-1=\varphi(L_{2^t})=4\left(\frac{10^m-1}{9}\right),\quad {\text{\rm therefore}}\quad L_{2^t}\equiv 5\pmod {10}.
$$
Thus, $5\mid L_n$ and this is not possible for any $n$. Suppose now that $d=8$. If $t\ge 2$, then
$$
L_{n/2}^2-2=L_n
$$
and reducing the above relation modulo $p_1$, we get that ${\displaystyle{\left(\frac{2}{p_1}\right)=1}}$. Since $p_1\equiv 1\pmod 4$, we read that $p_1\equiv 1\pmod 8$. Relation \eqref{eq2:4} shows that $r=1$ so $L_n=p_1^{\alpha_1}$.
By Lemma  \ref{lem1:BLMS}, we get again that $\alpha_1=1$ and by the existence of primitive divisors Lemma \ref{lem1:Car}, we get that $m=1$. Thus, 
$$
L_{2^{t}}-1=\varphi(L_{2^t})=8\left(\frac{10^m-1}{9}\right),\quad {\text{\rm therefore}}\quad L_{2^t}\equiv 4\pmod 5,
$$
which is impossible for $t\ge 2$, since $L_{n}\equiv 2\pmod 5$ whenever $n$ is a multiple of $4$. This shows that $t=1$, so $m>1$. Let $p\ge 5$ be a prime factor of $n$. Then $L_n$ is divisible by $3$ and by the primitive prime factor
of $L_{2p}$, and since $r\le 2$, we get that $r=2$, and $n=2p$. Thus, $L_n=L_{2p}=3p_2^{\alpha_2}$, and, by Lemma \ref{lem1:BLMS}, we get that $\alpha_2=1$. Reducing the above relation modulo $5$, we get that
$3\equiv 3p_2\pmod 5$, so $p_2\equiv 1\pmod 5$, showing that $5\mid (p_2-1)\mid \varphi(L_n)=8(10^m-1)/9$, which is impossible.

This shows that  in fact we have $p_i\equiv 3\pmod 4$ for $i=1,\ldots,r$. Reducing relation $L_n^2-5F_n^2=4$ modulo $p_i$, we get that ${\displaystyle{\left(\frac{-5}{p_i}\right)=1}}$ for $i=1,\ldots,r$. Since 
we already know that ${\displaystyle{\left(\frac{-1}{p_i}\right)=-1}}$, we get that ${\displaystyle{\left(\frac{5}{p_i}\right)=-1}}$ for all $i=1,\ldots,r$. Since in fact $p_i$ is always a primitive divisor for $L_{2^t d_i}$ for some divisor 
$d_i$ of $m$, we get that $p_i\equiv -1\pmod {2^t}$. Reducing relation
$$
L_n=p_1^{\alpha_1}\cdots p_r^{\alpha_r}
$$
modulo $4$, we get $3\equiv 3^{\alpha_1+\cdots+\alpha_r}\pmod 4$, therefore $\alpha_1+\cdots+\alpha_r$ is odd. Next, reducing the relation
$$
\varphi(L_n)=p_1^{\alpha_1-1}(p_1-1)\cdots p_r^{\alpha_r-1}(p_r-1)
$$
modulo $2^t$, we get
$$
d\left(\frac{10^m-1}{9}\right)=\varphi(L_n)\equiv (-1)^{\alpha_1+\cdots+\alpha_r-r} (-2)^r\pmod {2^t}\equiv -2^r\pmod {2^t}.
$$
Since $r\in \{2,3\}$ and $d=2^r$, we get that 
$$
\frac{10^m-1}{9}\equiv -1\pmod {2^{\max\{0,t-r\}}},\quad {\text{\rm so}}\quad 10^m\equiv 8\pmod {2^{\max\{0,t-r\}}}.
$$
Thus, if $m\ge 4$, then $t\le 6$. Suppose that $m\ge 4$. Computing $L_{2^t}$ for $t\in \{5,6\}$, we get that $p\equiv 1\pmod 5$ for each prime factor $p$ of them. Thus, $5\mid (p-1)\mid \varphi(L_n)=d(10^m-1)/9$, which is impossible.
Hence, $t\in \{1,2,3,4\}$. We get the relations
\begin{equation}
\label{eq2:10}
L_{2^t m}=L_{2^t} p_1^{\alpha_1},\quad {\text{\rm or}}\quad L_{2^t m} =L_{2^t} p_2^{\alpha_2}p_3^{\alpha_3}\quad {\text{\rm and}}\quad t\in \{1,2,3,4\}.
\end{equation}
Assume that the  first relation in \eqref{eq2:10} holds for some $t\in \{1,2,3,4\}$. Reducing the first equation in \eqref{eq2:10} modulo $5$, we get  $L_{2^t}\equiv L_{2^t} p_1^{\alpha_1}\pmod 5$, therefore $p_1^{\alpha_1}\equiv 1\pmod 5$.
If $\alpha_1$ is odd, we then obtain  $p_1\equiv 1\pmod 5$; hence, $5\mid (p_1-1)\mid \varphi(L_n)=d(10^m-1)/9$ with $d\in \{4,8\}$, which is impossible. If $\alpha_1$ is even, we then get that $L_n/L_{2^t}=p_1^{\alpha_1}=\square
$, and this is impossible since $n\ne 2^t \times 3$ by Lemma \ref{lem1:RM}. Assume now that the second relation in  \eqref{eq2:10} holds for some $t\in \{2,3,4\}$. Reducing it modulo $5$, we get $L_{2^t}\equiv L_{2^t} p_2^{\alpha_2} p_3^{\alpha_3}\pmod 5$.
Hence, $p_2^{\alpha_2}p_3^{\alpha_3}\equiv 1\pmod 5$. Now
\begin{eqnarray*}
8\left(\frac{10^m-1}{9}\right) & = & \varphi(L_n)=(L_{2^t}-1)p_2^{\alpha_2-1} p_3^{\alpha_3-1}(p_2-1)(p_3-1)\\
& \equiv & \left(\frac{p_2-1}{p_2}\right)\left(\frac{p_3-1}{p_3}\right) \pmod 5,
\end{eqnarray*}
so
$$
\left(\frac{p_2-1}{p_2}\right)\left(\frac{p_3-1}{p_3}\right)\equiv 3\pmod 5.
$$
The above relation shows that $p_2$ and $p_3$ are distinct modulo $5$, because  otherwise the left--hand side above is a quadratic residue modulo $5$ while $3$ is not a quadratic residue modulo $5$. Thus, $\{p_2,p_3\}\equiv \{2,3\}\pmod 5$, and we get
$$
\left(\frac{2-1}{2}\right)\left(\frac{3-1}{3}\right)\equiv 3\pmod 5\quad {\text{\rm or}}\quad 1\equiv 3^2\pmod 5,
$$
a contradiction. Finally, assume that $t=1$ and that the right relation \eqref{eq2:10} holds. Reducing it modulo $4$, we get $3\equiv 3^{\alpha_2+\alpha_3}\pmod 4$, therefore $\alpha_2+\alpha_3$ is even.
If $\alpha_2$ is even, then so is $\alpha_3$, so we get that $L_{2m}=3\square$, which is false by Lemma \ref{lem1:RM}. Hence, $\alpha_2$ and $\alpha_3$ are both odd. Furthermore, since $m$ is odd and not a multiple of $3$, 
we get that $2m\equiv 2\pmod 4$ and $2m\equiv 2,4\pmod 6$, giving $2m\equiv 2,10\pmod {12}$. Looking at the values of $\{L_n\}_{n\ge 1}$ modulo $8$, we see that the period is $12$, and $L_2\equiv L_{10}\equiv 3\pmod 8$, showing that $L_{2m}\equiv 3\pmod 8$.
This shows that $p_2^{\alpha_2} p_3^{\alpha_3}\equiv 1\pmod 8$, and since $\alpha_2$ and $\alpha_3$ are odd, we get the congruence $p_2p_3\equiv 1\pmod 8$. This together with the fact that $p_i\equiv 3\pmod 4$ for $i=1,2$, 
implies that $p_2\equiv p_3\pmod 8$. Thus, $(p_2-1)/2$ and $(p_3-1)/2$ are congruent modulo $4$ so their product is $1$ modulo $4$. Now we write
\begin{eqnarray*}
\varphi(L_n) & = & (3-1) (p_2-1)p_2^{\alpha_2-1} (p_3-1)p_3^{\alpha_3-1}\\
& = & 8 \frac{(p_2-1)}{2} \frac{(p_3-1)}{2} p_2^{\alpha_2-1} p_3^{\alpha_3-1}=8M,
\end{eqnarray*}
where $M\equiv 1\pmod 4$. However, since in fact $M=(10^m-1)/9$, we get that $M\equiv 3\pmod 4$ for $m\ge 2$, a contradiction. So, we must have $m\le 3$, therefore $L_n<4000$, so $n\le 17$, and such values can be dealt with by hand.

Thus, we have proved the following result.

\begin{lemma}
There is no $n>6$ even such that relation 
\begin{equation*}
\varphi(L_n)=d\left(\frac{10^m-1}{9}\right),\qquad d\in \{1,\ldots,9\},
\end{equation*}
holds.
\end{lemma}

\subsection{$r=3,~d=8$ and $m$ is even}

From now on, $n>6$ is odd and $L_n$ is also odd. If $p\mid L_n$, with $p$ a prime number, therefore reducing the equation $L_n^2-5F_n^2=-4$ modulo $p$ we get that ${\displaystyle{\left(\frac{5}{p}\right)=1}}$. Thus, 
$p\equiv 1,4\pmod 5$. If $p\equiv 1\pmod 5$, then $5\mid (p-1)\mid \varphi(L_n)=d(10^m-1)/9$ with $d\in \{4,8\}$, a contradiction. Thus, $p_i\equiv 4\pmod 5$ for all $i=1,\ldots,r$.

We next show that $p_i\equiv 3\pmod 4$ for all $i=1,\ldots,r$. Assume that this is not so and suppose that $p_1\equiv 1\pmod 4$. If $r=1$, then $L_n=p_1^{\alpha_1}$ and by Lemma 
\ref{lem1:BLMS}, we have $\alpha_1=1$. So,
$$
L_n-1=\varphi(L_n)=d\left(\frac{10^m-1}{9}\right)\quad {\text{\rm so}}\quad L_n=d\left(\frac{10^m-1}{9}\right)+1.
$$
If $d=4$, then $L_n\equiv 5\pmod {10}$, so $5\mid L_n$, which is false. When $d=8$, we get that $n\equiv 3\pmod 4$ following the fact that $L_n\equiv 4\pmod 5$. However, we also have that $L_n\equiv 1\pmod 8$, showing that 
$n\equiv 1\pmod {12}$; in particular, $n\equiv 1\pmod 4$, a contradiction.

Assume now that $r=2$. Then $L_n=p_1^{\alpha_1} p_2^{\alpha_2}$ and $d=8$. Then
\begin{equation}
\label{eq2:11}
\varphi(L_n)=(p_1-1)p_1^{\alpha_1-1} (p_2-1) p_2^{\alpha_2-1} = 8\left(\frac{10^m-1}{9}\right).
\end{equation}
Reducing the above relation \eqref{eq2:11} modulo $5$ we get $4^{\alpha_1+\alpha_2-2} \times 3^2\equiv 3\pmod 5$, which is impossible since the left--hand side is a quadratic residue modulo $5$ while the right--hand side is not. 

Thus, $p_i\equiv 3\pmod 4$ for $i=1,\ldots,r$. Assume next that $r=2$. Then $L_n=p_1^{\alpha_1} p_2^{\alpha_2}$ and $d=4$. Then
\begin{equation}
\label{eq2:12}
\varphi(L_n)=(p_1-1) p_1^{\alpha_1-1} (p_2-1) p_2^{\alpha_2-1}=4\frac{10^m-1}{9}.
\end{equation}
Reducing the above relation \eqref{eq2:12} modulo $5$, we get $4^{\alpha_1+\alpha_2-2} \times 3^2\equiv 4\pmod 5$, therefore $4^{\alpha_1+\alpha_2-2}\equiv 1\pmod 5$. Thus, $\alpha_1+\alpha_2$ is even.
If $\alpha_1$ is even, so is $\alpha_2$, so $L_n=\square$, and this is false by Lemma \ref{lem1:BLMS}. Hence, $\alpha_2$ and $\alpha_3$ are both odd. It now follows that $L_n\equiv 3^{\alpha_1+\alpha_2}\pmod 4$, so 
$L_n\equiv 1\pmod 4$, therefore $n\equiv 1\pmod 6$, and also $L_n\equiv 4^{\alpha_1+\alpha_2}\pmod 5$, so $L_n\equiv 1\pmod 5$, showing that $n\equiv 1\pmod 4$. Hence, $n\equiv 1\pmod {12}$, showing that $L_n\equiv 1\pmod 8$.
Thus, $p_1^{\alpha_1} p_2^{\alpha_2}\equiv 1\pmod 8$, and since $\alpha_1$ and $\alpha_2$ are odd and $p_1^{\alpha_1-1}$ and $p_2^{\alpha_2-1}$ are congruent to $1$ modulo $8$ (as perfect squares), we therefore get that
$p_1 p_2\equiv 1\pmod 8$. Since also $p_1\equiv p_2\equiv 3\pmod 4$, we get that in fact $p_1\equiv p_2\pmod 8$. Thus, $(p_1-1)/2$ and $(p_2-1)/2$ are congruent modulo $4$ so their product is $1$ modulo $4$. Thus,
$$
\varphi(L_n)=4 \left(\frac{(p_1-1)}{2} \frac{(p_2-1)}{2}\right) p_1^{\alpha_1-1} p_2^{\alpha_2-1} =4M,
$$
where $M\equiv 1\pmod 4$. Since in fact we have $M=(10^m-1)/9$, we get that $M\equiv 3\pmod 4$ for $m\ge 2$, a contradiction.

Thus, $r=3$ and $d=8$.  To get that $m$ is even, we write $L_n=p_1^{\alpha_1} p_2^{\alpha_2} p_3^{\alpha_3}$. So, 
\begin{equation}
\label{eq2:13}
\varphi(L_n)=(p_1-1) p_1^{\alpha_1-1} (p_2-1) p_2^{\alpha_2-1} (p_3-1) p_3^{\alpha_3-1} = 8\left(\frac{10^m-1}{9}\right),
\end{equation}
Reducing equation \eqref{eq2:13} modulo $5$ we get $4^{\alpha_1+\alpha_2+\alpha_3-3} \times 3^3\equiv 3\pmod 5$, giving $4^{\alpha_1+\alpha_2+\alpha_3}\equiv 1\pmod 5$. Hence, $\alpha_1+\alpha_2+\alpha_3$ is even.
It is not possible that all $\alpha_i$ are even for $i=1,2,3$, since then we would get $L_n=\square$, which is not possible by Lemma \ref{lem1:BLMS}. Hence, exactly one of them is even, say $\alpha_3$ and the other two are odd. 
Then $L_n\equiv 3^{\alpha_1+\alpha_2+\alpha_3}\equiv 1\pmod 4$ and $L_n\equiv 4^{\alpha_1+\alpha_2+\alpha_3}\equiv 1\pmod 5$. Thus, $n\equiv 1\pmod 6$ and $n\equiv 1\pmod 4$, so $n\equiv 1\pmod {12}$.This 
shows that $L_n\equiv 1\pmod 8$. Since $p_1^{\alpha_1-1} p_2^{\alpha_2-1} p_3^{\alpha_3}$ is congruent to $1$ modulo $8$ (as a perfect square), we get that $p_1p_2 \equiv 1\pmod 8$. Thus, 
$p_1\equiv p_2\pmod 8$, so $(p_1-1)/2$ and $(p_2-1)/2$ are congruent modulo $4$ so their product is $1$. Then
\begin{equation}
\label{eq2:14}
\varphi(L_n)=8 \left(\frac{(p_1-1)}{2} \frac{(p_2-1)}{2}\right) \left(\frac{p_3(p_3-1)}{2}\right) p_1^{\alpha_1-1} p_2^{\alpha_2-1} p_3^{\alpha_3-2}=8M,
\end{equation}
where $M=(10^m-1)/9\equiv 3\pmod 4$. In the above product, all odd factors are congruent to $1$ modulo $4$ except possibly for $p_3(p_3-1)/2$. 
This shows that $p_3(p_3-1)/2\equiv 3\pmod 4$, which shows that $p_3\equiv 3\pmod 8$. Now since $p_3^2\mid L_n$, we get 
that $p_3\mid \varphi(L_n)=8(10^m-1)/9$. So, $10^m\equiv 1\pmod {p_3}$. Assuming that $m$ is odd, we would get 
$$
1=\left(\frac{10}{p_3}\right)=\left(\frac{2}{p_3}\right)\left(\frac{5}{p_3}\right)=-1,
$$
which is a contradiction. In the above, we used that $p_3\equiv 3\pmod 8$ and $p_3\equiv 4\pmod 5$ and quadratic reciprocity to conclude that ${\displaystyle{\left(\frac{2}{p_3}\right)=-1}}$ as well as ${\displaystyle{\left(\frac{5}{p_3}\right)=\left(\frac{p_3}{5}\right)=1}}$. 

So, we have showed the following result.

\begin{lemma}
\label{lem3:9}
If $n>6$ is a solution of the equation 
\begin{equation*}
\varphi(L_n)=d\left(\frac{10^m-1}{9}\right),\qquad d\in \{1,\ldots,9\},
\end{equation*}
then $n$ is odd, $L_n$ is odd, $r=3,~d=8$ and $m$ is even. Further, $L_n=p_1^{\alpha_1} p_2^{\alpha_2} p_3^{\alpha_3}$, where $p_i\equiv 3\pmod 4$ and $p_i\equiv 4\pmod 5$ for $i=1,2,3$, $p_1\equiv p_2\pmod 8$, $p_3\equiv 3\pmod 8$, $\alpha_1$ and $\alpha_2$ are odd and $\alpha_3$ is even.
\end{lemma}

\subsection{$n\in \{p,p^2\}$ for some prime $p$ with $p^3\mid 10^{p-1}-1$}
\label{Jhon}

The factorizations of all Lucas numbers $L_n$ for $n\le 1000$ are known. We used them and Lemma \ref{lem3:9} and found no solution to equation \eqref{eq2:2} with $n\in [7,1000]$.

Let $p$ be a prime factor of $n$. Suppose first that $n=p^t$ for some positive integer $t$. If $t\ge 4$, then $L_n$ is divisible by at least four primes, namely primitive prime factors of $L_p,~L_{p^2},~L_{p^3}$ and $L_{p^4}$, respectively, which is false.
Suppose that $t=3$. Write
$$
L_n=L_p \left(\frac{L_{p^2}}{L_p}\right)\left(\frac{L_{p^3}}{L_{p^2}}\right).
$$
The three factors above are coprime, so they are $p_1^{\alpha_1},~p_2^{\alpha_2},~p_3^{\alpha_3}$ in some order. Since $\alpha_3$ is even, we get that one of $L_p,~L_{p^2}/L_p$ or $L_{p^3}/L_{p^2}$ is a square, which is false
by Lemmas \ref{lem1:BLMS} and \ref{lem1:RM}. Hence, $n\in \{p,~p^2\}$. All primes $p_1,~p_2,~p_3$ are quadratic residues modulo $5$. When $n=p$, they are primitive prime factors of $L_p$. When $n=p^2$, all of them are primitive prime factors 
of $L_p$ or $L_{p^2}$ with at least one of them being a primitive prime factor of $L_{p^2}$. Thus, $p_i\equiv 1\pmod p$ holds for all $i=1,2,3$ both in the case $n=p$ and $n=p^2$, and when $n=p^2$ at least one of the the above congruences holds modulo $p^2$. This shows that $p^3\mid (p_1-1)(p_2-1)(p_3-1)\mid \varphi(L_n)=8(10^m-1)/9$, so $p^3\mid 10^{m}-1$. When $n=p^2$, we in fact have $p^4\mid 10^m-1$. 
Assume now that $p^3\nmid 10^{p-1}-1$. Then the congruence $p^3\mid 10^{m}-1$ implies $p\mid m$, while the congruence $p^4\mid 10^m-1$ implies $p^2\mid m$. Hence, when $n=p$, we have
$$
2^p>L_p>\varphi(L_n)=8(10^m-1)/9>(10^p-1)/9>10^{p-1}
$$
which is false for any $p\ge 3$. Similarly, if $n=p^2$, then 
$$
2^{p^2}>L_{p^2}>\varphi(L_n)=8(10^m-1)/9>(10^{p^2}-1)/9>10^{p^2-1}
$$
which is false for any $p\ge 3$. So, indeed when $n$ is a power of a prime $p$, then the congruence $p^3\mid 10^{p-1}-1$ must hold. We record this as follows.

\begin{lemma}
If $n>6$ and $n=p^t$ is solution of he equation 
\begin{equation*}
\varphi(L_n)=d\left(\frac{10^m-1}{9}\right),\qquad d\in \{1,\ldots,9\},
\end{equation*}
with some $t\ge 1$ and $p$ prime, then $t\in \{1,2\}$ and $p^3\mid 10^{p-1}-1$.
\end{lemma}

Suppose now that  $n$ is divisible by two distinct primes $p$ and $q$. By Lemma \ref{lem1:Car}, $L_p,~L_q$ and $L_{pq}$ each have primitive prime factors. This shows that $n=pq$, for if $n>pq$, then 
$L_n$ would have (at least) one additional prime factor, which is a contradiction. Assume $p<q$ and
$$
L_n=L_p L_q \left(\frac{L_{pq}}{L_p L_q}\right).
$$
Unless $q=L_p$, the three factors above are coprime. Say $q\ne L_p$. Then the three factors above are $p_1^{\alpha_1},~p_2^{\alpha_2}$ and $p_3^{\alpha_3}$ in some order. By Lemmas \ref{lem1:BLMS} and up to relabeling the primes $p_1$ and $p_2$, we may assume that $\alpha_1=\alpha_2=1$, so $L_p=p_1$, $L_q=p_2$ and $L_{pq}/(L_pL_q)=p_3^{\alpha_3}$. On the other hand, if $q=L_p$, then $q^2\| L_{pq}$. This shows then that up to relabeling the primes we may assume that $\alpha_2=1,~\alpha_3=2$, $L_p=p_3,~L_q=p_2$, $L_{pq}/(L_pL_q)=p_3 p_1^{\alpha_1}$. However, in this case $p_3\equiv 3\pmod 8$, showing that $p\equiv 5\pmod 8$. In particular, we also have
$p\equiv 1\pmod 4$, so $p_3=L_p\equiv 1\pmod 5$, and this is not possible. So, this case cannot appear.

Write $m=2m_0$. Then
$$
(p_1-1)(p_2-1)(p_3-1)p_3^{\alpha_3-1}=\varphi(L_n)=\frac{8(10^{m_0}-1)(10^{m_0}+1)}{9}.
$$
If $m_0$ is even, then $p_3^{\alpha_3-1}\mid 10^{m_0}-1$ because $p_3\equiv 3\pmod 4$, so $p_3$ cannot divide $10^{m_0}+1=(10^{m_0/2})^2+1$. If $m_0$ is odd, then $p_3^{\alpha_3-1}\mid 10^{m_0}+1$, because if not we would have that $p_3\mid 10^{m_0}-1$, so $10^{m_0}\equiv 1\pmod {p_3}$, and since $m_0$ is odd we would get ${\displaystyle{\left(\frac{10}{p_3}\right)=1}}$, which is false since ${\displaystyle{\left(\frac{2}{p_3}\right)=-1}}$ and ${\displaystyle{\left(\frac{5}{p_3}\right)=1}}$. Thus, we get, using \eqref{eq2:lll}, that
\begin{eqnarray}
\label{eq2:15}
\alpha^{p+q} p_3 & > & (L_p-1) (L_q-1) p_3=p_1p_2p_3>(p_1-1)(p_2-1)(p_3-1)\nonumber\\
& \ge &  \frac{8(10^{m_0}-1)}{9}>\frac{8}{10}\times 10^{m_0}.
\end{eqnarray}
On the other hand, by inequality \eqref{eq2:ineq}, we have
$$
10^m>\frac{8(10^m-1)}{9}=\varphi(L_n)>\frac{L_n}{4},
$$
 so that
 \begin{equation}
 \label{eq2:16}
 10^{m_0}>\frac{{\sqrt{L_n}}}{2}>\frac{\alpha^{pq/2-0.5}}{2},
 \end{equation}
 where we used the inequality \eqref{eq2:llll}. From \eqref{eq2:15} and \eqref{eq2:16}, we get
 $$
 p_3>\frac{8}{20 {\sqrt{\alpha}}} \alpha^{pq/2-p-q}=\frac{8}{20 \alpha^{4.5}} \alpha^{(p-2)(q-2)}>\frac{\alpha^{(p-2)(q-2)}}{25}.
 $$
 Once checks that the inequality
\begin{equation}
\label{eq2:A}
\frac{\alpha^{(p-2)(q-2)/2}}{25}>\alpha^{q+1}
\end{equation}
is valid for all pairs of primes $5\le p<q$ with $pq>100$. Indeed, the above inequality \eqref{eq2:A} is implied by
\begin{equation}
\label{eq2:B}
(p-2)(q-2)/2-(q+1)-7> 0,\quad {\text{\rm or}}\quad (q-2)(p-4)>20.
\end{equation}
If $p\ge 7$, then $q>p\ge 11$ and the above inequality \eqref{eq2:B} is clear, whereas if $p=5$, then $q\ge 23$ and the inequality \eqref{eq2:B} is again clear. 
 
 We thus get that
 $$
 p_3>\frac{\alpha^{(p-2)(q-2)}}{25}>\alpha^{q+1}>L_q=p_2>L_p=p_1.
 $$
 We exploit the two relations
 \begin{eqnarray}
 \label{eq2:imp1}
 0 & < & 1-\frac{\varphi(L_n)}{L_n} =  1-\left(1-\frac{1}{p_1}\right)\left(1-\frac{1}{p_2}\right)\left(1-\frac{1}{p_3}\right)<\frac{3}{p_1}<\frac{5}{\alpha^p};\nonumber\\
&& 1-\frac{(L_p-1)\varphi(L_n)}{L_p L_n}  =1-\left(1-\frac{1}{p_2}\right)\left(1-\frac{1}{p_3}\right) < \frac{2}{p_2}<\frac{4}{\alpha^q}.
\end{eqnarray}
In the above, we used the inequality \eqref{eq1:7}. Since $n$ is odd, we have $L_n=\alpha^n-\alpha^{-n}$. Then 
$$
1+\frac{2}{\alpha^{2n}}>\frac{1}{1-\alpha^{-2n}}>1,
$$
so
$$
\frac{1}{\alpha^n}+\frac{2}{\alpha^{3n}}>\frac{1}{L_n}>\frac{1}{\alpha^n},
$$
or 
\begin{equation}
\label{eq2:17}
\frac{8\times 10^m}{9\alpha^n} +\frac{16 \times 10^m}{9\alpha^{3n}}-\frac{8}{9L_n}>\frac{8(10^m-1)}{9L_n}=\frac{\varphi(L_n)}{L_n}>\frac{8\times 10^m}{9\alpha^n}-\frac{8}{9L_n}.
\end{equation}
The first inequality  \eqref{eq2:imp1} and \eqref{eq2:17} show that
\begin{equation}
\label{eq2:18}
\left|1-(8/9)\times 10^{m}\times \alpha^{-n}\right|<\frac{3}{p_1}+\frac{8}{9L_n}+\frac{16\times 10^m}{9\alpha^{3n}}.
\end{equation}
Now
$$
8\times 10^{m-1}<\frac{8(10^m-1)}{9}=\varphi(L_n)<L_n<\alpha^{n+1},\quad {\text{\rm so}}\quad 10^m<\frac{10\alpha}{8} \alpha^n,
$$
showing that 
$$
\frac{16 \times 10^m}{9\alpha^{3n}}<\frac{20 \alpha}{9\alpha^{2n}}<\frac{0.5}{\alpha^n}\quad {\text{\rm for}}\quad n>1000.
$$
Since also 
$$
\frac{8}{9L_n}<\frac{8\alpha}{9\alpha^n}<\frac{1.5}{\alpha^n},
$$
we see
$$
\frac{16\times 10^m}{9\alpha^{3n}}+\frac{8}{9L_n}<\frac{0.5}{\alpha^n}+\frac{1.5}{\alpha^n}<\frac{2}{\alpha^n}.
$$
Since also $p_1<L_n^{1/3}<\alpha^{(n+1)/3}$, we get that \eqref{eq2:18} becomes
\begin{equation}
\label{eq2:20}
\left|1-(8/9)\times 10^{m}\times \alpha^{-n}\right|<\frac{3}{p_1}+\frac{2}{\alpha^n}<\frac{4}{p_1}=\frac{4}{L_p}<\frac{4\alpha}{\alpha^p}<\frac{7}{\alpha^p},
\end{equation}
where the middle inequality is implied by $\alpha^n>2\alpha^{(n+1)/3}>13p_1$, which holds for $n>1000$.

The same argument based on \eqref{eq2:17} shows that 
\begin{equation}
\label{eq2:21}
\left|1-\left(\frac{8(L_p-1)}{9L_p}\right)\times 10^{m}\times \alpha^{-n}\right|<\frac{4}{\alpha^{q}}+\frac{2}{\alpha^n}<\frac{5}{\alpha^{q}}.
\end{equation}

We are in a situation to apply Theorem \ref{thm3:Matveev} to the left--hand sides of \eqref{eq2:20} and \eqref{eq2:21}. These expressions  are nonzero, since any one of these expressions being zero means $\alpha^n\in {\mathbb Q}$ 
for some positive integer $n$, which is false. We always take ${\mathbb K}={\mathbb Q}({\sqrt{5}})$ for which $D=2$. We take $t=3$, $\alpha_1=\alpha,~\alpha_2=10$, so we can take $A_1=\log\alpha=2h(\alpha_1)$ and $A_2=2\log 10$. 
For \eqref{eq2:20}, we take $\alpha_3=8/9$, and $A_3=2\log 9=2h(\alpha_3)$. For \eqref{eq2:21}, we take $\alpha_3=8(L_p-1)/9L_p$, so we can take $A_3=2p>h(\alpha_3)$. This last inequality holds because 
$h(\alpha_3)\le \log(9L_p)<(p+1)\log \alpha+\log 9<p$ for all $p\ge 7$, while for $p=5$ we have $h(\alpha_3)=\log 99<5$. We take $\alpha_1=-n,~\alpha_2=m,~\alpha_3=1$. Since
$$
2^n>L_n>\varphi(L_n)>8\times 10^{m-1}
$$
it follows that $n>m$. So, $B=n$. Now Theorem \ref{thm3:Matveev} implies that
$$
\exp\left(-1.4\times 30^6\times 3^{4.5}\times  2^2\times (1+\log 2)(1+\log n)   (\log \alpha)  (2\log 10) (2\log 9)\right),
$$
is a lower bound of the left--hand side of \eqref{eq2:20}, so inequality \eqref{eq2:20} implies
$$
p\log\alpha-\log 7<9.5\times 10^{12} (1+\log n),
$$
which implies 
\begin{equation}
\label{eq2:p}
p<2\times 10^{13} (1+\log n).
\end{equation}
Now Theorem \ref{thm3:Matveev} implies that the right--hand side of inequality \eqref{eq2:21} is at least as large as
$$
\exp\left(-1.4\times 30^6\times 3^{4.5} \times 2^2 \times (1+\log 2)(1+\log n) (\log \alpha) (2\log 10) (2p)\right)
$$
leading to
$$
q\log \alpha-\log 4<4.3 \times 10^{12} (1+\log n) p.
$$
Using \eqref{eq2:p}, we get
$$
q<9\times 10^{12} (1+\log n) p<2\times 10^{26} (1+\log n)^2.
$$
Using again \eqref{eq2:p}, we get
$$
n=pq<4\times 10^{39} (1+\log n)^2,
$$
leading to 
\begin{equation}
\label{eq2:boundforn}
n<5\times 10^{43}. 
\end{equation}
Now we need to reduce the bound. We return to \eqref{eq2:20}. Put
$$
\Lambda=m\log 10-n\log \alpha+\log(8/9).
$$
Then \eqref{eq2:20} implies that 
\begin{equation}
\label{eq2:25}
|e^{\Lambda}-1|<\frac{7}{\alpha^p}.
\end{equation}
Assuming $p\ge 7$, we get that the right--hand side of \eqref{eq2:25} is $<1/2$. Analyzing the cases $\Lambda>0$ and $\Lambda<0$ and by a use of the inequality $1+x<e^x$ which holds for all $x\in\mathbb{R}$, we get that
$$
|\Lambda|<\frac{14}{\alpha^p}.
$$
Assume say that $\Lambda>0$. Dividing across by $\log\alpha$, we get
$$
0< m \left(\frac{\log 10}{\log\alpha}\right)-n+\left(\frac{\log(8/9)}{\log\alpha}\right)<\frac{30}{\alpha^p}.
$$
We are now ready to apply Lemma \ref{reduce} with the obvious parameters
$$
\gamma=\frac{\log 10}{\log \alpha},\quad \mu=\frac{\log(8/9)}{\log\alpha},\quad A=30,\quad B=\alpha.
$$
Since $m<n$, we can take $M=10^{45}$ by \eqref{eq2:boundforn}. Applying Lemma \ref{reduce}, performing the calculations and treating also the case when $\Lambda<0$, we obtain $p<250$.   Now we go to inequality \eqref{eq2:21} and for $p\in [5,250]$, we consider 
$$
\Lambda_p=m\log 10-n\log\alpha+\log\left(\frac{8(L_p-1)}{9L_p}\right).
$$
Then inequality \eqref{eq2:21} becomes
\begin{equation}
\label{eq2:26}
\left|e^{\Lambda_p}-1\right|<\frac{5}{\alpha^q}.
\end{equation}
Since $q\ge 7$, the right--hand side is smaller than $1/2$. We thus obtain
$$
|\Lambda_p|<\frac{10}{\alpha^q}.
$$
We proceed in the same way as we proceeded with $\Lambda$ by applying Lemma \ref{reduce} to $\Lambda_p$ and distinguishing the cases in which $\Lambda_p>0$ and $\Lambda_p<0$, respectively. In all cases, we get that $q<250$. Thus, $5\le p<q<250$. Note however that we must have either $p^2\mid 10^{p-1}-1$ or $q^2\mid 10^{q-1}-1$. Indeed, the point is that since all three prime factors of $L_n$ are quadratic residues modulo $5$, and they are primitive prime factors of 
$L_p,~L_q$ and $L_{pq}$, respectively, it follows that 
$p_1\equiv 1\pmod {p}$, $p_2\equiv 1\pmod q$ and $p_3\equiv 1\pmod {pq}$. Thus, $(pq)^2\mid (p_1-1)(p_2-1)(p_3-1)\mid \varphi(L_n)=8(10^m-1)/9$, which in turn shows that $(pq)^2\mid 10^{m}-1$. 
Assume that neither $p^2\mid 10^{p-1}-1$ nor $q^2\mid 10^{q-1}-1$. Then relation $(pq)^2\mid 10^{m}-1$ implies that $pq\mid m$. Thus, $m\ge pq$, leading to
$$
2^{pq}>L_n>\varphi(L_n)=\frac{8(10^m-1)}{9}>10^{m-1}\ge 10^{pq-1},
$$
a contradiction. So, indeed either $p^2\mid 10^{p-1}-1$ or $q^2\mid 10^{q-1}-1$. However, a computation with Mathematica revealed that there is no prime $r$ such that $r^2\mid 10^{r-1}-1$ in the interval $[5,250]$.
In fact, the first such $r>3$ is $r=487$, but $L_{487}$ is not prime!

This contradiction shows that indeed when $n>6$, we cannot have $n=pq$. Hence, $n\in \{p,p^2\}$ and $p^3\mid 10^{p-1}-1$. We record this as follows.

\begin{lemma}
The equation 
\begin{equation*}
\varphi(L_n)=d\left(\frac{10^m-1}{9}\right),\qquad d\in \{1,\ldots,9\},
\end{equation*}
has no solution $n>6$ which is not of the form $n=p$ or $p^2$ for some prime $p$ such that $p^3\mid 10^{p-1}-1$.
\end{lemma}

\subsection{Bounding $n$}

Finally, we bound $n$. We assume again that $n>1000$. Equation \eqref{eq2:3} becomes 
$$
L_n =p_1^{\alpha_1}p_2^{\alpha_2}p_3^{\alpha_3}. 
$$ 
Throughout this last section, we assume that $p_1<p_2<p_3$. First, we bound $p_1$, $p_2$ and $p_3$ in terms of $n$.
Using the first relation of  \eqref{eq2:imp1}, we have that 
\begin{equation}
\label{ep:bn1}
 0  < 1-\frac{\varphi(L_n)}{L_n} =  1-\left(1-\frac{1}{p_1}\right)\left(1-\frac{1}{p_2}\right)\left(1-\frac{1}{p_3}\right)<\frac{3}{p_1}.
\end{equation}
By the argument used when estimating \eqref{eq2:17}--\eqref{eq2:20}, we get that
\begin{equation}
\label{eq:bounp1}
 |1 -(8/9)\times 10^m \times\alpha^{-n}|<\frac{3}{p_1} + \frac{2}{\alpha^n}<\frac{4}{p_1},
\end{equation}
where the last inequality holds because $p_1\le L_n/(p_2p_3)<L_n/(7\times 11)<\alpha^n/2$. 

We apply Theorem \ref{thm3:Matveev} to the left-hand side of \eqref{eq:bounp1}  The expression there is nonzero by a previous argument. We  take again ${\mathbb K}={\mathbb Q}({\sqrt{5}})$ for which $D=2$. We take $t=3$, $\alpha_1=8/9,~\alpha_2=10$ and $\alpha_3=\alpha$. Thus, we can take $A_1=\log 9=2h(\alpha_1)$ , $A_2=2\log 10$ and $A_3=2\log \alpha=2h(\alpha_3)$. 
We also take $b_1=1,~b_2=m,~b_3=-n$. We already saw that $B=n$. Now Theorem \ref{thm3:Matveev} implies as before that 
$$
\exp\left(-1.4\times 30^6\times 3^{4.5}\times  2^2\times (1+\log 2)(1+\log n)  2^3 (\log \alpha)  (\log 10) (\log 9)\right),
$$
is at least a lower bound for the left--hand side of \eqref{eq:bounp1}, hence inequality \eqref{eq2:20} implies
$$
\log p_1-\log 4<1.89\times 10^{13} (1+\log n),
$$ 
Then we get 
\begin{equation}
\label{eq2:boundp1}
\log p_1 <1.9\times 10^{13}(1+\log n).
\end{equation}
We use the same argument to bound $p_2$. We have
 $$ 
 0 < 1 - \left(\frac{p_1 -1}{p_1}\right)\frac{\varphi(L_n)}{L_n}=\left(1-\frac{1}{p_2}\right)\left(1-\frac{1}{p_3}\right)< \frac{2}{p_2}. 
 $$
Thus, we get that:
\begin{equation}
\label{eq:bounp2}
\left|1-\left(\frac{8(p_1 -1)}{9p_1}\right)\times 10^m\alpha^{-n}\right|< \frac{2}{p_2}+\frac{2}{\alpha^n}<\frac{3}{p_2},
\end{equation}
where the last inequality follows again because $p_2\le L_n/(p_1p_3)<\alpha^n/2$. 
 
We apply Theorem \ref{thm3:Matveev} to the left--hand side of \eqref{eq:bounp2}.  We take $t=3$, $\alpha_1=8(p_1-1)/(9p_1)$,~$\alpha_2=10$ and $\alpha_3=\alpha$, so we take $A_1=2\log(9p_1)\ge 2h(\alpha_1)$, $A_2=2\log 10$ and $A_3=2\log \alpha$. Again $b_1=-1,~b_2=m,~b_3=-n$ and $B=n$. Now Theorem \ref{thm3:Matveev} implies that 
$$
\exp\left(-1.4\times 30^6\times 3^{4.5}\times  2^2\times (1+\log 2)(1+\log n) 2^3  (\log \alpha)  \log 10 \log (9p_1)\right).
$$
is a lower bound on the left--hand side of \eqref{eq:bounp2}. Using estimate \eqref{eq2:boundp1}, inequality \eqref{eq:bounp2} implies
\begin{equation}
\label{eq2:boundp2}
\log p_2 -\log 2< 1.8\times 10^{26}(1+\log n )^2.
\end{equation}
Using a similar argument, we get
\begin{equation}
\label{eq2:boundp3}
\log p_3-\log 2 <1.8\times 10^{39}(1+\log n )^3.
\end{equation}
Now we can bound $n$. Equation \eqref{eq2:3}, gives:
$$
\alpha^n + \beta^n =p_1^{\alpha_1}p_2^{\alpha_2}p_3^{\alpha_3}. 
$$
Thus,
\begin{equation}
\label{eq2:boundn}
 |p_1^{\alpha_1}p_2^{\alpha_2}p_3^{\alpha_3}\alpha^{-n} -1| = \frac{1}{\alpha^{2n}} 
\end{equation}
We can apply Theorem \ref{thm3:Matveev}, with $t=4$, $\alpha_1=p_1,~\alpha_2=p_2$, $\alpha_3=p_3$, and $\alpha_4= \alpha$. We take $A_1=2\log p_1=2h(\alpha_1)$, $A_2=2\log p_2$, $A_3=2\log p_3=2h(\alpha_3)$ and 
$A_4 =2\log \alpha$. We take $B=n$. Then Theorem \ref{thm3:Matveev} implies that 
$$
\exp\left(-1.4\times 30^7\times 4^{4.5}\times  2^2\times (1+\log 2)(1+\log n) 2^4   (\log \alpha) \prod_{i=1}^3 (\log p_i) \right).
$$
is a lower bound on the left--hand side of \eqref{eq2:boundn}. Using \eqref{eq2:boundn} and inequalities \eqref{eq2:boundp1}, \eqref{eq2:boundp2}, \eqref{eq2:boundp3}, we get
$$
n<8\times 10^{93}(1+\log n)^7,\quad {\text{\rm so}}\quad n<10^{111}.
$$
This gives the upper bound. As for the lower bound, a quick check with Mathematica revealed that the only primes $p<2\times 10^9$ such that $p^2\mid 10^{p-1}-1$ are 
$p\in \{3,487,56598313\}$ and none of these has in fact the stronger property that $p^3\mid 10^{p-1}-1$.



\section{Pell and Pell-Lucas Numbers With Only One Distinct Digit}
\label{sec24}

Here, we show that there are no Pell or Pell-Lucas numbers larger than $10$ with only one distinct digit.
\medskip

In this section, we do not use linear forms in logarithms, but show in an elementary way that $5$ and $6$ are respectively the largest Pell and Pell-Lucas numbers which has only one distinct digit in their decimal expansion. The method of the proofs is similar to the method from \cite{FL}, paper in which  the authors determined in an elementary way the largest repdigits in the Fibonacci and the Lucas sequences. We mention that the problem of determining the repdigits in the Fibonacci and Lucas sequence was revisited in \cite{BL},  where the authors determined all the repdigits in all generalized Fibonacci sequences $\{F_n^{(k)}\}_{n\ge 0}$, where this sequence starts with $k-1$ consecutive $0$'s followed by a $1$ and follows the recurrence $F_{n+k}^{(k)}=F_{n+k-1}^{(k)}+\cdots+F_n^{(k)}$ for all $n\ge 0$. However, for this generalization, the method used in \cite{BL} involved linear forms in logarithms.

\medskip

Our results are the following.

\begin{theorem}
\label{thm3:2}
If 
\begin{equation}
\label{eq22:P}
P_n=a\left(\frac{10^m-1}{9}\right) \quad {\text{\rm \it for some}}\quad a\in \{1,2,\ldots, 9\},
\end{equation}
then $n=0,1,2,3.$
\end{theorem}

\begin{theorem}
\label{thm3:3}
If 
\begin{equation}
\label{eq22:Q}
Q_n=a\left(\frac{10^m-1}{9}\right) \quad {\text{\rm \it for some}}\quad a\in \{1,2,\ldots,9\},
\end{equation}
then $n=0,1,2.$
\end{theorem}
\medskip

\section*{\small Proof of Theorem \ref{thm3:2}}

We start by listing the periods of $\{P_n\}_{n\ge 0}$ modulo $16$, $5$, $3$ and $7$ since they are useful later
\begin{eqnarray}
\label{eq22:period}
&& 0,~1,~2,~5,~12,~13,~6,~9,~8,~9,~10,~13,~4,~5,~14,~1,~0,~1 \pmod {16}\nonumber\\
&& 0,~1,~2,~0,~2,~4,~0,~4,~3,~0,~3,~1,~0,~\pmod 5\nonumber\\
&& 0,~1,~2,~2,~0,~2,~1,~1,~0,~1\pmod 3\\
&& 0,~1,~2,~5,~5,~1,~0,~1\pmod 7.\nonumber
\end{eqnarray}
We also compute $P_n$ for  $n\in [1,20]$ and conclude that the only solutions in this interval correspond to $n=1,2,3$. From now, we suppose that $n\geq 21$. Hence,
$$
P_n \ge P_{21}=38613965 >10^7.
$$
Thus, $m\geq 7$. Now we distinguish several cases according to the value of $a$. 

\medskip

{\bf Case $a=5.$}

\medskip

Since $m\geq 7$, reducing equation \eqref{eq22:P} modulo $16$ we get
$$
P_n=5\left(\frac{10^m-1}{9}\right)\equiv 3 \pmod{16} .
$$
A quick look at the first line in \eqref{eq22:period} shows that there is no $n$ such that $P_n\equiv 3 \pmod {16}$. 

\medskip

From now on, $a\ne 5$. Before dealing with the remaining cases, let us prove that $m$ has to be odd. We assume by contradiction that this is not the case i.e., $m$ is even. Hence, $2\mid m$, therefore
$$
11\Big|\frac{10^2-1}{9}\Big| \frac{10^m-1}{9}\Big|P_n.
$$
Since, $11\mid P_n$, it follows that $12\mid n$. Hence, 
$$
2^2\cdot 3^2 \cdot 5\cdot 7 \cdot 11=13860=P_{12}\mid P_n = a\cdot\frac{10^m-1}{9},
$$
and the last divisibility is not possible since $a(10^m-1)/9$ cannot be a multiple of $10$.
Thus, $m$ is odd.

\medskip

We can now compute the others cases.

\medskip

{\bf Case $a=1$.}

\medskip

Reducing equation \eqref{eq22:P} modulo $16$, we get
$P_n\equiv 7 \pmod{16}$. A quick look at the first line of \eqref{eq22:period} shows that there is no $n$ such that $P_n\equiv 7\pmod {16}$. Thus, this case  is impossible.

\medskip

{\bf Case $a=2$.}

\medskip

Reducing equation \eqref{eq22:P} modulo $16$, we get
$$
P_n=2\left(\frac{10^m-1}{9}\right)\equiv 14\pmod{16}.
$$ 
A quick look at the first line of \eqref{eq22:period} gives $n\equiv 14\pmod {16}$. Reducing also equation \eqref{eq22:P} modulo $5$, we get $P_n\equiv 2\pmod 5$, and now line two of \eqref{eq22:period} gives $n\equiv 2,4\pmod {12}$. Since also $n\equiv 14\pmod {16}$, we get that $n\equiv 14\pmod {48}$. Thus, $n\equiv 6\pmod 8$, and now row 
three of \eqref{eq22:period} shows that $P_n\equiv 1\pmod 3$. Thus,
$$
2\left(\frac{10^{m}-1}{9}\right)\equiv 1\pmod 3.
$$
The left-hand side above is $2(10^{m-1}+10^{m-2}+\cdots+10+1)\equiv 2m\pmod 3$, so we get $2m\equiv 1\pmod 3$,  so $2\equiv m\pmod 3$, and since $m$ is odd we get $5\equiv m\pmod 6$.
Using also the occurrence $n\equiv 2\pmod 6$, we get from the last row of \eqref{eq22:period} that $P_n\equiv 2\pmod 7$. Thus,
$$
2\left(\frac{10^m-1}{9}\right)\equiv 2\pmod 7,
$$
leading to $10^m-1\equiv 9\pmod 7$, so $1\equiv 10^{m-1}\pmod 7$. This gives $6\mid m-1$, or $m\equiv 1\pmod 6$, contradicting the previous conclusion that $m\equiv 5\pmod 6$. 

\medskip

{\bf Case $a=3$.}

\medskip

In this case, we have that $3\mid P_n$, therefore $4\mid n$ by the third line of \eqref{eq22:period}. Further,
$$
P_n=3\left(\frac{10^m-1}{9}\right)\equiv 5 \pmod{16}.
$$
The first line of \eqref{eq22:period} shows that  $n\equiv 3,13\pmod {16}$, contradicting the fact that $4\mid n$. Therefore, this case cannot occur.

\medskip

{\bf Case $a=4.$}

\medskip

In this case $4\mid P_n$, which implies that $4\mid n$. Reducing equation \eqref{eq22:P} modulo $5$ we get that $P_n\equiv 4\pmod 5$. Row two of \eqref{eq22:period} shows that 
$n\equiv 7,5\pmod {12}$. This is a contradiction with fact that $4\mid n$. Therefore, this case is  not possible.

\medskip

{\bf Case $a=6$.}

\medskip

Here, $3\mid P_n$, therefore $4\mid n.$ Hence,
$$
12\mid P_n=6\left(\frac{10^m-1}{9}\right),
$$
which is not possible.

\medskip

{\bf Case $a=7$.}

\medskip

Here, we have that $7\mid P_n$, therefore $6\mid n$ by row four of \eqref{eq22:period}. Hence,
$$
70= P_6\mid P_n= 7\left(\frac{10^m-1}{9}\right),
$$
which is impossible.

\medskip

{\bf Case $a=8.$}

\medskip

We have that $8\mid P_n$, so $8\mid n$. Hence,
$$
8\cdot 3 \cdot 17=408=P_8\mid P_n = 8\left(\frac{10^m-1}{9}\right),
$$
implying $17\mid 10^m-1$. This last divisibility condition implies that $16\mid m$, contradicting the fact that $m$ is odd.

\medskip

{\bf Case $a=9$.}

\medskip
 
We have $9\mid P_n$, thus $12\mid n$. Hence,
$$13860=P_{12}\mid P_n=10^m-1,$$
a contradiction.

\medskip

This completes  the proof of Theorem \ref{thm3:2}.

\section*{\small The proof of Theorem \ref{thm3:3}}

We list the periods of $\{Q_n\}_{n\ge 0}$ modulo $8,~5$ and $3$ getting
\begin{eqnarray}
\label{eq22:period1}
&& 2,~2,~6,~6,~2,~2\pmod 8\nonumber\\
&& 2,~2,~1,~4,~4,~2,~3,~3,~4,~1,~1,~3,~2,~2\pmod 5\\
&& 2,~2,~0,~2,~1,~1,~0,~1,~2,~2\pmod 3\nonumber\\
\end{eqnarray}
We next compute the first values of $Q_n$ for $n\in [1,20]$  and we see that there is no solution $n>3$ in this range. Hence, from now on, 
$$
Q_n>Q_{21}=109216786>10^8,
$$
so $m\geq 9$. Further, since $Q_n$ is always even and the quotient $(10^m-1)/9$ is always odd, it follows that $a\in\{2,4,6,8\}.$ Further, from row one of \eqref{eq22:period1} we see that $Q_n$ is never divisible by $4$. Thus, $a\in \{2,6\}$.

\medskip

{\bf Case $a=2$.}

\medskip
Reducing equation \eqref{eq22:Q} modulo $8$, we get that 
$$
Q_n= 2\left(\frac{10^m-1}{9}\right) \equiv 6 \pmod{8}.
$$
Row one of \eqref{eq22:period1} shows that $n\equiv 2,3\pmod 4$. Reducing equation \eqref{eq22:Q} modulo $5$ we get that $Q_n\equiv 2\pmod 5$, and now row two of \eqref{eq22:period1} gives that $n\equiv 0,1,5\pmod {12}$, so in particular $n\equiv 0,1\pmod 4$. Thus, we get a contradiction. 

\medskip

{\bf Case $a=6.$}

\medskip
First $3\mid n$, so by row three of \eqref{eq22:period1}, we have that $n\equiv 2,6\pmod 8$. Next reducing \eqref{eq22:Q} modulo $8$ we get 
$$
Q_n=6\left(\frac{10^m-1}{9}\right)\equiv 2\pmod{8}.
$$
and by the first row of \eqref{eq22:period1} we get $n\equiv 0,1\pmod 4$. Thus, this case cannot appear.
\medskip

This finishes the proof of Theorem \ref{thm3:3}.

\chapter{On Lehmer's Conjecture}
In this chapter, we are interested in finding members of the Lucas sequence $\{L_n\}_{n\ge 0}$ and of the Pell sequence $\{P_n\}_{n\ge 0}$ which are \textit{Lehmer Numbers}. Namely, we study respectively in Sections \textcolor{red}{\ref{sec31}} and \textcolor{red}{\ref{sec32}} values of $n$ for which, the divisibility relations $\varphi(L_n)\mid L_n-1$ and $\varphi(P_n)\mid P_n-1$ hold and further such that $L_n$ and/or $P_n$ is composite. The contents of this chapter are respectively the papers \cite{FL} and \cite{FL02}.


\section{Lucas Numbers with the Lehmer property}
\label{sec31}
We are interested in this section on members of the Lucas sequence $\{L_n\}_{n\ge 0}$ which are Lehmer numbers. Here, we will use some relations among Fibonacci and Lucas numbers, that can be easily proved using the Binet formulas \eqref{binet1} and \eqref{binet2}.
%
%
Our result is the following:

\begin{theorem}
\label{thm4:1}
There is no Lehmer number in the Lucas sequence.
\end{theorem}

\begin{proof}
Assume that $L_n$ is Lehmer for some $n$.
Clearly, $L_n$ is odd and $\omega(L_n)\ge 15$ by the main result from \cite{joh}. The product of the first $15$ odd primes exceeds $1.6\times 10^{19}$, so $n\ge 92$. Furthermore,

\begin{equation}
\label{eq3:15}
2^{15}\mid 2^{\omega(L_n)}\mid \varphi(L_n)\mid L_n-1.
\end{equation}
\medskip

If $n$ is even, Lemma \ref{eq1:fibluc} $(iii)$ shows that $L_n-1=L_{n/2}^2+1$ or $L_{n/2}^2-3$ and numbers of the form $m^2+1$ or $m^2-3$ for some integer $m$ are never multiples of $4$, so divisibility \eqref{eq3:15} is impossible. If $n\equiv 3\pmod 8$, Lemma \ref{eq1:fibluc} (viii) and relation \eqref{eq3:15} show that $2^{15}\mid L_{(n+1)/2}L_{(n-1)/2}$. This is also impossible since no member of the Lucas sequence is a multiple of $8$, fact which can be easily proved by listing its first $14$ members modulo $8$:
$$
2,~1,~3,~4,~7,~3,~2,~5,~7,~4,~3,~7,~2,~1,
$$
and noting that we have already covered the full period of $\{L_m\}_{m\ge 0}$ modulo $8$ (of length $12$) without having reached any zero.

\medskip

So, we are left with the case when $n\equiv 1\pmod{4}.$

Let us write
$$
n=p_1^{\alpha_1}\cdots p_k^{\alpha_k},
$$
with $p_1<\cdots<p_k$ odd primes and $\alpha_1,\ldots,\alpha_k$ positive integers. If $p_1=3$, then $L_n$ is even, which is not the case. Thus, $p_1\geq 5$.

\medskip

Here, we use the argument from \cite{FL1}  to bound $p_1$. Since most of the details are similar, we only sketch the argument. For $p\mid L_n$, using relation (iv) of Lemma \ref{eq1:fibluc}, we get that $-5F_n^2\equiv -4\pmod p$. In particular, $\(\frac{5}{p}\)=1$, so by Quadratic Reciprocity also $p$ is a quadratic residue modulo $5$. Now let $d$ be any divisor of $n$ which is a multiple of $p_1$. By Lemma \ref{lem1:Car}, there exists a primitive prime $p_d\mid L_d$, such that $p_d\nmid L_{d_1}$ for all positive $d_1<d.$ Since $n$ is odd and $d\mid n$, we have $L_d\mid L_n$, therefore $p_d\mid L_n$. Since $p_d$ is primitive for $L_d$ and a quadratic residue modulo $5$, we have $p_d\equiv 1\pmod d$ (if $p$ were not a quadratic residue modulo $5$, then we would have had that $p_d\equiv -1\pmod 5$, which would be less useful for our problem). In particular,
\begin{equation}
\label{eq3:p1}
p_1\mid d\mid p_d-1\mid \varphi(L_n).
\end{equation}
Collecting the above divisibilities \eqref{eq3:p1} over all divisors $d$ of $n$ which are multiples of $p_1$ and using Lemma \ref{eq1:fibluc} $(viii)$, we have
\begin{equation}
\label{eq3:5}
p_1^{\tau(n/p_1)} \mid \varphi(L_n)\mid L_n-1\mid 5 F_{(n-1)/2}F_{(n+1)/2}.
\end{equation}
If $p_1=5$, then $5\mid n$, therefore $5\nmid F_{(n\pm 1)/2}$ because a Fibonacci number $F_m$ is a multiple of $5$ if and only if its index $m$ is a multiple of $5$. Thus, $\tau(n/p_1)=1$, so $n=p_1$, which is impossible since $n>92$.

Assume now that $p_1>5$. Since
$$
\gcd(F_{(n+1)/2}, F_{(n-1)/2})=F_{\gcd((n+1)/2,(n-1)/2)}=F_1=1,
$$
divisibility relation \eqref{eq3:5} shows that $p_1^{\tau(n/p_1)}$
divides $F_{(n+\varepsilon)/2}$ for some $\varepsilon\in \{\pm 1\}$. Let $z(p_1)$ be the order of appearance of $p_1$ in the Fibonacci sequence. Write
\begin{equation}
\label{eq3:Wall}
F_{z(p_1)}=p_1^{e_{p_1}} m_{p_1},
\end{equation}
where $m_{p_1}$ is coprime to $p_1$. If $p_1^t\mid F_k$ for some $t>e_{p_1}$, then necessarily $p_1\mid k$. Since for us $(n+\varepsilon)/2$ is not a multiple of $p_1$ (because $n$ is a multiple of $p_1$), we get that $\tau(n/p_1)\le e_{p_1}$. In particular, if $p_1=7$, then $e_{p_1}=1$, 
so $n=p_1$, which is false since $n>92$. So, $p_1\ge 11$. We now follow along the argument from \cite{FL1} to get that
\begin{equation}
\label{eq3:30}
\tau(n)\leq 2\tau(n/p_1)\leq \frac{(p_1+1)\log\alpha}{\log p_1}.
\end{equation}
Further, since $(L_n-1)/\varphi(L_n)$ is an integer larger than $1$, we have
\begin{equation}
\label{r:1}
2<\frac{L_n}{\varphi(L_n)}\leq \prod_{p\mid L_n}\(1+\frac{1}{p-1}\)<\exp\(\sum_{p\mid L_n}\frac{1}{p-1}\),
\end{equation}
or
\begin{equation}
\label{eq3:4}
\log 2 \leq \sum_{p \mid L_n}\frac{1}{p-1}.
\end{equation}
For a divisor $d$ of $n$, we note ${\mathcal P_d}$ the set of primitive prime factors of $L_d$. Then the argument from \cite{FL1} gives
\begin{equation}
\label{eq3:imp}
\sum_{p \in \mathcal{P}_d}\frac{1}{p-1} \leq \frac{0.9}{d} + \frac{2.2\log\log d}{d}.
\end{equation}
Since the function $x\mapsto (\log\log x)/x$ is decreasing for $x>10$ and all divisors $d>1$ of $n$ satisfy $d>10$, we have, using \eqref{eq3:30}, that
\begin{eqnarray}
\label{eq3:8}
\sum_{p\mid L_n}\frac{1}{p-1}&=& \sum_{d\mid n}\sum_{p\in \mathcal{P}_d}\frac{1}{p-1} \leq \sum_{\substack{d\mid n\\ d>1}}\(\frac{0.9}{d} + \frac{2.2\log\log d}{d}\)\\
&\leq & \(\frac{0.9}{p_1} + \frac{2.2\log\log p_1}{p_1}\)\tau(n) \nonumber \\
&\leq & (\log\alpha)\frac{(p_1+1)}{\log p_1}\cdot \(\frac{0.9}{p_1} + \frac{2.2\log\log p_1}{p_1}\),\nonumber
\end{eqnarray}
which together with inequality \eqref{eq3:4} leads to
\begin{equation}
\label{eq3:9}
\log p_1\leq \frac{(\log\alpha)}{\log 2}\left(\frac{p_1+1}{p_1}\right)(0.9 + 2.2\log\log p_1).
\end{equation}
The above inequality \eqref{eq3:9} implies $p_1<1800$.
Since $p_1<10^{14}$, a calculation of McIntosh and Roettger \cite{MC} shows that $e_{p_1}=1$. Thus, $\tau(n/p_1)=1$, therefore $n=p_1.$ Since $n\ge 92$, we have $p_1\ge 97$. Going back to the inequalities \eqref{eq3:4} and \eqref{eq3:imp}, we get
$$
\log 2<\frac{0.9}{p_1}+\frac{2.2 \log\log p_1}{p_1},
$$
which is false for $p_1\ge 97$. Thus, Theorem \eqref{thm4:1} is proved.
\end{proof}

\section{Pell Numbers with the Lehmer property}
\label{sec32}
In this section, we study the members of Pell sequence $\{P_n\}_{n\ge 0}$  which are Lehmer numbers. From relation \eqref{eq1:sizePn}, we have the following inequality:
\begin{equation}
\label{eq3:sizePn}
P_n\ge 2^{n/2}
\end{equation}
which hold for all $n\ge 2$. 


Here, we prove the following result.
\begin{theorem}
\label{thm4:2}
There is no Pell and Pell-Lucas numbers with the Lehmer property.
\end{theorem}

\begin{proof}

Let us recall that if $N$ has the Lehmer property, then $N$ has to be odd and square-free.
In particular, if $P_n$ has the Lehmer property for some positive integer $n$, then Lemma \ref{lem1:orderof2} (ii) shows that $n$ is odd.  One checks with the computer that there is 
no number $P_n$ with the Lehmer property with $n\leq 200$. So, we can assume that $n>200$. Put $K=\omega(P_n)\geq 15$. 

From relation \eqref{eq3:RelP}, we have that
$$P_n-1=P_{(n-\epsilon)/2} Q_{(n+\epsilon)/2} \quad \hbox{where}\quad \epsilon\in\{\pm 1\}.$$
By Theorem 4 in \cite{CP}, we have that $P_n<K^{2^K}.$ By \eqref{eq3:sizePn}, we have that $K^{2^K}>P_n>2^{n/2}.$ Thus,
\begin{equation}
\label{eq3:2}
2^K\log K>\frac{n\log 2}{2}>\frac{n}{3}.
\end{equation}
We now check that the above inequality implies that  
\begin{equation}
\label{eq3:3}
2^K>\frac{n}{4\log\log n}.
\end{equation}
Indeed, \eqref{eq3:3} follows immediately from \eqref{eq3:2} when $K<(4/3)\log\log n$. On the other hand, when $K\ge (4/3)\log\log n$, we have $K\ge (\log n)^{4/3}$, so 
$$
2^K\ge 2^{(\log n)^{4/3}}>n,
$$
which holds because $n>200$.  Then, the relation \eqref{eq3:3} holds.

Let $q$ be any prime factor of $P_n$. Reducing relation
\begin{equation}
\label{eq3:nodd}
Q_n^2 - 8P_n^2 = 4 (-1)^n
\end{equation}
of Lemma \ref{lem1:PQ} modulo $q$, we get $Q_n^2\equiv -4\pmod q$. Since $q$ is odd, (because $n$ is odd), we get that $q\equiv 1\pmod 4$. This is true for all prime factors
$q$ of $P_n$. Hence,
$$
2^{2K}\mid \varphi(P_n)\mid P_n-1=P_{(n-\epsilon)/2} Q_{(n+\epsilon)/2}.
$$
Since $Q_n$ is never divisible by $4$, we have that $2^{2K-1}\mid $ divides $P_{(n+1)/2}$ or $P_{(n-1)/2}$. Hence, $2^{2K-1}$  divides $(n+1)/2$ or $(n-1)/2$. Using relation \eqref{eq3:3}, we have that 
$$
\frac{n+1}{2}\geq 2^{2K-1} \geq \frac{1}{2}\(\frac{n}{4\log\log n}\)^2.
$$
This last inequality leads to 
$$
n^2<16(n+1)(\log\log n)^2,
$$
giving that $n<21$, a contradiction, which completes the proof of Theorem \ref{thm4:2}.  
\end{proof}




\begin{thebibliography}{99}





\bibitem[Ba67]{baker} A. Baker, ``Linear forms in logarithms of algebraic numbers. I, II, III", {\it Mathematika\/} 13 (1966); 204-216, ibid. {\bf 14} (1967), 102-107; ibid. 14
(1967), 220-228.

\bibitem[Ba69]{AB} A. Baker, ``Bound For the solutions of Hyperelliptic equation``. {\it Mathematical Proceedings of the Cambridge Philosophical Society\/}. {\bf 65}. Issue 02. (1969), 439-444.

\bibitem[Bu08]{Yan} Y. Bugeaud, "Linear forms in the logarithms of algebraic numbers close to 1 and applications to Diophantine equations". {\it In: Diophantine equations\/}, 59-76, {\it Tata Inst. Fund. Res. Stud. Math.\/}, {\bf 20}, Tata Inst. Fund. Res., Mumbai, (2008), 21.

\bibitem[BD69]{BD} A. Baker and H. Davenport, ``The equations $3x^2-2 =y^2$ and $8x^2-7=z^2$", {\it Quart. J. Math. Oxford Ser. (2)\/} {\bf 20} (1969), 129-137.



\bibitem[BHV01]{bilu} Yu. Bilu, G. Hanrot and P. M. Voutier, ``Existence of primitive divisors
of Lucas and Lehmer numbers. With an appendix by M. Mignotte``, {\it J.
Reine Angew. Math.\/} {\bf 539} (2001), 75-122.



\bibitem[BL13]{BL1} J.~J. ~Bravo and F.~ Luca, ``On a conjecture about repdigits in $k$--generalized Fibonacci sequences", {\it Publ. Math. Debrecen\/}, {\bf 82} (2013), 623--639.

\bibitem[BLSTW02]{Cun} J.  Brillhart, D. H. Lehmer, J. L. Selfridge, B. Tuckerman and S. S. Wagstaff Jr., "Factorizations
of $b^n \pm 1$, $b = 2, 3, 5, 6, 7, 10, 11, 12$ up to high powers", {\it American Math. Soc.\/}, 2002.

\bibitem[BLMS08]{BLMS} Y. Bugeaud, F. Luca, M. Mignotte and S. Siksek, ``Almost powers in the Lucas sequence", {\it J. N. T. Bordeaux\/} {\bf 20} (2008), 555-600.

\bibitem[BMS06]{BMS} Y. Bugeaud, M. Mignotte and S. Siksek, ``Classical and modular approaches to exponential Diophantine equations. I. Fibonacci and Lucas perfect powers", {\it Ann. of Math.\/} {\bf 163}(2006), 969-1018.

\bibitem[Ca13]{Car} R. D. Carmichael, ``On the numerical factors of the arithmetic forms $\alpha^n\pm \beta^n$", {\it Ann. Math. (2)\/} {\bf 15} (1913), 30-70.

\bibitem[CH80]{coh} G. L. Cohen and P. Hagis, ``On the number of prime factors of $n$ if
$\phi(n)\mid n-1$``, {\it Nieuw Arch. Wisk\/}. {\bf 28} (1980), 177-185.

\bibitem[CL11]{F} J. Cilleruelo and F. Luca, ``Repunit Lehmer numbers", {\it Proceedings of the Edinburgh Mathematical Society\/} (2011), 55-65.


\bibitem[DFLT14]{Lu1}  M. T. Damir, B. Faye, F. Luca and A. Tall, ``Members of Lucas sequences whose Euler function is a power of $2$", {\it Fibonacci Quarterly\/}, {\bf 52} (2014), no. 1, 3-9.

\bibitem[DP98]{DP} A. Dujella and A. Peth\H o, ``A generalization of a theorem of Baker and Davenport", {\it Quart. J. Math. Oxford Ser. (2)\/} {\bf 49} (1998),  291-306. 

\bibitem[FL15a]{BL} B. Faye and F. Luca,  ``On the equation $\phi(X^m-1)=X^n-1$", {\it Int Journal of Number Theory.\/}  {\bf 11}, {\bf 5} (2015) 1691-1700.


\bibitem[FL15b]{BL2} B. Faye and F. Luca, ``Pell and Pell Lucas Numbers with only one Distinct Digit", {\it Annales Mathematicae and Informaticae\/} {\bf 45}, (2015), 55-60.

\bibitem[FL15c]{FL} B. Faye and F. Luca, "Lucas Numbers with Lehmer Property" {\it Mathematical Reports,\/} {\bf 19(69)}, 1(2017), 121-125.

\bibitem[FL15d]{FL1} B. Faye and F. Luca, ``Pell Numbers whose Euler function is a Pell Number", {\it Publications de l'Institut Math\'ematique(Beograd),\/} {\bf 101} (2015) 231-245. 

\bibitem[FL15e]{FL02} B. Faye and F. Luca, ``Pell Numbers with the Lehmer property", {\it Afrika Matematika,\/} {\bf 28} (2017), 291-294.

\bibitem[FLT15]{FLT} B. Faye, F. Luca and A. Tall, ``On the equation $\phi(5^m-1)=5^n-1$", {\it Bull. Korean Math. Soc.\/} {\bf 52} (2015), 513-524.


\bibitem[HR74]{HR} H. Halberstam and H. E. Richert, {\it Sieve Methods} (Academic Press, 1974).

\bibitem[Ja03]{Jomeson} G. O. Jameson, ``The Prime Number Theorem ", {\it London Mathematical Society. Student Texts \/}{\bf 53,} {\it Published in US of America by Cambrige Univ. Press, New York \/}  (2003),  291-306.

\bibitem[JBLT15]{JBLT} J.~J. ~Bravo, B. Faye, F.~ Luca and A. Tall, ``Repdigits as Euler functions of Lucas numbers", An. St. Math. Univ. Ovidius Constanta {\bf 24}(2) (2016) 105-126.

\bibitem[KO13]{Dajune} D. J. Kim and B. K. Oh,  ``Generalized Cullen numbers with the Lehmer Property``. {\it Bull. Korean Math. Soc \/}. {\bf 50} (2013), No. 6, 1981-1988.


\bibitem[KP75]{KP} P. Kanagasabapathy and T. Ponnudural, ``The simultaneous Diophantine equations $y^2-3x^2=-2$, $z^2-8x^2=-7$", {\it Quart. J of Math. Ser.(2)\/} {\bf 26} (1975), 275-278.

\bibitem[Le32]{Leh} D. H. Lehmer, ``On Euler totient function``, Bull. Amer. Math. Soc. {\bf 38} (1932) 745-751.

\bibitem[LLL82]{lll} A. K. Lenstra, H. W. Lenstra Jr. and L. Lovasz, "Factoring polynomials with rational coefficients", {\it Mathematische Annalen} {\bf 261} (1982), 515-534. 

\bibitem[LF09]{G} F. Luca and F. Nicolae, ``$\phi(F_m)=F_n$", {\it Integers} {\bf 9} (2009), A30.

\bibitem[Lj43]{Lj} W. Ljunggren, ``Noen Setninger om ubestemte likninger av formen $(x^n-1)/(x-1)=y^q$", {\it Norsk Mat. Tidsskr.\/} {\bf 25} (1943), 17-20.

\bibitem[LM06]{Lu4} F. Luca and M. Mignotte, ``$\phi(F_{11}) = 88$", {\it Divulgaciones Mat.\/} {\bf 14} (2006), 101-106.


\bibitem[LS14]{LuSt} F. Luca and P. St\u anic\u a, \emph{Equations with arithmetic functions of Pell numbers}, {\it Bull. Math. Soc. Sci. Math. Roumanie.} {\bf Tome 57(105)} (2014), 409-413.

\bibitem[Lu97]{A} F. Luca, ``Problem $10626$", {\it The Mathematical Monthly\/} {\bf 104} (1997), 81.


\bibitem[Lu99a]{D} F. Luca, ``On the equation $\phi(x^m + y^m) = x^n + y^n$", {\it Indian J. Pure Appl. Math.\/} {\bf 30} (1999), 183-197.

\bibitem[Lu99b]{Lu6} F. Luca, ``On the equation $\phi(|x^m+y^m|)=|x^n+y^n|$". {\it Indian J. Pure Appl. Math.\/} {\bf 30} (1999), 183-197.

\bibitem[Lu00a]{Lu2}  F. Luca, ``Equations involving arithmetic functions of Fibonacci numbers," {\it Fibonacci Quart.\/} {\bf 38} (2000), 49-55.


\bibitem[Lu00b]{FLl1} F. Luca, ``Fibonacci  and Lucas numbers with only one distinct digit", {\it Portugaliae Mathematica\/} {\bf 57} , (2000), 243-254.


\bibitem[Lu01]{Lu3} F. Luca, ``Multiply perfect numbers in Lucas sequences with odd parameters", {\it Publ. Math. Debrecen\/} {\bf 58} (2001), 121-155.

\bibitem[Lu02]{FL2} F. Luca,  ``Euler indicators of binary recurrence sequences",  {\it Collect. Math.\/} {\bf 53} (2002), 133-156. 


\bibitem[Lu07]{L} F. Luca, ``Fibonacci numbers with the Lehmer property", {\it Bull. Pol. Acad. Sci. Math.\/} {\bf 55} (2007), 7-15.

\bibitem[Lu08]{B} F. Luca, ``On the Euler function of repdigits",  {\it Czechoslovak Math. J.\/} {\bf 58} (2008), 51-59. 

\bibitem[Lu09a]{E} F. Luca, V. Janitzio Mej\'{\i}a Huguet and F. Nicolae, ``On the Euler Function of Fibonacci numbers",  {\it Journal of Integer sequences\/} {\bf 12} (2009), A09.6.6.

\bibitem[Lu09b]{dio} F. Luca, "Effective methods for Diphantine equations", {\it Bordeaux, January 23}, 2009.

 

\bibitem[Ma81]{FM} F. M\'aty\'as, ''On the common terms of second order linear recurrences,'' {\it Math. Sem. Notes Kobe Univ\/}., {\bf 9(1)}(1981), 89-97.

\bibitem[Ma00]{matveev} E. M. Matveev, ``An explicit lower bound for a homogeneous rational linear form in the logarithms of algebraic numbers", {\it Izv. Math.\/} {\bf 64} (2000), 1217-1269.


\bibitem[Mc02]{MD} W. McDaniel, ``On Fibonacci and Pell Numbers of the Form $k x^2$", {\it Fibonacci Quart.} {\bf 40} (2002), 41--42.

\bibitem[Mi78]{MM1} M. Mignotte, ``Intersection des images de certaines suites r\'eccurentes lin\'eaires", {\it Theor. Comp. Sci\/} {\bf 7} (1978), 117-121.

\bibitem[Mi79]{MM2} M. Mignotte, ``Une extention du th\'eoreme de Skolem-Mahler", {\it C. R. Acad. Sci. Paris\/} {\bf 288} (1979), 233-235.


\bibitem[MR98]{RM} W. L. McDaniel and P. Ribenboim, ``Square-classes in Lucas sequences having odd parameters", {\it J. Number Theory\/} {\bf 73} (1998), 14-27

\bibitem[MR07]{MC} R. J. McIntosh and E. L. Roettger, ``A search for Fibonacci Wiefrich and Wolsenholme
primes", {\it Math. Comp\/}. {\bf 76} (2007), 2087-2094.



\bibitem[MV73]{MV} H. L. Montgomery and R. C. Vaughan, ``The large sieve", {\it Mathematika\/} {\bf 20} (1973), 119-134.





\bibitem[Po77]{CP} C. Pomerance, ``On composite $n$ for which $\phi(n)\mid n-1$, II", {\it Pacific
J. Math.\/} {\bf  69} (1977), 177--186.

\bibitem[Re04]{joh} J. Renze,  ``Computational evidence for Lehmer's totient conjecture``. Published electronically at http://library.wolfram.com/ infocenter/MathSource/5483/, 2004.


\bibitem[RL12]{GL} J. M. Grau Ribas and F. Luca,``Cullen numbers with the Lehmer property", {\it Proc. Amer. Math. Soc\/}. {\bf 140} (2012), 129-134: Corrigendum {\bf 141} (2013), 2941--2943.

\bibitem[Ro83]{MNR}  R. Guy, ``Estimation de la fonction de Tchebychef $\theta$ sur le $k$-ieme nombre premier et grandes valeurs de la fonction $\omega(n)$ nombre de diviseurs premiers de $n$", {\it Acta Arith.\/} {\bf 42} (1983), 367-389.


\bibitem[RS62]{RS} J. B. Rosser and L. Schoenfeld, ``Approximate formulas for some functions of prime numbers``, {\it Illinois J. Math.\/} {\bf 6}
(1962), 64-94.


\bibitem[Ve80]{VM} M. Velupillai, ``The equations $z^2-3y^2=-2$ and $z^2-6x^2=-5$", {\it in "A Collection of Manuscripts Related to the Fibonacci Sequence" (V.E. Hoggatt and M. Bicknell-Johnson, Eds), The Fibonacci Association, Santa Clara,\/} (1980), 71-75.

\bibitem[YLT15]{LAA} A. D. Yovo, F. Luca and A. Togbe, ``On the $x$-coordinate of Pell Equation which are Repdigits``, {\it Publ. Math. Debrecen}, {\bf 88} (2016), 291-306.

\bibitem[Yu99]{KY} K. Yu, ``p-adic logarithm forms and groups varieties II``, {\it Act. Arith} {\bf 89} (1999), 337-378.


\bibitem[Zs92]{Zs} K. Zsigmondy, "Zur Theorie der Potenzreste",  {\it Monatsh. f. Math.}  {\bf 3} (1892), 265-284. doi:10.1007/BF01692444.

\end{thebibliography}
\end{document}